\documentclass[11pt]{amsart}

\usepackage[pdfauthor={Ziyang GAO}, 
        pdftitle={Enlarged mixed Shimura varieties}%
      dvips,colorlinks=true]{hyperref}
\hypersetup{
  bookmarksnumbered=true,
  linkcolor=black,
  citecolor=black,
  pagecolor=black, 
  urlcolor=black,  
}

\usepackage[utf8]{inputenc}

\usepackage{xtab}
\usepackage{amsmath, amssymb, cancel,verbatim}
\usepackage[all]{xy}
\usepackage[scr=euler,scrscaled=.96]{mathalfa}
 
\usepackage{enumerate}
\usepackage{mathtools,booktabs}
\usepackage{amscd}
\usepackage{bbm, stmaryrd}
\usepackage{yfonts}
\usepackage{color}

\usepackage{epstopdf} 
\usepackage{booktabs}

\newtheorem{teo}{Theorem}[section]
\newtheorem{thm}[teo]{Theorem}
\newtheorem{prop}[teo]{Proposition}
\newtheorem{lemma}[teo]{Lemma}
\newtheorem{cor}[teo]{Corollary}
\newtheorem{conj}[teo]{Conjecture}

\newtheorem{eg}[teo]{Example}
\newtheorem{defn}[teo]{Definition}

\newtheorem{rmk}[teo]{Remark}

\newtheorem{claim}{Claim}
\newtheorem{def-prop}[teo]{Definition-Proposition}
\newtheorem{notation}[teo]{Notation}

\numberwithin{equation}{section}

  \newcommand{\A}{\mathbb{A}}
  
  \newcommand{\C}{\mathbb{C}}

  \newcommand{\G}{\mathbb{G}}

  \newcommand{\N}{\mathbb{N}}
  
  \newcommand{\Q}{\mathbb{Q}}
  \newcommand{\R}{\mathbb{R}}  
  \renewcommand{\S}{\mathbb{S}}
  \newcommand{\V}{\mathbb{V}}
  \newcommand{\Z}{\mathbb{Z}}

  \newcommand{\bV}{\mathbf{V}}
  
  \renewcommand{\epsilon}{\varepsilon}
  
  \newcommand{\im}{\text{Im}}

  \renewcommand{\cong}{\simeq}
  \renewcommand{\bar}{\overline}
  \renewcommand{\tilde}{\widetilde}

  \providecommand{\frac}[1]{\mathrm{Frac}(#1)}

  \renewcommand{\hom}{\mathrm{Hom}}

  \newcommand{\stab}{\mathrm{Stab}}
  \newcommand{\MT}{\mathrm{MT}}
  \newcommand{\GL}{\mathrm{GL}}
  
  \newcommand{\Sh}{\mathrm{Sh}}

  \renewcommand{\ker}{\mathrm{Ker}}

  \newcommand{\codim}{\mathrm{codim}}

  \renewcommand{\lim}{\mathrm{lim}}
  \renewcommand{\deg}{\mathrm{deg}}

  \newcommand{\lie}{\mathrm{Lie}}
  \newcommand{\Gr}{\mathrm{Gr}}

  \newcommand{\der}{\mathrm{der}}
  \newcommand{\ad}{\mathrm{ad}}
  \newcommand{\sm}{\mathrm{sm}}

  \newcommand{\Zar}{\mathrm{Zar}}
  \newcommand{\GSp}{\mathrm{GSp}}
  \newcommand{\Sp}{\mathrm{Sp}}

  \newcommand{\res}{\mathrm{Res}}

  \newcommand{\unif}{\mathrm{unif}}

\newcommand{\cA}{\mathcal{A}}

\newcommand{\cD}{\mathcal{D}}
\newcommand{\cE}{\mathcal{E}}
\newcommand{\cF}{\mathcal{F}}

\newcommand{\cH}{\mathcal{H}}

\newcommand{\cM}{\mathcal{M}}
\newcommand{\cO}{\mathcal{O}}
\newcommand{\cR}{\mathcal{R}}
\newcommand{\cS}{\mathcal{S}}

\newcommand{\cV}{\mathcal{V}}
\newcommand{\cW}{\mathcal{W}}
\newcommand{\cX}{\mathcal{X}}
\newcommand{\cY}{\mathcal{Y}}
\newcommand{\cZ}{\mathcal{Z}}

\newcommand\supervisor[1]{\def\@supervisor{#1}}

\newcommand{\pullback}[1][dr]{\save*!/#1-1.2pc/#1:(-1,1)@^{|-}\restore}
\newcommand{\bigpullback}[1][dr]{\save*!/#1-1.7pc/#1:(-1.5,1.5)@^{|-}\restore}
\newcommand{\pushoutcorner}[1][dr]{\save*!/#1+1.7pc/#1:(1.5,-1.5)@^{|-}\restore}

\newcounter{elno}

\renewcommand{\cong}{\simeq}

\begin{document}
\title[EMSV, bi-algebraic system, Ax type transcendental results]{Enlarged mixed Shimura varieties, bi-algebraic system and some Ax type transcendental results}
\author{Ziyang Gao}
\address{CNRS and Institut Math\'{e}matiques de Jussieu-Paris Rive Gauche,
4 place de Jussieu,
75005 Paris,
France}
\email{ziyang.gao@imj-prg.fr}

\maketitle

\begin{abstract}
We develop a theory of enlarged mixed Shimura varieties, putting the universal vectorial bi-extension defined by Coleman into this framework to study some functional transcendental results of Ax type. We study their bi-algebraic systems, formulate the Ax-Schanuel conjecture and explain its relation with the logarithmic Ax theorem and the Ax-Lindemann theorem which we shall prove. 
All these bi-algebraic and transcendental results extend their counterparts for mixed Shimura varieties. In the end we briefly discuss about the Andr\'{e}-Oort and Zilber-Pink type problems for enlarged mixed Shimura varieties.
\end{abstract}

\tableofcontents

\section{Introduction}

\subsection{Motivations: From Manin-Mumford and Andr\'{e}-Oort to bi-algebraicity} We start with two famous conjectures in arithmetic geometry: the Manin-Mumford and the Andr\'{e}-Oort conjectures. In the table the base field is $\C$ and $Y$ is assumed to be irreducible.

\begin{tabular}{c|c|c}
& \textbf{Manin-Mumford} & \textbf{Andr\'{e}-Oort}\\ \hline
object & Abelian variety $A$ and a subvariety $Y$ & Shimura variety $S$ and a subvariety $Y$ \\
& $\Sigma$:=set of torsion points of $A$ & $\Sigma$:=set of special points of $S$ \\ \hline
hypothesis & $(Y\cap \Sigma)^{\Zar}=Y$ &  $(Y\cap \Sigma)^{\Zar}=Y$  \\ \hline
conclusion & $Y$ is a torsion coset & $Y$ is a Hodge subvariety\footnotemark \\
\end{tabular}
\footnotetext{In this case means an irreducible component of some Hecke translate of a Shimura subvariety}

When the Shimura variety $S$ is $\cA_g$, a point $s$ is special if and only if the abelian variety parametrized by $s$ is CM. By standard specialization argument it suffices to prove Manin-Mumford for abelian varieties over $\bar{\Q}$. These two conjectures are independent, although similar. The mixed Andr\'{e}-Oort conjecture, replacing Shimura varieties by mixed Shimura varieties, partly unifies the previous two conjectures because it implies Manin-Mumford for CM abelian varieties. Now in the spirit of \cite{UllmoStructures-spec} we make the following conjecture in terms of the bi-algebraic systems associated with abelian varieties and Shimura varieties, which unifies Manin-Mumford and Andr\'{e}-Oort completely.
\begin{conj}\label{ConjectureBiAlgebraicArithmeticAndGeometric}
Let $S$ be an algebraic variety over $\bar{\Q}$ such that the uniformization $\cX$ of $S_{\C}^{\mathrm{an}}$ has a structure of algebraic variety over $\bar{\Q}$, and suppose $Y$ is a subvariety of $S$ containing a Zariski dense subset of arithmetically bi-algebraic points. Then $Y$ is geometrically bi-algebraic.
\end{conj}

Let us explain this conjecture for $\cA_g$. For the uniformization $\unif\colon\cH_g^+\rightarrow\cA_g$
, we say that $s\in \cA_g(\bar{\Q})$ is \textbf{arithmetically bi-algebraic} if $\unif^{-1}(s)\subset\cH^+_g\cap M_{2g\times 2g}(\bar{\Q})$. A theorem of Cohen \cite{CohenHumbert-surface} and Shiga-Wolfart \cite{ShigaCriteria-for-co} asserts that $s$ is arithmetically bi-algebraic if and only if  $s$ is a special point. Hence the set $\Sigma$ in the Andr\'{e}-Oort conjecture for $\cA_g$ equals the set of arithmetically bi-algebraic points. On the other hand $\cH^+_g$ is an open semi-algebraic subset of $\C^{g(g+1)/2}$. We say that an irreducible subvariety $Y$ of $\cA_g$ is \textbf{geometrically bi-algebraic} if any complex analytic irreducible component of $\unif^{-1}(Y)$ is algebraizable in $\C^{g(g+1)/2}$. By Ullmo-Yafaev \cite{UllmoA-characterisat} the geometrically bi-algebraic subvarieties of $\cA_g$ containing special points are precisely the Shimura subvarieties of $\cA_g$. So Conjecture~\ref{ConjectureBiAlgebraicArithmeticAndGeometric} is equivalent to the Andr\'{e}-Oort conjecture when $S=\cA_g$. In this particular case Conjecture~\ref{ConjectureBiAlgebraicArithmeticAndGeometric} is proven by Tsimerman \cite{TsimermanA-proof-of-the-}, based on a definability result of Peterzil-Starchenko \cite{PeterzilDefinability-of} and a functional transcendental result of Pila-Tsimerman \cite{PilaAxLindemannAg}.

The situation for abelian varieties $A$ over $\bar{\Q}$ is more complicated. We wish to endow a $\bar{\Q}$-structure on $\lie A_{\C}$ so that torsion points of $A$ are precisely arithmetically bi-algebraic points. However using the Schneider-Lang and W\"{u}stholz' analytic subgroup theorems, Ullmo \cite[Proposition~2.6]{UllmoStructures-spec} proved: any torsion point of $A$, except the origin, becomes transcendental in $\lie A_{\C}$ if we endow $\lie A_{\C}$ with the canonical $\bar{\Q}$-structure coming from the given $\bar{\Q}$-structure on $A$. A solution to this problem is proposed by Bost (see \cite[$\mathsection$2.2.2]{UllmoStructures-spec}): instead of $A$ we study its universal vector extension $A^\natural$. We briefly recall some basic facts about $A^\natural$.

By a \textit{vector extension} of $A$, we mean an algebraic group $E$ such that there exist a vector group $W$ and an exact sequence $0\rightarrow W\rightarrow E\rightarrow A\rightarrow 0$. There exists a universal vector extension $A^\natural$ of $A$ such that any vector extension $E$ of $A$ is obtained as (see \cite[Chapter~1, Proposition~1.10]{MazurMessing}):
\[
\xymatrix{
0 \ar[r] & W_A \ar[r] \ar@{.>}[d] & A^\natural \ar[r] \ar[d] & A \ar[d]^{=} \ar[r] &0 \\
0 \ar[r] & W \ar[r]  & E \ar[r] \pushoutcorner  & A \ar[r]  &0}
\]
where the maps $W_A \rightarrow W$ and $A^\natural \rightarrow E$ are unique. 
In fact over $\C$, $A^\natural$ is constructed as follows: Let $\Gamma:=H_1(A(\C),\Z)\subset H_1(A(\C),\C)$ be the period lattice of $A$. For the Hodge decomposition $H_1(A(\C),\C)=H^{0,-1}(A_\C)\bigoplus H^{-1,0}(A_\C)$, the holomorphic part $H^{-1,0}(A_\C)$ equals the tangent space of $A$ at $0$, and the anti-holomorphic part $H^{0,-1}(A_\C)$ equals $\Omega^1_{A_\C^\vee}$. We have
\[
\xymatrix{
0 \ar[r] & H^{0,-1}(A_\C) \ar[r] \ar[d]^{=} & H_1(A(\C),\C) \ar[r] \ar[d]^{\unif^\natural} & H^{-1,0}(A_\C) \ar[r] \ar[d]^{\unif} &0 \\
0 \ar[r] & \Omega^1_{A_\C^\vee} \ar[r] & A^\natural(\C)\cong\Gamma\backslash H_1(A(\C),\C) \ar[r] & A(\C)\cong\Gamma\backslash H^{-1,0}(A_{\C}) \ar[r] &0
}
\]
and the bottom line is nowhere split. Take the $\bar{\Q}$-structure $H_1(A(\C),\Z)\otimes\bar{\Q}$ on $H_1(A(\C),\C)$. As an application of W\"{u}stholz' analytic subgroup theorem \cite[Theorem~1]{WustholzAlgebraic-group}, Ullmo \cite[Th\'{e}or\`{e}me~2.10]{UllmoStructures-spec} proved
\[
z\in H_1(A(\C),\Z)\otimes\bar{\Q}\text{ such that }\unif^\natural(z)\in A^\natural(\bar{\Q})\Leftrightarrow \unif^\natural(z)\text{ is a torsion point of }A^\natural.
\]
Thus we get an arithmetic bi-algebraic description for the torsion points of $A$ because the projection $A^\natural\rightarrow A$ induces a bijection between their torsion points. This suggests that in view of Conjecture~\ref{ConjectureBiAlgebraicArithmeticAndGeometric}, the abelian varieties are not the good objects to study. Instead, one should study their universal vector extensions. We will characterize geometrically bi-algebraic subvarieties of $A^\natural$ in the paper, from which we can see that Conjecture~\ref{ConjectureBiAlgebraicArithmeticAndGeometric} for $A^\natural$ implies the Manin-Mumford conjecture without much effort.

\subsection{Motivation: Vector extensions, from Grothendieck to Laumon} 
The study of $A^\natural$ goes back to Grothendieck ($\natural$-extensions) and he studied the Gau\ss-Manin connection on $\lie(\mathfrak{A}^\natural/B)$ for any abelian scheme $\mathfrak{A}/B$ \cite{GrothendieckNaturalExtension} (see also \cite[Chapter~I, $\mathsection$3]{MazurMessing}). Later on Buium \cite{Buiumdifferential-Al, BuiumGeometry-of-dif} studied arithmetic differential equations on $\mathfrak{A}^\natural$, making $\mathfrak{A}^\natural$ a differential algebraic group (algebraic $D$-group), although the idea was already used by Manin for his work on algebraic curves over function fields. These are extra tools which do not exist for $\mathfrak{A}$, and thus it is often convenient or even necessary to work with $\mathfrak{A}^\natural$ in order to prove properties for $\mathfrak{A}$. Examples of this fact include, but are not limited to: Manin's proof of the geometric Mordell conjecture \cite{ManinAlgebraic-curve, ManinRational-points}, Buium's effective bound for the geometric Lang conjecture \cite{BuiumEffective-bound}, Bertrand-Pillay's relative Lindemann-Weierstra\ss~theorem of  \cite{BertrandA-Lindemann-Wei}, Bertrand-Masser-Pillay-Zannier's proof of the relative Manin-Mumford conjecture for semi-abelian surfaces \cite{BertrandRelative-Manin-}, results on Galois groups and Manin maps of Bertrand-Pillay \cite{BertrandGalois-theory-f}, etc.

A $1$-motive over $\C$, after Deligne, consists of a complex semi-abelian variety $G$ and a group homomorphism $\Z^n\rightarrow G$. However Deligne's $1$-motives do not allow vector extensions, although they appear implicitly in the de Rham realizations. To fix this, Laumon \cite{LaumonTransformation-} extended the notion of $1$-motives by replacing $\Z^n$ by a formal group $F$. Then the identity component $F^\circ$ of $F$ gives rise to vector extensions. We refer to \cite{Barbieri-VialeSharp-de-Rham-r} for more details on (realizations of) Laumon $1$-motives.

The moduli space of Deligne $1$-motives is a mixed Shimura variety. Based on works of Deligne, Milne, Brylinski and others, Pink \cite{PinkThesis} finished the framework for the study of mixed Shimura varieties. However mixed Shimura varieties do NOT allow any vector extension. In particular, let $\mathfrak{A}_g$ be the universal family over $\cA_g$, where $\cA_g$ is the moduli space of principally polarized abelian varieties with level-$4$-structure, and let $\mathfrak{A}_g^\natural$ be the universal vector extension of the abelian scheme $\mathfrak{A}_g/\cA_g$. Then $\mathfrak{A}_g^\natural$ is NOT a mixed Shimura variety. This suggests that in order to parametrize Laumon $1$-motives we must extend this notion. Note that $\mathfrak{A}_g^\natural$ is still not large enough: it cannot give any information on the weight $-2$ part of Laumon $1$-motives, \textit{i.e.}  the toric part of $G$.

\subsection{Goals of this paper} In $\mathsection$\ref{SubsectionModuliSpaceInTheDelignePinkLanguage}-$\mathsection$\ref{SubsectionQuasiLinearSubvarietiesAndGeometricBiAlgebraicity} of the Introduction, we briefly explain the three main goals of the paper, each one occupying a subsection.

\subsection{Moduli space in the Deligne-Pink language: enlarged mixed Shimura varieties}\label{SubsectionModuliSpaceInTheDelignePinkLanguage}
The \textit{\textbf{first main goal}} of this paper is to develop the theory of enlarged mixed Shimura varieties, in order to allow vector extensions and parametrize Laumon $1$-motives. This is done in $\mathsection$\ref{SectionEnlargedMixedShimuraData} and $\mathsection$\ref{SectionEnlargedMixedShimuraVarieties}. We define a pair, called an enlarged mixed Shimura datum, consisting of a $\Q$-group $P$ and a simply-connected complex analytic space $\cX^\natural$ satisfying some properties (see Definition~\ref{DefinitionEnlargedMixedShimuraDatum}), and define an enlarged mixed Shimura variety to be the quotient of $\cX^\natural$ by a congruence subgroup of $P(\Q)$. Then we prove that any such defined enlarged mixed Shimura variety admits a structure of algebraic variety over a number field canonically associated to the enlarged mixed Shimura datum. This is exactly what Deligne and Pink did for (mixed) Shimura varieties, so we call this language of enlarged mixed Shimura data and enlarged mixed Shimura varieties the \textbf{Deligne-Pink language}. The importance of the Deligne-Pink language for the study of mixed Shimura varieties is revealed by almost any article on (mixed) Shimura varieties. As for enlarged mixed Shimura varieties, let me just point out that every result we prove in this paper (characterization of geometrically bi-algebraic subvarieties, the logarithmic Ax theorem, the Ax-Lindemann theorem and the Ax-Schanuel conjecture for the unipotent part) relies heavily on this language.

The $\mathfrak{A}_g^\natural$ defined at the end of last subsection is an example of enlarged mixed Shimura varieties which are not mixed Shimura varieties. Another such example is the universal vectorial bi-extension $\mathfrak{P}^\natural_g$ studied by Coleman \cite{ColemanThe-universal-v} (see Example~\ref{ExampleEMSV}(2)), which is defined as follows in geometric terms: Let $\mathfrak{P}_g$ be the universal Poincar\'{e} biextension, \textit{i.e.}  the $\G_m$-torsor over $\mathfrak{A}_g\times\mathfrak{A}_g^\vee$ whose fiber $(\mathfrak{P}_g)_a$ for any point $a\in\cA_g$ is the Poincar\'{e} biextension over $(\mathfrak{A}_g)_a\times(\mathfrak{A}_g^\vee)_a$. Let $\mathfrak{A}_g^\natural$ be the universal vector extension of the abelian scheme $\mathfrak{A}_g/\cA_g$. Then $\mathfrak{P}_g^\natural$ is the pullback of $\mathfrak{P}_g$ by $\mathfrak{A}_g^\natural\times(\mathfrak{A}_g^\vee)^{\natural} \rightarrow \mathfrak{A}_g\times\mathfrak{A}_g^\vee$. The typical example of enlarged mixed Shimura variety to keep in mind is $(\mathfrak{P}_g^{\natural})^{[n]}$, \textit{i.e.}  the $n$-fiber product of $\mathfrak{P}_g^\natural$ over $\mathfrak{A}_g^\natural\times(\mathfrak{A}_g^\vee)^{\natural}$. The advantage of $(\mathfrak{P}_g^{\natural})^{[n]}$ to $\mathfrak{A}_g^\natural$ is that it also reflects information on the weight $-2$ part of Laumon $1$-motives.

\subsection{Functional transcendental results: Ax-Schanuel}\label{SubsectionFunctionalTranscendentalResults}
Let $S^\natural$ be a connected enlarged mixed Shimura variety and let $\unif^\natural\colon\cX^{\natural+}\rightarrow S^\natural$ be its uniformization. 
As we shall see in $\mathsection$\ref{SubsectionAlgebraicStructureOnXnatural}, there exists a complex algebraic variety $\cX^{\natural,\vee}$ such that $\cX^{\natural+}$ can be embedded as a semi-algebraic open subset (in the usual topology) of $\cX^{\natural,\vee}$. The \textit{\textbf{second main goal}} of this paper is to study some functional transcendental results. We start with (Theorem~\ref{TheoremAxlogarithmic} and $\mathsection$\ref{SectionAxLindemann}):

\begin{thm}[logarithmic Ax]\label{TheoremIntroductionAxlogarithmic}
Let $Z^\natural$ be an irreducible subvariety of $S^\natural$ and let $\tilde{Z}^\natural$ be a complex analytic irreducible component of $(\unif^{\natural})^{-1}(Z^\natural)$. Then the image of $(\tilde{Z}^\natural)^{\mathrm{Zar}}$, the Zariski closure of $\tilde{Z}^\natural$ in $\cX^{\natural+}$,\footnote{It means that $(\tilde{Z}^\natural)^{\mathrm{Zar}}$ is the complex analytic irreducible component of $\cX^{\natural+}\cap$(Zariski closure of $\tilde{Z}^\natural$ in $\cX^{\natural,\vee}$) which contains $\tilde{Z}^\natural$.} under $\unif^\natural$ is quasi-linear. 
\end{thm}

\begin{thm}[Ax-Lindemann]\label{TheoremIntroductionAxLindemann}
Let $\tilde{Z}^\natural$ be a semi-algebraic subset of $\cX^{\natural+}$. Then any irreducible component of $(\unif^\natural(\tilde{Z}^\natural))^{\mathrm{Zar}}$ is quasi-linear.
\end{thm}

Quasi-linear subvarieties of $S^\natural$ will be defined in the next subsection of the Introduction (Definition~\ref{DefinitionQuasiLinearIntroduction}). Theorem~\ref{TheoremIntroductionAxlogarithmic} is proven for pure Shimura varieties by Moonen \cite[3.6,~3.7]{MoonenLinearity-prope} and for mixed Shimura varieties by the author \cite[Theorem~8.1]{GaoTowards-the-And}. Theorem~\ref{TheoremIntroductionAxLindemann} plays an essential role in the proof of the Andr\'{e}-Oort conjecture. It is proven for $Y(1)^n$ by Pila \cite{PilaO-minimality-an}, for projective pure Shimura varieties by Ullmo-Yafaev \cite{UllmoThe-Hyperbolic-}, for $\cA_g$ by Pila-Tsimerman \cite{PilaAxLindemannAg}, for any pure Shimura variety by Klingler-Ullmo-Yafaev \cite{KlinglerThe-Hyperbolic-}, and for any mixed Shimura variety by the author \cite[Theorem~1.2]{GaoTowards-the-And}. 

We make some comparison of the proofs of Theorem~\ref{TheoremIntroductionAxlogarithmic} and Theorem~\ref{TheoremIntroductionAxLindemann} with the author's previous work on mixed Shimura varieties \cite{GaoTowards-the-And}. In both situations it is important to study the geometric bi-algebraic systems associated with the ambient varieties (we will elaborate on this in the next subsection). We need the characterization of geometrically bi-algebraic subvarieties in both situations. For mixed Shimura varieties, logarithmic Ax follows from Andr\'{e}'s result on the algebraic monodromy groups without much effort and the characterization comes as a byproduct. But for enlarged mixed Shimura varieties, Hodge theory itself is not enough. We should study the geometry of $S^\natural$ more carefully and it is better to separate the proofs into two parts: first prove that geometrically bi-algebraic subvarieties of $S^\natural$ are precisely quasi-linear subvarieties, then prove that the target objects in the conclusions are geometrically bi-algebraic. The first part is new compared to \cite{GaoTowards-the-And}, but the second part is only a slight modification of \cite{GaoTowards-the-And}.

A common generalization of Theorem~\ref{TheoremIntroductionAxlogarithmic} and Theorem~\ref{TheoremIntroductionAxLindemann} is the following conjecture.

\begin{conj}[Ax-Schanuel]\label{ConjectureIntroductionAxSchanuel}
Let $\Delta^\natural\subset\cX^{\natural+}\times S^\natural$ be the graph of $\unif^\natural$. Let $\mathbscr{Z}^\natural=\mathrm{graph}(\tilde{Z}^\natural\xrightarrow{\unif^\natural}Z^\natural)$ be a complex analytic irreducible subvariety of $\Delta^\natural$. Let $F^\natural$ be the smallest quasi-linear subvariety of $S^\natural$ which contains $Z^\natural$. 
Then
\begin{enumerate}
\item $\dim(\tilde{Z}^\natural)^{\Zar}+\dim(Z^\natural)^{\Zar} \dim\tilde{Z}^\natural\ge \dim F^\natural$.
\item Let $\mathbscr{B}^\natural:=(\mathbscr{Z}^\natural)^{\Zar}\subset\cX^{\natural+}\times S^\natural$. Then $\dim\mathbf{pr}^{\mathrm{ws}}_{F^\natural}(\mathbscr{B}^\natural)-\dim\mathbf{pr}^{\mathrm{ws}}_{F^\natural}(\mathbscr{Z}^\natural)\ge \dim(F^\natural)^{\mathrm{ws}}$.
\end{enumerate}
\end{conj}
Let us emphasize that while $\mathbscr{Z}^\natural$ is complex analytic, there is no reason that $Z^\natural$ is closed in $S^\natural$. So $F^\natural$ contains the complex analytic closure of $Z^\natural$ in $S^\natural$.

We will explain part (2) in the next subsection of the Introduction (below Definition~\ref{DefinitionQuasiLinearIntroduction}). For the moment let me just point out why both parts of Conjecture~\ref{ConjectureIntroductionAxSchanuel} are needed:
\begin{itemize}
\item Conjecture~\ref{ConjectureIntroductionAxSchanuel} is a natural generalization of the Ax-Schanuel theorem \cite{AxOn-Schanuels-co} (see \cite[Introduction]{TsimermanAx-Schanuel-and}). Ax' theorem, concerning $\C^n\rightarrow(\C^*)^n$, implies both logarithmic Ax and Ax-Lindemann for this bi-algebraic system, and describes the Zariski closure of complex analytic subvarieties of the graph. This is still expected to hold for mixed Shimura varieties (Conjecture~\ref{ConjectureAxSchanuelMSV}). If we want to make an Ax-Schanuel conjecture for enlarged mixed Shimura varieties, it should then reflect both aspects. Note that for mixed Shimura varieties, part (2) already implies part (1).
\item Part (1) of Conjecture~\ref{ConjectureIntroductionAxSchanuel} implies both logarithmic Ax (Theorem~\ref{TheoremIntroductionAxlogarithmic}) and Ax-Lindemann (Theorem~\ref{TheoremIntroductionAxLindemann}), while part (2) does not.
\item Part (2) of Conjecture~\ref{ConjectureIntroductionAxSchanuel} describes the Zariski closure of $\mathbscr{Z}^\natural$ in the ambient space $\cX^{\natural+}\times S^\natural$, while part (1) does not.
\end{itemize}

We will prove the following cases for Conjecture~\ref{ConjectureIntroductionAxSchanuel}.
\begin{thm}\label{TheoremAxSchanuelKnownCases}
Conjecture~\ref{ConjectureIntroductionAxSchanuel} holds in the following cases:
\begin{enumerate}
\item When $\tilde{Z}^\natural$ is algebraic (equivalent to Ax-Lindemann by Theorem~\ref{TheoremAxSchanuelAndAxLindemann}).
\item When $Z^\natural$ is algebraic (equivalent to logarithmic Ax by Theorem~\ref{TheoremAxSchanuelAndAxlogarithmic}).
\end{enumerate}
\end{thm}

Little is known for Conjecture~\ref{ConjectureIntroductionAxSchanuel} beyond Ax-Lindemann and logarithmic Ax: Ax \cite{AxSome-topics-in-} proved part (2) for the universal vector extension of any semi-abelian variety\footnote{Ax' theorem was about any complex algebraic group with its Lie algebra, but this is easily reduced to the case of abelian groups with their Lie algebras, and hence is equivalent to the statement for the universal vector extension of any semi-abelian variety.}, and Pila-Tsimerman \cite{PilaAx-Schanuel-for} proved the conjecture for $Y(1)^n$ (one crucial point is that each simple factor of the group attached to $Y(1)^n$, which is $\mathrm{SL}_2(\Q)$, is small).

\subsection{Quasi-linear subvarieties and geometric bi-algebraicity}\label{SubsectionQuasiLinearSubvarietiesAndGeometricBiAlgebraicity}
Notation: for any abelian scheme $\mathfrak{A}\rightarrow B$ with unit section $\epsilon$, denote by $\omega_{\mathfrak{A}/B}:=\epsilon^*\Omega_{\mathfrak{A}/B}^1$.

A crucial ingredient for Conjecture~\ref{ConjectureIntroductionAxSchanuel} is to understand the geometric bi-algebraic system associated with $S^\natural$. This is the \textit{\textbf{third main goal}} of this paper. We briefly explain this in this subsection. We say that an irreducible subvariety $Y^\natural$ of $S^\natural$ is \textbf{geometrically bi-algebraic} if one (and hence all) complex analytic irreducible component of $(\unif^{\natural})^{-1}(Y^\natural)$ is algebraizable, \textit{i.e.}  its dimension equals the dimension of its Zariski closure in $\cX^{\natural,\vee}$. As we pointed out, the first step to prove any Ax type theorem is to establish the characterization of geometrically bi-algebraic subvarieties of $S^\natural$. The result cannot be as elegant as for mixed Shimura varieties. Take for example $S^\natural=\mathfrak{A}_g^\natural$. Any algebraic subvariety $Y^\natural$ contained in a fiber of $\mathfrak{A}_g^\natural\rightarrow \mathfrak{A}_g$ is geometrically bi-algebraic. However we prove that this is the only problem: for the exact sequence $0\rightarrow \omega_{\mathfrak{A}_g^\vee/\cA_g} \rightarrow \mathfrak{A}_g^\natural \rightarrow \mathfrak{A}_g \rightarrow 0$ of groups over $\cA_g$, the ``non-linear'' part of any geometrically bi-algebraic subvariety $Y^\natural$ of $\mathfrak{A}_g^\natural$ can only lie in the  TRIVIAL subbundle of the vector bundle part, \textit{i.e.}  of $\omega_{\mathfrak{A}_g^\vee/\cA_g}|_{Y_G}$ where $Y_G$ is the image of $Y^\natural$ in $\cA_g$. More precisely we prove (Theorem~\ref{TheoremCharacterizationOfBiAlgebraicSubvarieties}):

\begin{thm}[Characterization of geometrically bi-algebraic subvarieties of enlarged mixed Shimura varieties]\label{TheoremIntroductionCharacterizationOfBiAlgebraicSubvarieties}
An irreducible subvariety $Y^\natural$ of $S^\natural$ is geometrically bi-algebraic if and only if it is quasi-linear.
\end{thm}

We define quasi-linear subvarieties of $S^\natural$. We shall see in $\mathsection$\ref{SubsectionNotationEMSD} that there is a commutative diagram for any connected enlarged mixed Shimura variety (on the right for $(\mathfrak{P}^\natural_g)^{[n]}$):
\begin{equation}\label{EquationCartesianDiagramForEMSV}
\xymatrix{
S^\natural \ar[r]^{[\pi^\sharp]} \ar[d]_{[\pi^\natural_{P/U}]} \pullback & S \ar[rd]^{[\pi]} \ar[d]_{[\pi_{P/U}]}   & & & (\mathfrak{P}^\natural_g)^{[n]}  \ar[r] \ar[d] \pullback & \mathfrak{P}_g^{[n]} \ar[rd] \ar[d] \\
S^\natural_{P/U} \ar[r]_{[\pi^\sharp_{P/U}]} & S_{P/U} \ar[r]_{[\pi_G]} & S_G & & \mathfrak{A}_g^\natural \times (\mathfrak{A}_g^\vee)^\natural \ar[r] & \mathfrak{A}_g \times \mathfrak{A}_g^\vee \ar[r] & \cA_g
}
\end{equation}
where
\begin{itemize}
\item all maps in the diagram are projections defined in some natural way;
\item $S$ is a connected mixed Shimura variety, and $S_{P/U}$ is the quotient of $S$ by its weight $-2$ part and $S_G$ is its pure part (and hence $S_{P/U}$ is an abelian scheme over $S_G$);
\item $S^\natural_{P/U}$ is the universal vector extension of the abelian scheme $S_{P/U}\rightarrow S_G$.
\end{itemize}

To define quasi-linear subvarieties of $S^\natural$ we need some preparation. Let $Y_G$ be a subvariety of $S_G$. Let $Y_{P/U}\subset S_{P/U}|_{Y_G}:=[\pi_{P/U}]^{-1}(Y_G)$ be the translate of an abelian subscheme of $S_{P/U}|_{Y_G}\rightarrow Y_G$ by a torsion section and then by a constant section of its isotrivial part. Denote by $Y_{P/U}^{\mathrm{univ}}$ the universal vector extension of the abelian scheme $Y_{P/U}\rightarrow Y_G$. Then by the universal property of the universal vector extension there is a unique embedding $Y_{P/U}^{\mathrm{univ}}\subset S^\natural_{P/U}|_{Y_{P/U}}:=[\pi^\sharp_{P/U}]^{-1}(Y_{P/U})$ compatible with the embedding $Y_{P/U}\subset S_{P/U}|_{Y_G}$ mentioned above. More concretely there is a unique embedding $i$ (left vertical arrow) of vector groups over $Y_G$ inducing the following push-out:
\[
\xymatrix{
0 \ar[r] &\omega_{Y_{P/U}^\vee/Y_G} \ar[r] \ar@{_(->}[d]^{i} & Y_{P/U}^{\mathrm{univ}} \ar[r] \ar@{.>}[d] & Y_{P/U} \ar[r] \ar[d]^{=} & 0 \\
0 \ar[r] &\omega_{[\pi_G]^{-1}(Y_G)^\vee/Y_G} \ar[r] &S^\natural_{P/U}|_{Y_{P/U}} \ar[r] \pushoutcorner &Y_{P/U} \ar[r] &0
}
\]
Then we obtain another vector extension of $Y_{P/U}$
\[
0 \rightarrow \frac{\omega_{[\pi_G]^{-1}(Y_G)^\vee/Y_G}}{\omega_{Y_{P/U}^\vee/Y_G}} \rightarrow \frac{S^\natural_{P/U}|_{Y_{P/U}}}{Y_{P/U}^{\mathrm{univ}}} \rightarrow Y_{P/U} \rightarrow 0,
\]
with the unique map $Y_{P/U}^{\mathrm{univ}} \rightarrow \frac{S^\natural_{P/U}|_{Y_{P/U}}}{Y_{P/U}^{\mathrm{univ}}}$ being $0$. Hence $\frac{S^\natural_{P/U}|_{Y_{P/U}}}{Y_{P/U}^{\mathrm{univ}}} \cong Y_{P/U} \times_{Y_G} \frac{\omega_{[\pi_G]^{-1}(Y_G)^\vee/Y_G}}{\omega_{Y_{P/U}^\vee/Y_G}}$. Thus
\[
S^\natural_{P/U}|_{Y_{P/U}}=Y_{P/U}^{\mathrm{univ}}\times_{Y_G}\frac{\omega_{[\pi_G]^{-1}(Y_G)^\vee/Y_G}}{\omega_{Y_{P/U}^\vee/Y_G}}
\]
Denote by $\bV^{(0)}|_{Y_G}$ the largest trivial subbundle of $\frac{\omega_{[\pi_G]^{-1}(Y_G)^\vee/Y_G}}{\omega_{Y_{P/U}^\vee/Y_G}}$. For simplicity we use $\omega^{\mathrm{extr}}$ to denote $\frac{\omega_{[\pi_G]^{-1}(Y_G)^\vee/Y_G}}{\omega_{Y_{P/U}^\vee/Y_G}}$.

If furthermore $Y_G$ is a weakly special subvariety of $S_G$, then denote by $H$ the connected algebraic monodromy group of $Y_G$. Then the pullback of $\omega^{\mathrm{extr}}$ under the universal cover $\tilde{Y}_G \rightarrow Y_G$, which we call $\tilde{\omega}^{\mathrm{extr}}$, is an $H(\R)$-bundle. We say that a subvariety $\mathbf{K}^\natural$ of $\omega^{\mathrm{extr}}$ is an \textbf{\textit{automorphic subvariety}} if it is the image of $H(\R)^+\tilde{K}^\natural$ under the natural projection $\tilde{\omega}^{\mathrm{extr}} \rightarrow \omega^{\mathrm{extr}}$ for some $\tilde{K}^\natural$ in a fiber of $\tilde{\omega}^{\mathrm{extr}} \rightarrow \tilde{Y}_G$. Note that $\tilde{K}^\natural$ can be chosen to be invariant under a maximal compact subgroup of $H(\R)^+$; see \eqref{EqAutomorphicSubvarietyFiberStableCompact}.

Now we are ready to define
\begin{defn}\label{DefinitionQuasiLinearIntroduction}
An irreducible subvariety $Y^\natural$ of $S^\natural$ is called \textbf{quasi-linear} if the followings hold: under the following notations for $Y^\natural$ compatible with \eqref{EquationCartesianDiagramForEMSV}
\[
\xymatrix{
Y^\natural \ar@{|->}[r] \ar@{|->}[d] & Y \ar@{|->}[rd] \ar@{|->}[d]\\
Y^\natural_{P/U} \ar@{|->}[r] & Y_{P/U} \ar@{|->}[r] & Y_G
}
\]
\begin{enumerate}
\item $Y$ is a weakly special subvariety of $S$. In particular $Y_{P/U}$ is the translate of an abelian subscheme of $S_{P/U}|_{Y_G}\rightarrow Y_G$ by a torsion section and then by a constant section of its isotrivial part.
\item Under the notations above the theorem, $Y^\natural_{P/U}=Y_{P/U}^{\mathrm{univ}}\times_{Y_G}(L^\natural\times Y_G)\times_{Y_G} \mathbf{K}^\natural$, where $L^\natural$ is an irreducible algebraic subvariety of any fiber of $\bV^{(0)}|_{Y_G}\rightarrow Y_G$, and $\mathbf{K}^\natural$ is an irreducible automorphic subvariety of the bundle $\frac{\omega_{[\pi_G]^{-1}(Y_G)^\vee/Y_G}}{\omega_{Y_{P/U}^\vee/Y_G}}$ whose intersection with $\bV^{(0)}|_{Y_G}$ is contained in the zero section.
\item $Y^\natural=Y\times_{Y_{P/U}}Y^\natural_{P/U}$ for the cartesian diagram in \eqref{EquationCartesianDiagramForEMSV}.
\end{enumerate}
\end{defn}

Now we are ready to explain the terminology in part (2) of Conjecture~\ref{ConjectureIntroductionAxSchanuel}. Apply Theorem~\ref{TheoremIntroductionCharacterizationOfBiAlgebraicSubvarieties} to the bi-algebraic subvariety $F^\natural$ of $S^\natural$ (hence we change every letter ``$Y$'' by ``$F$''), then we define
\[
(F^\natural)^{\mathrm{ws}}:=F\times_{F_{P/U}} F_{P/U}^{\mathrm{univ}}
\]
and $[pr]^{\mathrm{ws}}_{F^\natural}$ the natural projection $F^\natural\rightarrow(F^\natural)^{\mathrm{ws}}$. Let $pr^{\mathrm{ws}}_{\tilde{F}^\natural}$ be the natural projection from $\tilde{F}^\natural$ to $(\tilde{F}^\natural)^{\mathrm{ws}}$, the uniformization of $(F^\natural)^{\mathrm{ws}}$. Then we define
\[
\mathbf{pr}^{\mathrm{ws}}_{F^\natural}:=(pr^{\mathrm{ws}}_{\tilde{F}^\natural},[pr]^{\mathrm{ws}}_{F^\natural})\colon\tilde{F}^\natural\times F^\natural\rightarrow(\tilde{F}^\natural)^{\mathrm{ws}}\times(F^\natural)^{\mathrm{ws}}.
\]


\subsection{Andr\'{e}-Oort and Zilber-Pink type problems for enlarged mixed Shimura varieties} In $\mathsection$\ref{SectionSpecialSubvarieties}, we briefly discuss about special points and special subvarieties of enlarged mixed Shimura varieties. We will discuss Conjecture~\ref{ConjectureBiAlgebraicArithmeticAndGeometric}.

\subsection*{Structure of the paper} We review Pink's work on equivariant families of mixed Hodge structures in $\mathsection$\ref{SectionReviewOfPinksWork}. This is the foundation of the Deligne-Pink language for enlarged mixed Shimura varieties. Then we define and prove basic properties of enlarged mixed Shimura varieties in $\mathsection$\ref{SectionEnlargedMixedShimuraData} (for enlarged mixed Shimura data) and $\mathsection$\ref{SectionEnlargedMixedShimuraVarieties} (for enlarged mixed Shimura varieties). In particular we study their relationship with the mixed Shimura varieties, explaining both the categorical comparison and the geometrical comparison. We also prove that any enlarged mixed Shimura variety, a priori defined as a complex analytic space, is algebraic and can be canonically descended to its reflex field. $\mathsection$\ref{SectionBiAlgebraicSettingForEnlargedMixedShimuraVarieties} discusses about the geometric bi-algebraicity and proves the characterization of geometrically bi-algebraic subvarieties of enlarged mixed Shimura varieties. Then we study Ax type transcendental results for enlarged mixed Shimura varieties in the next sections. We prove logarithmic Ax in $\mathsection$\ref{SectionTheAxlogarithmicTheorem} and Ax-Lindemann in $\mathsection$\ref{SectionAxLindemann}. 
In $\mathsection$\ref{SectionAxSchanuel} we formulate the Ax-Schanuel conjecture, explain its meaning and prove its relation with logarithmic Ax and with Ax-Lindemann. Finally we make a small discussion about Andr\'{e}-Oort and Zilber-Pink type problems in $\mathsection$\ref{SectionSpecialSubvarieties}.

\subsection*{Acknowledgements} The paper originated from a discussion with Daniel Bertrand. I would like to thank him for the inspiration. I would also like to thank him for teaching me the theory of universal vector extensions and his suggestions to improve the writing. I would like to thank Gisbert W\"{u}stholz for our discussion on the arithmetic bi-algebraicity. I have benefited from discussions with Richard Pink on the first part of the paper, especially his suggestion to simplify the proof of Theorem~\ref{TheoremEMSVCanonicalModel}. I would like to thank him for the interest in the paper and his help. I would like to thank Jonathan Pila, Jacob Tsimerman and Umberto Zannier for their comments and suggestions to improve the presentation of the introduction. I am grateful to the referee for their detailed reading and valuable comments. Part of the work was done when the author was at the IAS (NJ, USA) and the author would like to thank its hospitality. The author was supported by NSF grant DMS-1128155.

\subsection*{Convention} For any abelian scheme $\mathfrak{A}\rightarrow B$ with unit section $\epsilon$, denote by $\omega_{\mathfrak{A}/B}:=\epsilon^*\Omega_{\mathfrak{A}/B}^1$.

Since we will only talk about geometric bi-algebraicity in the main body of the paper, we abbreviate ``geometrically bi-algebraic'' to ``bi-algebraic''.

When we say ``definable'', we mean definable in the o-minimal structure $\R_{\mathrm{an},\exp}$.

\section{Pink's work on equivariant families of mixed Hodge structures}\label{SectionReviewOfPinksWork}
In this section we review Pink's work on equivariant families of mixed Hodge structures. The reference of the whole section is \cite[Chapter~1]{PinkThesis}.

\subsection{Mixed Hodge structure}\label{SubsectionMixedHodgeStructures}
In this subsection we recall some background knowledge about rational mixed Hodge structures. In this subsection, the ring $R=\Z$ or $\Q$.
\subsubsection{Basic facts about mixed Hodge structures}
We start by collecting some basic notions about Hodge structures.

Let $M$ be a free $R$-module of finite rank. A \textbf{pure Hodge structure of weight $n\in\Z$} on $M$ is a decomposition $M_\C=\bigoplus_{p+q=n}M^{p,q}$ into $\C$-vector spaces such that for all $p,q\in\Z$ with $p+q=n$ one has $\bar{M^{q,p}}=M^{p,q}$. The associated \textbf{Hodge filtration} on $M_\C$ is defined by $F^pM_\C:=\bigoplus_{p^\prime\ge  p}M^{p^\prime,q}$. It determines the Hodge structure uniquely, because $M^{p,q}=F^pM_\C\cap\bar{F^qM_\C}$.

A \textbf{mixed $R$-Hodge structure} on $M$ is a triple $(M,\{W_n M\}_{n\in\Z},\{F^pM_\C\}_{p\in\Z})$ consisting of an ascending exhausting separated filtration $\{W_n M\}_{n\in\Z}$ of $M$ by $R$-modules of finite rank with each $M/W_nM$ free, called \textbf{weight filtration}, together with a descending exhausting separated filtration $\{F^pM_\C\}_{p\in\Z}$ of $M_\C$, called \textbf{Hodge filtration}, such that the Hodge filtration induces a pure Hodge structure of weight $n$ on $\Gr^W_n M:=W_n M/W_{n-1}M$ for all $n\in\Z$. A pure Hodge structure of weight $n$ is then a special case of a mixed Hodge structure by defining the weight filtration as $W_{n^\prime}M=M$ for $n^\prime\ge  n$ and $W_{n^\prime}M=0$ for $n^\prime<n$. The notions \textbf{of weight $\leq  n$} and \textbf{of weight $\ge  n$} are defined in the obvious way.

The \textbf{Hodge numbers} are $h^{p,q}:=\dim_\C(\Gr^W_{p+q}M)^{p,q}$. They satisfy $h^{q,p}=h^{p,q}$, almost all $h^{p,q}$ are zero, and $
\sum h^{p,q}=\dim M$. If $A\subset\Z\oplus\Z$ is an arbitrary subset, then we say that the Hodge structure $(M,\{W_n M\}_{n\in\Z},\{F^pM_\C\}_{p\in\Z})$ is \textbf{of type $A$} if $h^{p,q}=0 \Leftrightarrow (p,q)\notin A$.

A \textbf{morphism of mixed $R$-Hodge structures} is a homomorphism $f\colon M\rightarrow M^\prime$ such that $f(W_n M)\subset W_n M^\prime$ and $f(F^pM_\C)\subset F^pM^\prime_\C$ for all $n,p\in\Z$. The rational mixed Hodge structures form an abelian category with these morphisms. Given mixed $R$-Hodge structures on $M_1$ and $M_2$, there are canonical rational mixed Hodge structures on $M_1\oplus M_2$, on the dual $M_1^\vee$ and on $\hom(M_1,M_2)$.

A mixed Hodge structure on $M$ is said to \textbf{split over $\R$} if there exists a decomposition $M_\C=\bigoplus_{p,q}M^{p,q}$ such that $W_n M_\C=\bigoplus_{p+q\leq  n}M^{p,q}$, $F^pM_\C=\bigoplus_{p^\prime\ge  p}M^{p^\prime,q}$ and $\bar{M^{q,p}}=M^{p,q}$. This decomposition is then uniquely determined by these properties. Every pure Hodge structure splits over $\R$, but not every mixed Hodge structure does. However we still have (see \cite[1.2]{PinkThesis})

\begin{prop}[Deligne]\label{PropositionUniquenessOfDecomposition}
Fix a mixed $R$-Hodge structure on $M$.
\begin{enumerate}
\item There exists a decomposition $M_\C=\bigoplus_{p,q}M^{p,q}$ such that $W_n M_\C=\bigoplus_{p+q\leq  n}M^{p,q}$ and $F^pM_\C=\bigoplus_{p^\prime\ge  p}M^{p^\prime,q}$.
\item The Hodge structure is uniquely determined by any such decomposition.
\item There exists a unique decomposition as in (1) which also satisfies
\[
\bar{M^{q,p}}\equiv M^{p,q}\mod\bigoplus_{p^\prime<p,q^\prime<q}M^{p^\prime,q^\prime}.
\]
\end{enumerate}
\end{prop}

\subsubsection{Deligne torus}
Let $\S:=\res_{\C/\R}\G_{m,\C}$. The torus $\S$ is called the \textbf{Deligne torus}. Over $\C$ it is canonically isomorphic to $\G_{m,\C}\times\G_{m,\C}$, but the action of complex conjugation is twisted by the automorphism $c$ that interchanges the two factors. In particular $\S(\R)=\C^*\subset\S(\C)=\C^*\times\C^*$ consists of the points $(z,\bar{z})$ with $z\in\C^*$. While the character group of $\G_{m,\C}$ is $\Z$ in the standard way, we identify the character group of $\S$ with $\Z\oplus\Z$ such that the character $(p,q)$ maps $z\in\S(\R)=\C^*$ to $z^{-p}\bar{z}^{-q}\in\C^*$. Under this identification the complex conjugation operates on $\Z\oplus\Z$ by interchanging the two factors. The following homomorphisms are important:
\begin{itemize}
\item the weight $\omega\colon\G_{m,\R}\hookrightarrow\S$ induced by $\R^*\subset\C^*$;
\item $\mu\colon\G_{m,\C}\rightarrow\S_\C$ sending $z\in\C^*\mapsto(z,1)\in\C^*\times\C^*=\S(\C)$;
\item the norm $N\colon\S\twoheadrightarrow\G_{m,\R}$ sending $z\in\S(\R)=\C^*\mapsto z\bar{z}\in\R^*$. The kernel $\S^1$ of $N$ is anisotropic over $\R$, and we have a short exact sequence $1\rightarrow\S^1\rightarrow\S\rightarrow\G_{m,\R}\rightarrow1$.
\end{itemize}

Let $M$ be a free $R$-module of finite rank. The choice of a representation $k\colon\S_\C\rightarrow\GL(M_\C)$ is equivalent to the choice of a decomposition $M_\C=\bigoplus_{p,q}M^{p,q}$, where $M^{p,q}$ is the eigenspace in $M_\C$ to the character $(p,q)$. Define $W_n M_\C=\bigoplus_{p+q\leq  n}M^{p,q}$ and $F^pM_\C=\bigoplus_{p^\prime\ge  p}M^{p^\prime,q}$. We want to understand when the triple $(M,\{W_n M_\C\},\{F^pM_\C\})$ is a mixed $R$-Hodge structure on $M$ (so in particular $W_n M_\C$ is defined over $R$ for all $n$). The following two propositions of Pink will tell us under which condition on $k$ this is the case for $R=\Q$.
\begin{prop}\label{HodgeStructureProp1}
(\cite[1.4]{PinkThesis}) Let $P$ be a connected $\Q$-linear algebraic group. Let $W:=\cR_u(P)$ be its unipotent radical, let $G:=P/W$ and let $\pi\colon P\rightarrow G$ be the quotient map. Let $h\colon\S_\C\rightarrow P_\C$ be a homomorphism such that the following conditions hold:
\begin{itemize}
\item $\pi\circ h\colon\S_\C\rightarrow G_\C$ is defined over $\R$;
\item $\pi\circ h\circ\omega\colon\G_{m,\R}\rightarrow G_\R$ is a cocharacter of the center of $G$ defined over $\Q$;
\item Under the weight filtration on $(\lie P)_\C$ defined by $\mathrm{Ad}_P\circ h$ we have $W_{-1}(\lie P)=\lie W$.
\end{itemize}
Then
\begin{enumerate}
\item For every ($\Q$-)representation $\rho\colon P\rightarrow\GL(M)$, the homomorphism ${\rho\circ h}\colon$ $\S_\C\rightarrow\GL(M_\C)$ induces a rational mixed Hodge structure on $M$.
\item The weight filtration on $M$ is stable under $P$.
\item For any $p\in P(\R)W(\C)$, the assertions (1) and (2) also hold for ${\mathrm{int}(p)\circ h}$ in place of $h$. The weight filtration and the Hodge numbers in any representation are the same for ${\mathrm{int}(p)\circ h}$ and for $h$.
\end{enumerate}
\end{prop}

\begin{prop}\label{HodgeStructureProp2}
Let $M$ be a finite dimensional $\Q$-vector space. A representation $k\colon\S_\C\rightarrow\GL(M_\C)$ defines a rational mixed Hodge structure on $M$ if and only if  there exist a connected $\Q$-linear algebraic group $P$, a representation $\rho\colon P\rightarrow\GL(M)$ and a homomorphism $h\colon\S_\C\rightarrow P_\C$ such that $k=\rho\circ h$ and the conditions in Proposition~\ref{HodgeStructureProp1} are satisfied. Moreover, every rational mixed Hodge structure on $M$ is obtained by a unique representation $k\colon\S_\C\rightarrow\GL(M_\C)$ with the property above.
\begin{proof} This is \cite[1.5]{PinkThesis} except the ``Moreover'' part, where the existence of $k$ has been explained in the paragraph before Proposition~\ref{HodgeStructureProp1} and the uniqueness of $k$ follows from Proposition \ref{PropositionUniquenessOfDecomposition}(3).
\end{proof}
\end{prop}

\subsubsection{Mumford-Tate group and polarizations}
Let $M$ be a free $R$-module of finite rank equipped with a mixed $R$-Hodge structure $(M,\{W_n M\}_{n\in\Z},\{F^pM_\C\}_{p\in\Z})$. By Proposition~\ref{HodgeStructureProp2}, the corresponding rational mixed Hodge structure on $M_\Q$ gives rises to a representation $k\colon\S_\C\rightarrow\GL(M_\C)$.

\begin{defn} The \textbf{Mumford-Tate group} of this mixed $R$-Hodge structure is defined to be the smallest $\Q$-subgroup $P$ of $\GL(M_\Q)$ such that $k(\S_\C)\subset P_\C$.
\end{defn}

Before defining the polarizations of pure Hodge structures, we introduce the \textbf{Tate Hodge structure}, which is defined to be the free $R$-module of rank $1$ $R(1):=2\pi\sqrt{-1}R$ with the pure $R$-Hodge structure of type $(-1,-1)$. For every $n\in\Z$, we get a pure $R$-Hodge structure of type $(-n,-n)$ on $R(n):=R(1)^{\otimes n}$.

\begin{defn} Suppose that the $R$-Hodge structure on $M$ is pure of weight $n$. A \textbf{polarization} of this Hodge structure is a homomorphism of Hodge structures
\[
Q\colon M\otimes M\rightarrow R(-n)
\]
which is $(-1)^n$-symmetric and such that the real-valued symmetric bilinear form $Q^\prime(u,v):=(2\pi\sqrt{-1})^nQ(Cu,v)$ is positive-definite on $M_\R$, where $C$ acts on $M^{p,q}$ by $C|_{M^{p,q}}=(\sqrt{-1})^{p-q}$.
\end{defn}

\subsubsection{Variation of mixed Hodge structures}\label{SubsubsectionVariationOfMixedHodgeStructures}
\begin{defn}(\cite[Definition~14.44]{PetersMixed-Hodge-Str}) Let $S$ be a complex manifold. A \textbf{variation of mixed $R$-Hodge structures over $S$} is a triple $(\V,W_\cdot,\cF^\cdot)$ with
\begin{enumerate}
\item a local system $\V$ of free $R$-modules of finite rank on $S$;
\item a finite increasing filtration $\{W_m\}$ of the local system $\V$ by local sub-systems with torsion free $\Gr^W_n\V$ for each $n$ (this is called the weight filtration);
\item a finite decreasing filtration $\{\cF^p\}$ of the holomorphic vector bundle $\cV:=\V\otimes_{R_S}\cO_S$, where $R_S$ is the constant sheaf over $S$, by holomorphic subbundles (this is called the Hodge filtration).
\end{enumerate}
such that 
\begin{enumerate}
\item for each $s\in S$, the filtrations $\{\cF^p(s)\}$ and $\{W_m\}$ of $\V(s)\cong\V_s\otimes_R\C$ define a mixed Hodge structure on the $R$-module of finite rank $\V_s$;
\item the connection $\nabla:\cV\rightarrow\cV\otimes_{\cO_S}\Omega_S^1$ whose sheaf of horizontal sections is $\V_\C$ satisfies the Grif and only if iths' transversality condition
\[
\nabla(\cF^p)\subset\cF^{p-1}\otimes\Omega_S^1.
\]
\end{enumerate}
\end{defn}

\begin{defn} A variation of mixed Hodge structures over $S$ is said to be \textbf{graded-polarizable} if the induced variations of pure Hodge structure $\Gr^W_n\V$ are all polarizable, \textit{i.e.}  for each $n$, there exists a flat morphism of variations
\[
Q_n\colon\Gr^W_n\V\otimes\Gr^W_n\V\rightarrow R(-n)_S
\]
which induces on each fibre a polarization of the corresponding Hodge structure of weight $n$.
\end{defn}

\subsection{Equivariant families of Hodge structures}
Now we are ready to discuss equivariant families of Hodge structures, or more precisely homogeneous spaces parametrizing certain rational mixed Hodge structures.

\begin{prop}[Pink {\cite[1.7]{PinkThesis}}]\label{PropositionComplexStructureOfUniformizingSpaceHodgeStructureViewpoint}
Let $P$ be a connected $\Q$-linear group and let $W:=\cR_u(P)$ be its unipotent radical. Let $\cD^\natural$ be a $P(\R)W(\C)$-conjugacy class in $\hom(\S_\C,P_\C)$. Assume that for one (and hence for all by Proposition~\ref{HodgeStructureProp1}(3)) $h\in\cD^\natural$, the conditions in Proposition~\ref{HodgeStructureProp1} holds. Let $M$ be any faithful representation of $P$ and let $\varphi$ be the obvious map
\[
\varphi\colon\cD^\natural\rightarrow\{\text{rational mixed Hodge structures on }M\}
\]
given by Propostion~\ref{HodgeStructureProp1}(1). Then:
\begin{enumerate}
\item There exists a unique structure on $\varphi(\cD^\natural)$ as a complex manifold such that the Hodge filtration on $M_\C$ depends analytically on $\varphi(h)\in\varphi(\cD^\natural)$. This structure is $P(\R)W(\C)$-invariant and $W(\C)$ acts analytically on $\varphi(\cD^\natural)$.
\item For any other representation $M^\prime$ of $P$ the analogous map
\[
\varphi^\prime\colon\cD^\natural\rightarrow\{\text{rational mixed Hodge structures on }M^\prime\}
\]
factors through $\varphi(\cD^\natural)$. The Hodge filtration on $M^\prime$ varies analytically with $\varphi(h)\in\varphi(\cD^\natural)$.
\item If in addition $M^\prime$ is faithful, then $\varphi(\cD^\natural)$ and $\varphi^\prime(\cD^\natural)$ are canonically isomorphic and the isomorphism is compatible with the complex structure.
\end{enumerate}
\end{prop}

\begin{rmk}\label{RemarkStructureOfVarphi}
By the proof of \cite[1.7]{PinkThesis}, the map $\varphi$ factors through
\[
\cD^\natural\cong P(\R)W(\C)/\mathrm{Cent}_{P(\R)W(\C)}(h)\rightarrow P(\C)/F^0_h P_\C\hookrightarrow\mathrm{Grass}(M)(\C),
\]
where the last map is a closed embedding.
\end{rmk}

The following lemma will be useful:
\begin{lemma}[Pink {\cite[1.8]{PinkThesis}}]\label{LemmaCentralizerOfh}
Let $P$, $\cD^\natural$, $M$ and $\varphi$ be as in Proposition~\ref{PropositionComplexStructureOfUniformizingSpaceHodgeStructureViewpoint}. Then for any $h\in\cD^\natural$, the projection $\pi\colon P\rightarrow G:=P/\cR_u(P)$ induces an isomorphism
\[
\mathrm{Cent}_{P(\R)W(\C)}(h)\xrightarrow{\sim}\mathrm{Cent}_{G(\R)}(\pi\circ h).
\]
\end{lemma}

Two natural questions about this equivariant family of Hodge structures arise: under which condition do we get a variation of rational Hodge structure on $M$ over $\varphi(\cD^\natural)$? Is there a subset $\cD$ of $\cD^\natural$ having the same image under $\varphi$ such that $\varphi|_{\cD}$ is finite? Both questions are answered by Pink. In the following two propositions, we let $P$, $\cD^\natural$, $M$ and $\varphi$ be as in Proposition~\ref{PropositionComplexStructureOfUniformizingSpaceHodgeStructureViewpoint}.

\begin{prop}[Pink {\cite[1.10]{PinkThesis}}]\label{PropositionVariationOnUniformizingSpace}
We have a variation of rational mixed Hodge structures on $M$ over $\varphi(\cD^\natural)$ if and only if  for one (and hence for all) $h\in\cD^\natural$ the Hodge structure on $\lie P$ is of type
\[
\{(-1,1),(0,0),(1,-1),(-1,0),(0,-1),(-1,-1)\}.
\]
\end{prop}

\begin{prop}[Pink {\cite[1.16]{PinkThesis}}]\label{PropositionReplaceByASmallerOrbit}
Let $U<W$ be the unique connected subgroup such that $\lie U=W_{-2}(\lie W)$ (by Proposition~\ref{HodgeStructureProp1}(3), it does not depend on $h\in\cD^\natural$). Let $\pi^\prime$ be the quotient $P\rightarrow P/U$. Let
\[
\cD:=\{h\in\cD^\natural|~\pi^\prime\circ h\colon\S_\C\rightarrow(P/U)_\C\text{ is defined over }\R\}.
\]
Then
\begin{enumerate}
\item $\cD$ is a non-empty $P(\R)U(\C)$-orbit in $\hom(\S_\C,P_\C)$;
\item $\varphi(\cD)=\varphi(\cD^\natural)$;
\item If $F^0(\lie U)_\C=0$, then $\varphi(\cD)\cong\cD$.
\end{enumerate}
\end{prop}

Before proceeding to the next section, we make the following remark: in the study of mixed Shimura varieties one often identifies $\varphi(\cD)$ and $\cD$ by Proposition~\ref{PropositionReplaceByASmallerOrbit}(3), however to study enlarged mixed Shimura varieties it is important to distinguish $\varphi(\cD)$ and $\cD$. We will be careful about this in the whole paper.

\section{Enlarged mixed Shimura data}\label{SectionEnlargedMixedShimuraData}

\begin{defn}\label{DefinitionEnlargedMixedShimuraDatum}
An \textbf{enlarged mixed Shimura datum $(P,\cX^\natural,h)$} is a triple where
\begin{itemize}
\item $P$ is a connected linear algebraic group over $\Q$. We denote by $W$ its unipotent radical and by $U \subset W$ be the connected normal subgroup of $P$ uniquely determined by condition \eqref{definition of U} below;
\item $\cX^\natural$ is a left homogeneous space under the subgroup $P(\R)W(\C)\subset P(\C)$, and $\cX^\natural\xrightarrow{h}\hom(\S_\C,P_\C)$ is a $P(\R)W(\C)$-equivariant map\footnote{See Remark~\ref{RemarkAfterDefinitionOfEMSD}\eqref{RemarkMSD} for some discussion.} such that every fibre of $h$ consists of at most finitely many points,
\end{itemize}
such that for some (equivalently for all) $x\in\cX^\natural$,
\begin{enumerate}
\item the composite homomorphism $\S_\C\xrightarrow{h_x}P_\C\rightarrow(P/W)_\C$ is defined over $\R$,
\item the adjoint representation induces on $\lie P$ a rational mixed Hodge structure of type
\[
\{(-1,1),(0,0),(1,-1)\}\cup\{(-1,0),(0,-1)\}\cup\{(-1,-1)\},
\]
\item\label{definition of U}
the weight filtration on $\lie P$ is given by
\[
W_n(\lie P)=\begin{cases}
0 & \text{if }n<-2 \\
\lie U &\text{if }n=-2 \\
\lie W &\text{if }n=-1 \\
\lie P &\text{if }n\ge  0
\end{cases},
\]
\item the conjugation by $h_x(\sqrt{-1})$ induces a Cartan involution on $G_\R^\ad$ where $G:=P/W$, and $G^\ad$ possesses no $\Q$-factor $H$ such that $H(\R)$ is compact,
\item $P/P^\der=Z(G)$ acts on $U$ and on $V:=W/U$ through a torus which is an almost direct product of a $\Q$-split torus with a torus of compact type over $\Q$.
\end{enumerate}
If in addition $U$ is trivial, then $(P,\cX^\natural,h)$ is said to be \textbf{of Kuga type}. For simplicity we mostly write $(P,\cX^\natural)$ since we always consider exactly one map $h$ for every pair $(P,\cX^\natural)$.
\end{defn}

\begin{rmk}\label{RemarkAfterDefinitionOfEMSD}
\begin{enumerate}
\item Conditions (2) and (3) together imply that the composite homomorphism $\G_{m,\C}\xrightarrow{\omega}\S_\C\xrightarrow{h_x}P_\C\rightarrow(P/W)_\C$ is a cocharacter of the center of $P/W$ defined over $\R$. This map is called the weight.
\item\label{RemarkMSD} We shall compare Definition~\ref{DefinitionEnlargedMixedShimuraDatum} with that of mixed Shimura data in \cite[2.1]{PinkThesis} in $\mathsection$\ref{SubsectionSomeConstructionsOfEMSD} (categorical comparison) and $\mathsection$\ref{SubsectionStructureOfUnderlyingSpace} (geometric comparison). According to Pink \cite[2.1]{PinkThesis}, a mixed Shimura datum is a triple $(P,\cX,h)$ where $P$ is as in Definition~\ref{DefinitionEnlargedMixedShimuraDatum}, $\cX$ is a left homogeneous space under $P(\R)U(\C)$ and $h\colon\cX\rightarrow\hom(\S_\C,P_\C)$ is $P(\R)U(\C)$-equivariant such that all the conditions in Definition~\ref{DefinitionEnlargedMixedShimuraDatum} are satisfied for any $x\in\cX$, with $P/W$ replaced by $P/U$ in condition~(1). For those who are familiar with mixed Shimura data, the action of $P(\R)U(\C)$ on $\cX$ can be extended to an action of $P(\R)W(\C)$.\footnote{Let $M$ and $\varphi$ be as in Proposition~\ref{PropositionComplexStructureOfUniformizingSpaceHodgeStructureViewpoint}, then the action of $P(\R)U(\C)$ on $\varphi\circ h(\cX)$ extends to an action of $P(\R)W(\C)$ by Proposition~\ref{PropositionReplaceByASmallerOrbit}. Hence the action of $P(\R)U(\C)$ on $\cX$ extends to an action of $P(\R)W(\C)$ because $\varphi(h(\cX))\cong h(\cX)$ and every fiber of $h$ is at most finite.} But the map $h\colon\cX\rightarrow\hom(\S_\C,P_\C)$ is NOT $P(\R)W(\C)$-equivariant because $\mathrm{Cent}_{P(\R)W(\C)}(h_x)$ is reductive by Lemma~\ref{LemmaCentralizerOfh}.
\item Any pure Shimura datum or any mixed Shimura datum with trivial weight $-1$ part is by definition an enlarged mixed Shimura datum.
\end{enumerate}
\end{rmk}

In view of Remark~\ref{RemarkAfterDefinitionOfEMSD}(2) and Proposition~\ref{PropositionReplaceByASmallerOrbit}, enlarged mixed Shimura data and mixed Shimura data encrypt the same information from the Hodge theory. Therefore in order to define morphisms between enlarged mixed Shimura data, other aspects should be taken into consideration. We will postpone it to later subsections. For the moment let us define:
\begin{defn}
A \textbf{coarse morphism} $(P_1,\cX^\natural_1,h_1)\rightarrow(P_2,\cX^\natural_2,h_2)$ between enlarged mixed Shimura data consists of a homomorphism $f\colon P_1\rightarrow P_2$ and a $P_1(\R)W_1(\C)$-equivariant map $f_s\colon\cX^\natural_1\rightarrow\cX^\natural_2$ such that the following diagram commutes:
\[
\xymatrix{ 
    \cX^\natural_1 \ar[r]^{f_s} \ar[d]^{h_1}  & \cX^\natural_2 \ar[d]^{h_2} \\
    \hom(\S_\C,P_{1,\C}) \ar[r]^{h\mapsto f\circ h} & \hom(\S_\C,P_{2,\C}) }
\]
\end{defn}

\subsection{Categorical comparison and some constructions}\label{SubsectionSomeConstructionsOfEMSD}
We define morphisms between enlarged mixed Shimura data in the current subsection. Its geometric aspect (beyond Hodge theory) will be discussed in the next subsection.

We need the following preparation.

Given an enlarged mixed Shimura datum $(P,\cX^\natural,h)$, let $\cD^\natural:=h(\cX^\natural)$ and let $\cD$ be the subset of $\cD^\natural$ defined in Proposition~\ref{PropositionReplaceByASmallerOrbit}. Let $\cX:=h^{-1}(\cD)\subset\cX^\natural$. Then it is easy to check that $(P,\cX,h|_{\cX})$ is a mixed Shimura datum. We say that $(P,\cX,h|_{\cX})$ is the \textbf{mixed Shimura datum associated with} $(P,\cX^\natural,h)$.

\begin{defn}\label{DefinitionMorphismsInTheCategoryOfEMSD}
\begin{enumerate}
\item A \textbf{morphism} $f \colon (P_1,\cX^\natural_1, h_1) \rightarrow (P_2, \cX^\natural,h_2)$ between enlarged mixed Shimura data is a coarse morphism satisfying the following property. For the associated mixed Shimura data $(P_1,\cX_1,h_1|_{\cX_1})$ and $(P_2,\cX_2,h_2|_{\cX_2})$ defined as above, we have $f(\cX_1) \subset \cX_2$.
\item The \textbf{category of enlarged mixed Shimura data} \underline{$\mathbf{\cE\cM\cS\cD}$} is defined as follows: its objects are the collection of enlarged mixed Shimura data, and its morphisms are the collection of morphisms between enlarged mixed Shimura data.
\end{enumerate}
\end{defn}

We have the following \textit{categorical comparison}.
\begin{lemma}[Categorical Comparison Lemma between \underline{$\mathbf{\cE\cM\cS\cD}$} and MSD]\label{LemmaCategoricalComparisonBetweenEMSDandMSD}
The category \underline{$\mathbf{\cE\cM\cS\cD}$} and the category of mixed Shimura data \underline{$\mathbf{\cM\cS\cD}$} are equivalent.
\begin{proof} By definition of \underline{$\mathbf{\cE\cM\cS\cD}$}, we have a functor $\mathbf{F} \colon \underline{\mathbf{\cE\cM\cS\cD}} \rightarrow \underline{\mathbf{\cM\cS\cD}}$.

Given a mixed Shimura datum $(P,\cX,h)$, let $\cD:=h(\cX)$ and let $\cD^\natural$ be the $P(\R)W(\C)$-conjugacy class in $\hom(\S_\C,P_\C)$ of any element of $\cD$. Let $M$ and $\varphi$ be as in Proposition~\ref{PropositionComplexStructureOfUniformizingSpaceHodgeStructureViewpoint}. Then we have $\varphi(\cD^\natural)=\varphi(\cD)\cong\cD$ by Proposition~\ref{PropositionReplaceByASmallerOrbit}. Now take the fiber product $\cX^\natural:=\cX\times_{\varphi(\cD^\natural)}\cD^\natural$ and let $h^\natural$ be the projection of $\cX^\natural$ to the second factor in the fiber product. Recall that $\varphi$ is $P(\R)W(\C)$-equivariant. Then it is not hard to check that $(P,\cX^\natural,h^\natural)$ is an enlarged mixed Shimura datum. We say that $(P,\cX^\natural,h^\natural)$ is the \textbf{enlarged mixed Shimura datum associated with} $(P,\cX,h)$. The upshot is that $\mathbf{F}$ is essentially surjective.

Now given a morphism $f \colon (P_1,\cX_1,h_1) \rightarrow (P_2,\cX_2,h_2)$ of mixed Shimura data, let us construct a morphism of the associated enlarged mixed Shimura data $f^\natural \colon (P_1,\cX_1^\natural,h_1^\natural) \rightarrow (P_2,\cX_2^\natural,h_2^\natural)$. Now $f$ gives a $P_1(\R)U_1(\C)$-invariant map $\cD_1 \rightarrow \cD_2$, which naturally extends to a $P_1(\R)W_1(\C)$-invariant map $\cD_1^\natural \rightarrow \cD_2^\natural$. Hence by construction of $\cX^\natural$ we obtain a coarse morphism $f^\natural \colon (P_1,\cX_1^\natural,h_1^\natural) \rightarrow (P_2,\cX_2^\natural,h_2^\natural)$ of enlarged mixed Shimura data. It is not hard to check that $f^\natural$ is a morphism between enlarged mixed Shimura data. Hence $\mathbf{F}$ is full.

It remains to show that $\mathbf{F}$ is faithful. Suppose $f,f'$ are morphisms $(P_1,\cX_1^\natural,h_1) \rightarrow (P_2,\cX_2^\natural,h_2)$ such that $f|_{(P_1,\cX_1,h_1|_{\cX_1})} = f'|_{(P_1,\cX_1,h_1|_{\cX_1})}$ where $(P_1,\cX_1,h_1|_{\cX_1})$ is the associated mixed Shimura datum. Since $\cX_1$ is defined to be a subset of $\cX^\natural_1$ and $f,f'$ are both $P_1(\R)W_1(\C)$-equivariant, we have $f = f'$. Hence we are done.
\end{proof}
\end{lemma}

With this lemma and known results of mixed Shimura data, we can do some construction for enlarged mixed Shimura data.
\begin{prop}[Quotient]\label{PropositionQuotientOfEMSD}
Let $(P,\cX^\natural)$ be an enlarged mixed Shimura datum and let $P_0$ be a normal subgroup of $P$. Then there exists a quotient enlarged mixed Shimura datum $(P,\cX^\natural)/P_0$ and a morphism $(P,\cX^\natural)\rightarrow(P,\cX^\natural)/P_0$, unique up to isomorphism, such that every morphism $(P,\cX^\natural)\rightarrow(Q,\cY^\natural)$, where the homomorphism $P\rightarrow Q$ factors through $P/P_0$, factors in a unique way through $(P,\cX^\natural)/P_0$. Moreover let $(P,\cX)$ be the mixed Shimura datum associated with $(P,\cX^\natural)$, then $(P,\cX^\natural)/P_0$ is the enlarged mixed Shimura datum associated with $(P,\cX)/P_0$.
\begin{proof} Pink \cite[2.9]{PinkThesis} proved the existence of the quotient mixed Shimura datum $(P,\cX)/P_0$. Then the proposition follows from the categorical comparison lemma (Lemma~\ref{LemmaCategoricalComparisonBetweenEMSDandMSD}) above.
\end{proof}
\end{prop}

\begin{prop}[Unipotent extension]\label{PropositionUnipotentExtensionOfEMSD}
Let $(P,\cX^\natural)$ be an enlarged mixed Shimura datum and let $1\rightarrow W_0\rightarrow Q\rightarrow P\rightarrow 1$ be an extension of $P$ by a unipotent group $W_0$. Assume
\begin{itemize}
\item the adjoint action of $Q$ on every abelian subquotient of $W_0$ factors through $P$;
\item every irreducible subquotient of $\lie W_0$ is of type $\{(-1,0),(0,-1),(-1,-1)\}$ as a representation of $G:=P/\cR_u(P)$;
\item the center of $G$ acts on every irreducible subquotient of $\lie W_0$ through a torus that is an almost direct product of a $\Q$-split torus with a torus of compact type over $\Q$.
\end{itemize}
Then we have
\begin{enumerate}
\item There exists an enlarged mixed Shimura datum $(Q,\cY^\natural)$ and a morphism $(Q,\cY^\natural)\rightarrow(P,\cX^\natural)$, both unique up to isomorphism, extending the homomorphism $Q\rightarrow P$ such that $(Q,\cY^\natural)/W_0\cong(P,\cX^\natural)$.
\item For every morphism $(P_1,\cX^\natural_1)\rightarrow(P,\cX^\natural)$ and every factorization $P_1\rightarrow Q\rightarrow P$, there exists a unique extension $(P_1,\cX^\natural_1)\rightarrow(Q,\cY^\natural)\rightarrow(P,\cX^\natural)$.
\item Let $(P,\cX)$ be the mixed Shimura datum associated with $(P,\cX^\natural)$ and let $(Q,\cY)$ be the unipotent extension of $(P,\cX)$ be $W_0$ defined by Pink \cite[2.17]{PinkThesis}. Then $(Q,\cY^\natural)$ is the enlarged mixed Shimura datum associated with $(Q,\cY)$.
\end{enumerate}
\begin{proof} In fact Conclusion~(3) gives the construction and the categorical comparison lemma (Lemma~\ref{LemmaCategoricalComparisonBetweenEMSDandMSD}) proves that this is what we desire.
\end{proof}
\end{prop}

\subsection{Structure of the underlying space and the geometric comparison}\label{SubsectionStructureOfUnderlyingSpace}
In this subsection we study the underlying space $\cX^\natural$ of an enlarged mixed Shimura datum. In particular we will make a \textit{geometric comparison} between enlarged mixed Shimura data and mixed Shimura data. This explains the geometric aspect of the definition of morphisms of enlarged mixed Shimura data.

Given an enlarged mixed Shimura datum $(P,\cX^\natural,h)$, let $\cD^\natural:=h(\cX^\natural)$ and let $\cD$ be the subset of $\cD^\natural$ defined in Proposition~\ref{PropositionReplaceByASmallerOrbit}. Let $M$ and $\varphi$ be as in Proposition~\ref{PropositionComplexStructureOfUniformizingSpaceHodgeStructureViewpoint}. Then $\varphi(\cD^\natural)=\varphi(\cD)\cong\cD$ by Proposition~\ref{PropositionReplaceByASmallerOrbit}. Hence the triple $(P,\varphi(\cD),\varphi|_{\cD}^{-1})$ is a mixed Shimura datum. Consider the following projection map, which is the \textbf{geometric comparison} We use:
\begin{equation}\label{EquationEnlargedMSDToMSD}
\pi^\sharp\colon(P,\cX^\natural,h)\rightarrow(P,\varphi(\cD),\varphi|_{\cD}^{-1}),\quad(p,x)\mapsto(p,\varphi(h_x)).
\end{equation}
We write $(P,\cX^\natural)$ and $(P,\varphi(\cD))$ for simplicity. Note that any morphism of enlarged mixed Shimura data $f \colon (P_1,\cX^\natural_1) \rightarrow (P_2,\cX^\natural_2)$ induces a natural morphism of mixed Shimura data $\bar{f} \colon (P_1,\varphi_1(\cD_1)) \rightarrow (P_2,\varphi_2(\cD_2))$ such that the following diagram commutes
\[
\xymatrix{
(P_1,\cX^\natural_1) \ar[r]^-{\pi^\sharp_1} \ar[d]_{f} & (P_1,\varphi_1(\cD_1)) \ar[d]^{\bar{f}} \\
(P_2,\cX^\natural_2) \ar[r]^-{\pi^\sharp_2} & (P_2,\varphi_2(\cD_2))
}
\]
in the following way: let $(P_i,\cX_i)$ be the associated mixed Shimura data as defined above Definition~\ref{DefinitionMorphismsInTheCategoryOfEMSD}, then by definition of morphisms of enlarged mixed Shimura data we have $f(\cX_1) \subset \cX_2$. Thus $f$ induces a map from $\cD_1 = h(\cX_1)$ to $\cD_2 = h(\cX_2)$. As $\varphi_i|_{\cD_i} \colon \cD_i \cong \varphi_i(\cD_i)$, we get a map $\bar{f} \colon \varphi_1(\cD_1) \rightarrow \varphi_2(\cD_2)$, which is certainly $P_1(\R)U_1(\C)$-equivariant by construction.

Let $\pi^\natural$ be the composition
\[
(P,\cX^\natural)\xrightarrow{\pi^\sharp}(P,\varphi(\cD))\xrightarrow{\pi}(P,\varphi(\cD))/W=:(G,\varphi(\cD)_G).
\]
We have $\cD^\natural=h(\cX^\natural)\cong P(\R)W(\C)/\mathrm{Cent}(h_x)$ for any $h_x\in\cD^\natural$, and hence by Lemma~\ref{LemmaCentralizerOfh} $\cD^\natural$ is a $W(\C)$-torsor over the complex manifold $\varphi(\cD)_G$. This endows $\cD^{\natural}$ with a natural complex structure. More precisely, any Levi-decomposition $P = W \rtimes G$ gives a section $\varphi(\cD)_G \rightarrow \cD^\natural$, thus trivializing the $W(\C)$-torsor $\cD^\natural \cong W(\C) \times \varphi(\cD)_G$. This endows $\cD^\natural$ with a complex structure. Since two Levi-decompositions of $P$ differ from the conjugation of an element in $W(\Q)$, the complex structure thus defined does not depend on the choice the Levi decomposition. Note that this complex structure is $P(\R)W(\C)$-invariant.

Since $h\colon\cX^\natural\rightarrow\cD^\natural$ is a local diffeomorphism, we have a complex structure on $\cX^\natural$, invariant under $P(\R)W(\C)$. With this complex structure, the underlying map of every coarse morphism $(P_1,\cX^\natural_1,h_1)\rightarrow(P_2,\cX^\natural_2,h_2)$ on the underlying spaces $\cX^\natural_1\rightarrow\cX^\natural_2$ is holomorphic. Note that $\pi^\sharp\colon\cX^\natural\rightarrow\varphi(\cD)$ is holomorphic and $P(\R)W(\C)$-equivariant by Proposition~\ref{PropositionComplexStructureOfUniformizingSpaceHodgeStructureViewpoint}(1).

Next we define a left $P(\R)V(\C)$-homogeneous space $\cD^\natural_{P/U}\subset\hom(\S_\C,(P/U)_\C)$ to be the image of $\cD^\natural$ under the obvious morphism $\hom(\S_\C,P_\C)\rightarrow\hom(\S_\C,(P/U)_\C)$. Then $(P/U,\cD^\natural_{P/U},\mathrm{id})$ is an enlarged mixed Shimura datum. Apply the geometric comparison discussed above to this enlarged mixed Shimura datum we get
\[
\pi^\sharp_{P/U}\colon(P/U,\cD^\natural_{P/U})\rightarrow(P/U,\varphi(\cD)_{P/U}).
\]

We give a better description of $\pi^\sharp_{P/U}$. The homologous description for $\pi^\sharp$ is given in $\mathsection$\ref{SubsectionNotationEMSD}. By reason of level (condition~\eqref{definition of U} of Definition~\ref{DefinitionEnlargedMixedShimuraDatum}), we know that $\lie V$ is commutative, and hence the exponential morphism $\lie V \rightarrow V$ is an isomorphism as algebraic groups. So $V$ is commutative, and hence gives rise to a $\Q$-representation of $G$ and furthermore a variation of rational Hodge structures over $\varphi(\cD)_G$ of weight $-1$ and type $\{(-1,0),(0,-1)\}$. Let $\cV:=V(\C)\times\varphi(\cD)_G$ be the corresponding holomorphic bundle over $\varphi(\cD)_G$ and let $\cF^0\cV$ be the holomorphic subbundle of $\cV$ which induces the Hodge filtration in each fiber.

Back to $(P/U,\cD^\natural_{P/U})$. The space $\cD^\natural_{P/U}$ is a $(P/U)(\R)V(\C)$-orbit, and there exists a $(P/U)(\R)$-orbit $\cD_{P/U} \subset \cD^\natural_{P/U}$ given by Proposition~\ref{PropositionReplaceByASmallerOrbit}. Now any Levi-decomposition $P = W\rtimes G$ gives rise to a pair of identifications $\cD^\natural_{P/U} = \cV$ and $\cD_{P/U} = V(\R) \times \varphi(\cD)_G$. Under this identification we have that
\begin{equation}\label{EquationKugaTypeUniformizingSpace}
\pi^\sharp_{P/U}\colon\cD^\natural_{P/U}\rightarrow\varphi(\cD)_{P/U}\text{ is the natural projection }\cV\rightarrow\cV/\cF^0\cV.
\end{equation}
Hence $\pi^\sharp_{P/U}$ is holomorphic.


The following discussion will be useful to understand the morphisms of enlarged mixed Shimura varieties. For simplicity, let us temporarily assume that all enlarged mixed Shimura data in the rest of this subsection satisfy $U = 1$. Let $f \colon (P_1,\cD_1^\natural) \rightarrow (P_2,\cD^\natural_2)$ be a morphism of enlarged mixed Shimura data. Let $f_G$ be the induced morphism $(G_1,\varphi_1(\cD_1)_{G_1}) \rightarrow (G_2,\varphi_2(\cD_2)_{G_2})$. Denote by $\cV_i = V_i(\C) \times \varphi_i(\cD_i)_{G_i}$. Fix two Levi-decompositions $P_1 = V_1 \rtimes G_1$ and $P_2 = V_2 \rtimes G_2$, which give rise to identifications $\cD_i^\natural = \cV_i$ and $\cD_i = V_i(\R) \times \varphi_i(\cD_i)_{G_i}$ for $i = 1,2$. Consider the group $P_*$ of $P_2$ generated by $V_2$ and $f(0\rtimes G_1)$. Then we have two Levi-decompositions of $P_*$, induced by $f(P_1) < P_*$ and by $P_2 = V_2 \rtimes G_2$ respectively. They differ from the conjugation of an element $v_2$ in $\cR_u(P_*)(\Q) = V_2(\Q)$. Thus we have
\begin{equation}\label{EqMorEMSVUnivSpace}
f(\cV_1) = (v_2 + f(V_1)(\C)) \times f_G(\varphi_1(\cD_1)_{G_1}) \subset \cV_2\text{ and } f(\cD_1) = (v_2 + f(V_1)(\R)) \times f_G(\varphi_1(\cD_1)_{G_1}) \subset \cD_2.
\end{equation}

\subsection{Structure of the underlying group}\label{SubsectionStructureOfUnderlyingGroup}
One direct corollary of (the proof of) Lemma~\ref{LemmaCategoricalComparisonBetweenEMSDandMSD} is that the underlying group remains the same for an enlarged mixed Shimura datum and its associated mixed Shimura datum. Therefore for any enlarged mixed Shimura datum, to study its underlying group it suffices to look at its associated mixed Shimura datum. Now by Pink \cite[2.15]{PinkThesis}, we have the following result:

Given an enlarged mixed Shimura datum $(P,\cX^\natural)$, we can associate to $P$ a 4-tuple $(G,V,U,\Psi)$ which is defined as follows:
\begin{itemize}
\item $G:=P/\cR_u(P)$ is the reductive part of $P$;
\item $U$ is the normal subgroup of $P$ as in Definition~\ref{DefinitionEnlargedMixedShimuraDatum} (hence the weight $-2$ part) and $V:=\cR_u(P)/U$ (hence the weight $-1$ part). Both of them are vector groups with an action of $G$ induced by conjugation in $P$ (which factors through $G$ for reason of weight);
\item The commutator on $W:=\cR_u(P)$ induces a $G$-equivariant alternating form $\Psi\colon V\times V\rightarrow U$ by reason of weight. Moreover, $\Psi$ is given by a polynomial with coefficients in $\Q$.
\end{itemize}

On the other hand, $P$ is uniquely determined up to isomorphism by this 4-tuple:
\begin{itemize}
\item let $W$ be the central extension of $V$ by $U$ defined by $\Psi$. More concretely, $W=U\times V$ as a $\Q$-variety and the group law on $W$ is $(u,v)(u^\prime,v^\prime)=(u+u^\prime+\frac{1}{2}\Psi(v,v^\prime),v+v^\prime)$;
\item define the action of $G$ on $W$ by $g((u,v)):=(gu,gv)$;
\item define $P:=W\rtimes G$.
\end{itemize}

\subsection{Notations for EMSD}\label{SubsectionNotationEMSD}
We fix some notations for the rest of the paper. Let $(P,\cX^\natural,h)$ be an enlarged mixed Shimura datum, which for simplicity is often denoted by $(P,\cX^\natural)$.
\begin{center}
  \begin{tabular}{@{} rl @{}}
    \hline
    notation & meaning \\ 
    \hline
    $(P,\cX)$ & the associated mixed Shimura datum under Lemma~\ref{LemmaCategoricalComparisonBetweenEMSDandMSD} \\ 
    $(P/U,\cX_{P/U})$ & the quotient mixed Shimura datum $(P,\cX)/U$ \\
    $(P/U,\cX^\natural_{P/U})$ & $(P,\cX^\natural)/U$=the enlarged mixed Shimura datum associated with $(P/U,\cX_{P/U})$ \\
    $G$, $W$, $U$, $V$ & as in $\mathsection$\ref{SubsectionStructureOfUnderlyingGroup} \\ 
    $\cD^\natural$ & $h(\cX^\natural)\subset\hom(\S_\C,P_\C)$ \\
    $\cD$ & the subset of $\cD^\natural$ defined in Proposition~\ref{PropositionReplaceByASmallerOrbit} \\
    $\varphi$ & the map $\cD^\natural\rightarrow\{\text{rational mixed Hodge structures on }M\}$ for a faithful \\
    & representation $M$ of $P$ as in Proposition~\ref{PropositionComplexStructureOfUniformizingSpaceHodgeStructureViewpoint} \\
    $(P/U,\varphi(\cD)_{P/U})$ & the mixed Shimura datum $(P,\varphi(\cD))/U$ \\
    $(G,\varphi(\cD)_G)$ & the mixed Shimura datum $(P,\varphi(\cD))/W$ \\
    \hline
  \end{tabular}
\end{center}
The notations we introduced in $\mathsection$\ref{SubsectionStructureOfUnderlyingSpace} are summarized in the diagram
\begin{equation}\label{EquationCartesianDiagramUniformizingSpaceProjectToMSD}
\xymatrix{
    (P,\cX^\natural) \ar[r]^{\pi^\sharp} \ar[d]_{\pi_{P/U}^\natural} \bigpullback \ar@/_5pc/[dd]_{\pi^\natural} & (P,\varphi(\cD))  \ar[d]^{\pi_{P/U}} \ar@/^5pc/[dd]^{\pi}\\
    (P/U,\cX^\natural_{P/U}) \ar[r]^{\pi_{P/U}^\sharp} \ar[d]_{\pi_G^\natural} & (P/U,\varphi(\cD)_{P/U}) \ar[d]^{\pi_G} \\
   (G,\varphi(\cD)_G) \ar[r]^{\mathrm{id}} & (G,\varphi(\cD)_G)
  }
\end{equation}
We explain why the first square is a pullback. Pink proved in \cite[2.18, 2.19]{PinkThesis} that $\varphi(\cD)$ can be identified with a holomorphic complex vector bundle over $\varphi(\cD)_G$. In fact let $L\cW:=\lie W_\C\times\varphi(\cD)_G$ and let $\cF^0L\cW$ be the holomorphic subbundle of $L\cW$ whose fiber over $x_G$ is $F^0_{x_G}\lie W_\C$, then $\varphi(\cD)=L\cW/\cF^0L\cW$. But $F^0_{x_G}V_\C=F^0_{x_G}\lie W_\C\cong\exp(F^0_{x_G}\lie W_\C)$ as algebraic varieties and $\exp(F^0_{x_G}\lie W_\C)$ is a subgroup of $W(\C)$ by reason of weight and type for every $x_G\in\varphi(\cD)_G$. So $\cF^0L\cW$ is a subgroup of $W(\C)\times\varphi(\cD)_G$ over $\varphi(\cD)_G$ and we denote by $\cF^0\cV:=\cF^0L\cW$. Therefore
\begin{equation}\label{EquationUniformizingSpaceProjectToMSD}
\cF^0\cV\subset\cD^\natural=W(\C)\times\varphi(\cD)_G\xrightarrow{\pi^\sharp}\varphi(\cD).
\end{equation}
So the first square of the diagram above is a pullback by \eqref{EquationKugaTypeUniformizingSpace}.

Before ending this subsection, we remark that it is often enough to consider the case where $\cX^\natural=\cD^\natural$ (hence $\cX=\cD$) because of the following lemma:
\begin{lemma}\label{LemmaFromXnaturalToDnatural}
For any enlarged mixed Shimura datum $(P,\cX^\natural,h)$, the canonical morphism
\[
(P,\cX^\natural)\rightarrow(P,\cX^\natural)/P^{\der}\times(P,h(\cX^\natural))
\]
is injective for both the underlying group and the underlying space. We call such a morphism a \textbf{Shimura embedding}.
\begin{proof} This follows from the categorical comparison Lemma~\ref{LemmaCategoricalComparisonBetweenEMSDandMSD} and the parallel result for mixed Shimura data \cite[2.11]{PinkThesis}.
\end{proof}
\end{lemma}

\subsection{Enlarged mixed Shimura data of Siegel type}\label{SubsectionEMSDOfSiegelType}
In this subsection we focus on some important examples. Let $g\in\N_{>0}$. Let $V_{2g}$ be a $\Q$-vector space of dimension $2g$ and let
\begin{equation}\label{EquationAlternatingForm}
\Psi\colon V_{2g}\times V_{2g}\rightarrow U_{2g}:=\G_{a,\Q}
\end{equation}
be a non-degenerate alternating form. Define
\[
\GSp_{2g}:=\{g\in\GL(V_{2g})|\Psi(gv,gv^\prime)=\nu(g)\Psi(v,v^\prime)\text{ for some }\nu(g)\in\G_m\},
\]
and $\cH_g\subset\hom(\S,\GSp_{2g,\R})$ to be the set of all homomorphisms inducing a pure Hodge structure of type $\{(-1,0),(0,-1)\}$ on $V_{2g}$ and for which $\Psi$ or $-\Psi$ defines a polarization. It is well-known that $\GSp_{2g}$ is a reductive group and $\GSp_{2g}(\R)$ acts transitively on $\cH_g$.

\subsubsection{}\label{SubsubsectionEMSDUniversalVectorExtension}
We start with the enlarged mixed Shimura datum $(P_{2g,\mathrm{a}},\cX^\natural_{2g,\mathrm{a}})$,\footnote{The index ``a'' is short for ``abelian''.} which corresponds to the universal vector extension of the universal abelian variety (over a fine moduli space). See $\mathsection$\ref{SubSubsectionEMSVUniversalVectorExtensionOfUniversalAbelianVariety} for further details.

Let
\[
\cX_{2g,\mathrm{a}}:=V_{2g}(\R)\rtimes\cH_g\subset\hom(\S,V_{2g,\R}\rtimes\GSp_{2g,\R})
\]
denote the conjugacy class under $V_{2g}(\R)\rtimes\GSp_{2g}(\R)$ generated by any element of $\cH_g$. Then let
\[
\cX^\natural_{2g,\mathrm{a}}:=V_{2g}(\C)\rtimes\cH_g\subset\hom(\S_\C,V_{2g,\C}\rtimes\GSp_{2g,\C})
\]
be the conjugacy class under $V_{2g}(\C)\rtimes\GSp_{2g}(\R)$ generated by any element of $\cX_{2g,\mathrm{a}}$. The notation $V_{2g}(R)\rtimes\cH_g$ ($R=\R,~\C$) is justified by the natural bijection
\begin{equation}\label{EquationUniformizingSpaceUniversalAbelianVariety}
V_{2g}(R)\times\cH_g\xrightarrow{\sim}V_{2g}(R)\rtimes\cH_g,~~(v^\prime,x)\mapsto\mathrm{int}(v^\prime)\circ x.
\end{equation}
Under this bijection the action of $(v,t)\in V_{2g}(\C)\rtimes\GSp_{2g}(\R)$ on $V_{2g}(\C)\rtimes\cH_g$ is given by $(v,t)\cdot(v^\prime,x):=(v+tv^\prime,tx)$.

Denote by $(P_{2g,\mathrm{a}},\cX^\natural_{2g,\mathrm{a}}):=(V_{2g}\rtimes\GSp_{2g},V_{2g}(\C)\rtimes\cH_g)$. Then $(P_{2g,\mathrm{a}},\cX_{2g,\mathrm{a}})$ is a mixed Shimura datum (\cite[2.25]{PinkThesis}) and $(P_{2g,\mathrm{a}},\cX^\natural_{2g,\mathrm{a}})$ is the enlarged mixed Shimura datum associated with $(P_{2g,\mathrm{a}},\cX_{2g,\mathrm{a}})$.

Now let us turn to the geometric comparison considered in $\mathsection$\ref{SubsectionStructureOfUnderlyingSpace}. Let $M$ and $\varphi$ be as in Proposition~\ref{PropositionComplexStructureOfUniformizingSpaceHodgeStructureViewpoint}. In fact we can take $M$ to be a $\Q$-vector space of dimension $2g+1$ because
\[
P_{2g,\mathrm{a}}=\left(\begin{array}{cc} \GSp_{2g} & V_{2g} \\ 0 & 1 \end{array}\right)<\GL_{2g+1}
\]
and then $\varphi$ factors through $\{\text{rational mixed Hodge structures on }M\text{ of type }\{(-1,0),(0,-1),(0,0)\}$. We shall work with the mixed Shimura datum $(P_{2g,\mathrm{a}},\varphi(\cX_{2g,\mathrm{a}}))$ instead of $(P_{2g,\mathrm{a}},\cX_{2g,\mathrm{a}})$. Now denote by $\cV_{2g}:=V_{2g}(\C)\times\cH_g$ and let $\cF^0\cV_{2g}$ be the holomorphic subbundle of $\cV_{2g}$ such that its fiber over $x\in\cH_g$ is $(V_{2g})_x^{0,-1}$. Then by \eqref{EquationKugaTypeUniformizingSpace}, we have that on the underlying spaces
\begin{equation}\label{EquationPiSharpForUniversalAbelianVariety}
(P_{2g,\mathrm{a}},\cX_{2g,\mathrm{a}}^\natural)\xrightarrow{\pi^\sharp}(P_{2g,\mathrm{a}},\varphi(\cX_{2g,\mathrm{a}}))\text{ is the natural projection }\cV_{2g}\rightarrow\cV_{2g}/\cF^0\cV_{2g}.
\end{equation}

\subsubsection{}\label{SubsubsectionEMSDPullOfTheCanonicalAmpleGmTorsorToUniversalVectorExtension}
Next we define the enlarged mixed Shimura datum $(P_{2g},\cX^\natural_{2g})$, which corresponds to the pullback of the canonical relatively ample $\G_m$-torsor over the universal abelian variety (over a fine moduli space) to its universal vector extension. See Example~\ref{ExampleEMSV}(1) for further details.

Let $W_{2g}$ be the central extension of $V_{2g}$ by $U_{2g}$ defined by $\Psi$ and let $P_{2g}$ be the group associated to the 4-tuple $(\GSp_{2g},V_{2g},U_{2g},\Psi)$ as in $\mathsection$\ref{SubsectionStructureOfUnderlyingGroup}. Then $P_{2g}=W_{2g}\rtimes\GSp_{2g}$. The action of $\GSp_{2g}$ on $W_{2g}$ induces a Hodge structure of type $\{(-1,0),(0,-1),(-1,-1)\}$ on $\lie W_{2g,\C}$. Let
\[
\cX_{2g}\subset\hom(\S_\C,P_{2g,\C})
\]
be the conjugacy class under $P_{2g}(\R)U_{2g}(\C)$ generated by any element of $\cH_g$. Then $(P_{2g},\cX_{2g})$ defines a mixed Shimura datum (see \cite[2.25]{PinkThesis}). Pink \cite[10.10]{PinkThesis} proved that the mixed Shimura datum $(P_{2g},\cX_{2g})$ corresponds to the canonical ample $\G_m$-torsor over the universal abelian variety (over a fine moduli space). Finally we define $(P_{2g},\cX^\natural_{2g})$ to be the enlarged mixed Shimura datum associated with $(P_{2g},\cX_{2g})$ (see Lemma~\ref{LemmaCategoricalComparisonBetweenEMSDandMSD}).

\subsubsection{}\label{SubsubsectionEMSDPullOfThePoincareBiextensionToUniversalVectorExtension}
Now define the enlarged mixed Shimura datum $(P_{2g,\mathrm{b}},\cX^\natural_{2g,\mathrm{b}})$\footnote{The index ``b'' is short for ``biextension''.} which corresponds to the universal vectorial bi-extension defined by Coleman \cite{ColemanThe-universal-v}. See Example~\ref{ExampleEMSV}(2) for further details.

Define the group $P_{2g,\mathrm{b}}$ to be the unipotent extension of $P_{2g,\mathrm{a}}=V_{2g}\rtimes\GSp_{2g}$ by $V_{2g}\bigoplus U_{2g}$ for the action $P_{2g,\mathrm{a}}\curvearrowright V_{2g}\bigoplus U_{2g}$ defined by $(v,g)(v^\prime,u):=(gv^\prime,gu+\Psi(v,v^\prime))$. Let $(P_{2g,\mathrm{b}},\cX^\natural_{2g,\mathrm{b}})$ be the unipotent extension of $(P_{2g,\mathrm{a}},\cX^\natural_{2g,\mathrm{a}})$ by $V_{2g} \oplus U_{2g}$ given by Proposition~\ref{PropositionUnipotentExtensionOfEMSD}. Then its associated mixed Shimura datum $(P_{2g,\mathrm{b}},\cX_{2g,\mathrm{b}})$ is the unipotent extension of $(P_{2g,\mathrm{a}},\cX_{2g,\mathrm{a}})$ by $V_{2g}\bigoplus U_{2g}$ in the sense of Pink \cite[2.17]{PinkThesis}. Moreover by \cite[Remark~2.13]{PinkA-Combination-o}, the mixed Shimura datum $(P_{2g,\mathrm{b}},\cX_{2g,\mathrm{b}})$ corresponds to the Poincar\'{e} biextension over the product of the universal abelian variety with its dual (over a fine moduli space). We denote by $W_{2g,\mathrm{b}}:=\cR_u(P_{2g,\mathrm{b}})$.

\subsection{Reduction Theorem for enlarged Shimura data}
\begin{defn}\label{DefinitionEMSDWithGenericMTGroup}
An enlarged mixed Shimura datum $(P,\cX^\natural,h)$ (resp. mixed Shimura datum $(P,\cX,h)$) is said to \textbf{have generic Mumford-Tate group} if $P$ possesses no proper normal subgroup $Q$ such that for one (equivalently all) $x\in\cX^\natural$ (resp. $x\in\cX$), $h_x$ factors through $Q_\C\subset P_\C$. We shall denote this case by $P=\MT(\cX^\natural)$ (resp. $P=\MT(\cX)$).
\end{defn}
\begin{rmk}
\begin{enumerate}
\item Pink \cite[2.13]{PinkThesis} used the term ``irreducible'' for this definition in the case of mixed Shimura data.
\item Let $M$ and $\varphi$ be as in Proposition~\ref{PropositionComplexStructureOfUniformizingSpaceHodgeStructureViewpoint}. Then an enlarged mixed Shimura datum $(P,\cX^\natural,h)$ has generic Mumford-Tate group if and only if  $P=\MT(\varphi(h(\cX^\natural)))$ (resp. a mixed Shimura datum $(P,\cX,h)$ has generic Mumford-Tate group if and only if  $P=\MT(\varphi(h(\cX)))$). But $\varphi(h(\cX^\natural))=\varphi(h(\cX))$ by Proposition~\ref{PropositionReplaceByASmallerOrbit}, so an enlarged mixed Shimura datum has generic Mumford-Tate group if and only if  its associated mixed Shimura datum has generic Mumford-Tate group.
\end{enumerate}
\end{rmk}

\begin{prop}\label{PropositionEMSDWithGenericMTGroupAlwaysExists}
Let $(P,\cX^\natural)$ be an enlarged mixed Shimura datum, then
\begin{enumerate}
\item there exists an enlarged mixed Shimura datum $(Q,\cY^\natural)$ having generic Mumford-Tate group such that $(Q,\cY^\natural)\hookrightarrow(P,\cX^\natural)$ and a connected component $\cY^{\natural+}$ of $\cY^\natural$ is sent isomorphically to a connected component $\cX^{\natural+}$ of $\cX^\natural$ under this embedding;
\item if $(P,\cX^\natural)$ has generic Mumford-Tate group, then $P$ acts on $U$ via a scalar. In particular, any subgroup of $U$ is normal in $P$.
\end{enumerate}
\begin{proof}
This is true for mixed Shimura data by Pink \cite[2.13, 2.14]{PinkThesis}. Then it suffices to apply Lemma~\ref{LemmaCategoricalComparisonBetweenEMSDandMSD}.
\end{proof}
\end{prop}

We close this section by the following Reduction Theorem for enlarged mixed Shimura data.
\begin{thm}[Reduction Theorem]\label{ReductionTheorem}
Let $(P,\cX^\natural)$ be an enlarged mixed Shimura datum with generic Mumford-Tate group. Assume $\dim V=2g>0$ (otherwise $V=0$ and $(P,\cX^\natural)$ is a mixed Shimura datum), and let $r = \dim U + 1$. Then there exist pure Shimura data $(T,\cY)$ and $(G_0,\cD_0)$, where $T$ is a torus and $\cD_0\subset\hom(\S,G_{0,\R})$ such that the following hold:
there exists an enlarged mixed Shimura datum $(P_*,\cX^\natural_*)$ with morphisms
\begin{align*}
(P_*,\cX^\natural_*)\rightarrow(P,\cX^\natural)\text{ which is a }\left((P_0,\cX_0)\rightarrow(\G_m,\cH_0)\right)\text{-torsor}\\
\text{and }(P_*,\cX^\natural_*)\stackrel{\lambda}{\hookrightarrow}(T,\cY)\times(G_0,\cD_0)\times\prod_{i=1}^r(P_{2g},\cX^\natural_{2g})
\end{align*}
such that $\lambda|_V\colon V\cong V_{2g}\rightarrow\bigoplus_{i=1}^r V_{2g}$ is the diagonal map, $\lambda|_{U_*}\colon U_*\cong\bigoplus_{i=1}^r U_{2g}$ and $G\xrightarrow{\bar{\lambda}}T\times G_0\times\prod_{i=1}^r\GSp_{2g}\rightarrow\GSp_{2g}$ is non-trivial for each projection.
\begin{proof} The first part is \cite[2.26(a)]{PinkThesis}. For the second part, let $(P,\cX)$ be mixed Shimura datum associated with $(P,\cX^\natural)$. Then by \cite[2.26(b)]{PinkThesis} and \cite[Lemma~2.12]{GaoTowards-the-And}, there exists a mixed Shimura datum $(P_*,\cX_*)$ with Shimura morphisms
\[
(P_*,\cX_*)\twoheadrightarrow(P,\cX)\text{ and }(P_*,\cX_*)\hookrightarrow(T,\cY)\times(G_0,\cD_0)\times\prod_{i=1}^r(P_{2g},\cX_{2g})
\]
with the desired properties. Now it suffices to take $(P_*,\cX^\natural_*)$ to be the enlarged mixed Shimura datum associated with $(P_*,\cX_*)$.
\end{proof}
\end{thm}

\section{Enlarged mixed Shimura varieties}\label{SectionEnlargedMixedShimuraVarieties}
\subsection{Basic definition and complex space structure}
\begin{defn}\label{DefinitionEnlargedMixedShimuraVarieties}
\begin{enumerate}
\item
Let $(P,\cX^\natural)$ be an enlarged mixed Shimura datum and let $K$ be an open compact subgroup of $P(\A_f)$. Define the corresponding \textbf{enlarged mixed Shimura variety} as
\[
\mathrm{EM}_K(P,\cX^\natural)_{\C}:=P(\Q)\backslash\cX^\natural\times P(\A_f)/K,
\]
where $P(\Q)$ acts diagonally on both factors from the left. This enlarged mixed Shimura variety is called \textbf{of Kuga type} if $(P,\cX^\natural)$ is of Kuga type. We shall see that $\mathrm{EM}_K(P,\cX^\natural)_{\C}$ is a complex analytic variety (Proposition~\ref{PropositionEMSVConnectedDecomposition}).
\item Under the notation of (1) and $\mathsection$\ref{SubsectionNotationEMSD}, we say that the mixed Shimura variety
\[
\mathrm{M}_K(P,\cX)(\C)=P(\Q)\backslash\cX\times P(\A_f)/K
\]
is the \textbf{mixed Shimura variety associated with} $\mathrm{EM}_K(P,\cX^\natural)_{\C}$ and $\mathrm{EM}_K(P,\cX^\natural)_{\C}$ is the \textbf{enlarged mixed Shimura variety associated with} $\mathrm{M}_K(P,\cX)(\C)$. Here $\mathrm{M}_K(P,\cX)(\C)$ means the $\C$-points of the algebraic variety $\mathrm{M}_K(P,\cX)$ over the number field $E(P,\cX)$ on which it is defined (this is proven by Pink).
\end{enumerate}
\end{defn}

\begin{rmk}\label{RemarkEMSVConnectedDecomposition}
Let $\cX^{\natural+}$ be a connected component of $\cX^\natural$, and thus $\cX^{\natural+}$ is a $P(\R)^+W(\C)$-homogeneous space. Then as for mixed Shimura varieties, we have
\[
\mathrm{EM}_K(P,\cX^\natural)_{\C}=\bigcup_{[p]\in P(\Q)_+\backslash P(\A_f)/K}\Gamma(p)\backslash\cX^{\natural+}
\]
where $\Gamma(p):=P(\Q)_+\cap pKp^{-1}$.
\end{rmk}

\begin{prop}\label{PropositionEMSVConnectedDecomposition}
The action of $\Gamma(p)$ on $\cX^{\natural+}$ factors through a quotient $\bar{\Gamma(p)}$ which acts properly discontinuously on $\cX^{\natural+}$. There is a canonical structure of a normal complex space on $\mathrm{EM}_K(P,\cX^\natural)_{\C}$, whose singularities are quotient singularities by finite groups. Moreover $\mathrm{EM}_K(P,\cX^\natural)_{\C}$ is a complex manifold if $K$ is neat (\textit{cf.} \cite[0.6]{PinkThesis} for neatness).
\begin{proof}(Compare with \cite[3.3]{PinkThesis}) The image of $\Gamma(p)$ in $(P/P^{\der}Z(P))(\Q)$ is finite by Definition~\ref{DefinitionEnlargedMixedShimuraDatum}(5), so it suffices to prove the assertation for $\Gamma:=\Gamma(p)\cap(P^{\der}Z(P))(\Q)$. Note that $\Gamma=\Gamma(p)$ when $K$ is neat. Now $Z(P)(\R)$ acts trivially on $\cX^{\natural+}$, so we may replace $(P,\cX^\natural)$ by $(P,\cX^\natural)/Z(P)$. Then $\Gamma$ is an arithmetic subgroup of $P^{\der}(\Q)$, torsion free if $K$ is neat, and therefore is a discrete subgroup of $P^{\der}(\R)$. So $\Gamma$ acts properly discontinuously on $P_\infty:=P^{\der}(\R)^+W(\C)$. But $\cX^{\natural+}\cong P_\infty/\mathrm{Stab}_{P_\infty}(x)$ for any $x\in\cX^{\natural+}$ and $\mathrm{Stab}_{P_\infty}(x)$ is compact by Lemma~\ref{LemmaCentralizerOfh} and Definition~\ref{DefinitionEnlargedMixedShimuraDatum}(4). So $\Gamma$ acts properly discontinuously on $\cX^{\natural+}$. Note that $\Gamma\cap\mathrm{Stab}_{P_\infty}(x)$ is finite, and hence trivial if $K$ is neat. The complex structure of $\Gamma(p)\backslash\cX^{\natural+}$ comes from the $P(\R)W(\C)$-invariant complex structure on $\cX^\natural$, and $\Gamma(p)\backslash\cX^{\natural+}$ is a complex manifold when $K$ is neat.
\end{proof}
\end{prop}

\subsection{Algebraic structure and canonical model}
The goal of this section is to prove:
\begin{thm}\label{TheoremEMSVCanonicalModel}
Every enlarged mixed Shimura variety $\mathrm{EM}_K(P,\cX^\natural)_{\C}$ admits a structure of algebraic variety with the following properties:
\begin{enumerate}
\item It has a model $\mathrm{EM}_K(P,\cX^\natural)$ over $E(P,\cX):=$the reflex field of $(P,\cX)$\footnote{See \cite[11.1]{PinkThesis} for definition.}, called the \textbf{canonical model}, such that the map induced by the geometric comparison \eqref{EquationEnlargedMSDToMSD}
\[
[\pi^\sharp]\colon\mathrm{EM}_K(P,\cX^\natural)_{\C}\rightarrow\mathrm{M}_K(P,\varphi(\cD))(\C)
\]
descends to a map over the field $E(P,\cX)=E(P,\varphi(\cD))$
\[
\mathrm{EM}_K(P,\cX^\natural)\rightarrow\mathrm{M}_K(P,\varphi(\cD))
\]
where $\mathrm{M}_K(P,\varphi(\cD))$ is the canonical model of $\mathrm{M}_K(P,\varphi(\cD))_{\C}$.
\item Let $f\colon(P_1,\cX^\natural_1)\rightarrow(P_2,\cX^\natural_2)$ be a morphism of enlarged mixed Shimura data, and let $K_1\subset P_1(\A_f)$ and $K_2\subset P_2(\A_f)$ be open compact subgroups such that $f(K_1)\subset K_2$. Then the canonical map induced by $f$
\[
[f]\colon\mathrm{EM}_{K_1}(P_1,\cX^\natural_1)_{\C}\rightarrow\mathrm{EM}_{K_2}(P_2,\cX^\natural_2)_{\C},\quad[(x,p)]\mapsto[(f(x),f(p))]
\]
is algebraic.
\end{enumerate}
\end{thm}
In this subsection we only give the construction of the structure of algebraic variety (with property (1)). We will prove property (2) in the next subsection.

An important tool to prove this theorem is the geometric comparison considered in $\mathsection$\ref{SubsectionStructureOfUnderlyingSpace}.

\subsubsection{Universal vector extension of the universal abelian variety}\label{SubSubsectionEMSVUniversalVectorExtensionOfUniversalAbelianVariety}
Let $(P_{2g,\mathrm{a}},\cX^\natural_{2g,\mathrm{a}})$ and $(P_{2g,\mathrm{a}},\cX_{2g,\mathrm{a}})$ as in $\mathsection$\ref{SubsubsectionEMSDUniversalVectorExtension}. Let $N\ge  3$ be an integer. Define
\[
K_{\GSp}(N):=\{h\in\GSp_{2g}(\Z)|h\equiv1\bmod N\}\text{ and }K_{2g,\mathrm{a}}(N):=V_{2g}(\Z)\rtimes K_{\GSp}(N).
\]
Consider $\mathrm{EM}_{K_{2g,\mathrm{a}}(N)}(P_{2g,\mathrm{a}},\cX^\natural_{2g,\mathrm{a}})_{\C}$ and the geometric comparison \eqref{EquationPiSharpForUniversalAbelianVariety}. Denote by $\cA_g(N):=\Sh_{K_{\GSp}(N)}(\GSp_{2g},\cH_g)$ and $\mathfrak{A}_g(N):=M_{K_{2g,\mathrm{a}}(N)}(P_{2g,\mathrm{a}},\varphi(\cX_{2g,\mathrm{a}}))$.

Pink proved \cite[10.10]{PinkThesis} that $\mathrm{M}_{K_{2g,\mathrm{a}}(N)}(P_{2g,\mathrm{a}},\varphi(\cX_{2g,\mathrm{a}}))\cong\mathfrak{A}_g(N)$ is the universal abelian variety over the fine moduli space $\Sh_{K_{\GSp}(N)}(\GSp_{2g},\varphi(\cH_g))\cong\cA_g(N)$ and therefore is an algebraic variety. To better illustrate the moduli interpretation of $\mathrm{EM}_{K_{2g,\mathrm{a}}(N)}(P_{2g,\mathrm{a}},\cX^\natural_{2g,\mathrm{a}})_{\C}$, we consider its connected components given by Remark~\ref{RemarkEMSVConnectedDecomposition}.

Denote by $\cV_{2g}^+:=V_{2g}(\C)\times\cH_g^+$. Let $\Gamma_{\GSp}(N):=\Sp_{2g}(\Z)\cap K_{\GSp}(N)$ and
\[
\Gamma_{2g,\mathrm{a}}(N):=P_{2g,\mathrm{a}}^{\der}(\Z)\cap K_{2g,\mathrm{a}}(N)=V_{2g}(\Z)\rtimes\Gamma_{\GSp}(N). 
\]
Denote by $\mathfrak{A}_g(N)^+$ (resp. $\cA_g(N)^+$) the connected component $\Gamma_{2g,\mathrm{a}}(N)\backslash\varphi(\cX_{2g,\mathrm{a}}^+)$ (resp. $\Gamma_{\GSp}(N)\backslash\varphi(\cH_g^+)$) of $\mathfrak{A}_g(N)$ (resp. $\cA_g(N)$). From the geometric comparison \eqref{EquationPiSharpForUniversalAbelianVariety} we have
\begin{align*}
& 0\rightarrow\cF^0\cV_{2g}^+\rightarrow\cX_{2g,\mathrm{a}}^{\natural+}=\cV_{2g}^+\rightarrow\varphi(\cX_{2g,\mathrm{a}}^+)=\cV_{2g}^+/\cF^0\cV_{2g}^+\rightarrow 0\text{ as vector bundles over }\cH_g^+ \\
\Rightarrow & 0\rightarrow\Gamma_{\GSp}(N)\backslash\cF^0\cV_{2g}^+\rightarrow\Gamma_{\GSp}(N)\backslash\cV_{2g}^+\rightarrow\Gamma_{\GSp}(N)\backslash\varphi(\cX_{2g,\mathrm{a}}^+)\rightarrow 0 \text{ as vector bundles\footnotemark over }\cA_g(N)^+ \\
\Rightarrow & 0\rightarrow\Gamma_{\GSp}(N)\backslash\cF^0\cV_{2g}^+\rightarrow V_{2g}(\Z)\backslash\left(\Gamma_{\GSp}(N)\backslash\cV_{2g}^+\right)\rightarrow V_{2g}(\Z)\backslash\left(\Gamma_{\GSp}(N)\backslash\varphi(\cX_{2g,\mathrm{a}}^+)\right)\rightarrow 0, \\
&\quad\quad\quad\quad\quad\text{\textit{i.e.}  }0\rightarrow\Gamma_{\GSp}(N)\backslash\cF^0\cV_{2g}^+\rightarrow\Gamma_{2g,\mathrm{a}}(N)\backslash\cV_{2g}^+\rightarrow\mathfrak{A}_g(N)^+\rightarrow 0 \text{ as Lie groups over }\cA_g(N)^+ \\
\Rightarrow & 0\rightarrow\omega_{\mathfrak{A}_g(N)^{+\vee}/\cA_g(N)^+}\rightarrow\Gamma_{2g,\mathrm{a}}(N)\backslash\cV_{2g}^+\rightarrow\mathfrak{A}_g(N)^+\rightarrow 0\text{ as Lie groups over }\cA_g(N)^+
\end{align*}
\footnotetext{In fact these vector bundles are automorphic bundles over $\cA_g(N)^+$. This is because the variation of pure Hodge structures $(\cV_{2g}^+,\cF^0)$ over $\cH_g^+$ extends naturally to $\cH_g^\vee$: the fibers over $\cH_g^+$ define polarizable Hodge structures of weight $-1$, whereas a general fiber over $\cH_g^\vee$ defines Hodge structure of weight $-1$ not necessarily polarizable.}
Here the last implication is the canonical isomorphism $\omega_{\mathfrak{A}_g(N)^{+\vee}/\cA_g(N)^+}\cong\Gamma_{\GSp}(N)\backslash\cF^0\cV_{2g}^+$ implied by the definition of $\cF^0\cV_{2g}^+$, and $\mathfrak{A}_g(N)^{\natural+}:=\Gamma(N)\backslash\cV_{2g}^+$ is the universal vector extension of the universal abelian variety $\mathfrak{A}_g(N)^+$ over $\cA_g(N)^+$. In particular $\mathfrak{A}_g(N)^{\natural+}$ is an algebraic variety and the maps in the last line above are algebraic.

The above discussion holds for any $pK_{2g,\mathrm{a}}(N)p^{-1}\cap P_{2g,\mathrm{a}}(\Q)_+$. So by Remark~\ref{RemarkEMSVConnectedDecomposition} $\mathfrak{A}_g(N)^\natural:=\mathrm{EM}_{K_{2g,\mathrm{a}}(N)}(P_{2g,\mathrm{a}},\cX^\natural_{2g,\mathrm{a}})_{\C}$ is the universal vector extension of the universal abelian variety $\mathfrak{A}_g(N)$ over the fine moduli space $\cA_g(N)$, and in particular is an algebraic variety. But $\mathfrak{A}_g(N)$ has a canonical model over $\Q$, so $\mathfrak{A}_g(N)^\natural$ has a model over $\Q$ and the projection $\mathfrak{A}_g(N)^\natural\rightarrow\mathfrak{A}_g(N)$ is defined over $\Q$.

\subsubsection{General Case}\label{SubsubsectionEMSVAlgebraicityOverNFGeneralCase}
It suffices to consider the case $\dim V=2g>0$ since otherwise $\mathrm{EM}_K(P,\cX^\natural)_{\C}$ is a mixed Shimura variety and the result is known by Pink \cite{PinkThesis}. By Proposition~\ref{PropositionEMSDWithGenericMTGroupAlwaysExists}(1) and Remark~\ref{RemarkEMSVConnectedDecomposition}, we may assume that $(P,\cX^\natural)$ has generic Mumford-Tate group. By Lemma~\ref{LemmaFromXnaturalToDnatural} we may assume that $\cX^\natural\subset\hom(\S_\C,P_{\C})$. Denote by $K_{P/U}:=\pi_{P/U}(K)<P/U$, $K_W:=K\cap W(\A_f)$ and $K_V:=K_W/(K\cap U(\A_f))$. Up to replacing $K$ be a subgroup of finite index we may assume $K=K_W\rtimes K_G$ and $K_V\subset K_{2g,\mathrm{a}}(N)$ for some $N\ge  4$ even.

We start with $\mathrm{EM}_{K_{P/U}}(P/U,\cX^\natural_{P/U})_{\C}$. Use the notation of $\mathsection$\ref{SubsectionNotationEMSD} (note that we are in the case $\cX^\natural=\cD^\natural$). Apply Theorem~\ref{ReductionTheorem} to $(P/U,\cX^\natural_{P/U})$ and we get an inclusion
\[
(P/U,\cX^\natural_{P/U})\hookrightarrow(T,\cY)\times(G_0,\cD_0)\times(P_{2g,\mathrm{a}},\cX^\natural_{2g,\mathrm{a}})
\]
under which $V \cong V_{2g}$. If we by abuse of notation denote by $\mathrm{Sh}_{K_G}(G,\varphi(\cX)_G)$ its image in $\Sh_{K_T}(T,\cY)(\C)\times\Sh_{K_{G_0}}(G_0,\cD_0)(\C)\times\mathcal{A}_g(N)$, then we have that $\mathrm{EM}_{K_{P/U}}(P/U,\cX^\natural_{P/U})_{\C}$ is the restriction of
\[
\Sh_{K_T}(T,\cY)(\C)\times\Sh_{K_{G_0}}(G_0,\cD_0)(\C)\times\mathfrak{A}_g(N)^\natural \rightarrow \Sh_{K_T}(T,\cY)(\C)\times\Sh_{K_{G_0}}(G_0,\cD_0)(\C)\times \cA_g(N)
\]
to $\mathrm{Sh}_{K_G}(G,\varphi(\cX)_G)$ because $V \cong V_{2g}$. In particular we have the following Cartesian diagram in the category of complex varieties, where all objects and morphisms are algebraic except those concerning $\mathrm{EM}_{K_{P/U}}(P/U,\cX^\natural_{P/U})_{\C}$:
\[
\xymatrix{
\mathrm{EM}_{K_{P/U}}(P/U,\cX^\natural_{P/U})_{\C} \ar[r] \ar[d]_{[\pi_{P/U}^\sharp]} \bigpullback& \Sh_{K_T}(T,\cY)(\C)\times\Sh_{K_{G_0}}(G_0,\cD_0)(\C)\times\mathfrak{A}_g(N)^\natural \ar[d] \\
\mathrm{M}_{K_{P/U}}(P/U,\varphi(\cX)_{P/U})(\C) \ar[r] & \Sh_{K_T}(T,\cY)(\C)\times\Sh_{K_{G_0}}(G_0,\cD_0)(\C)\times\mathfrak{A}_g(N) \\
}.
\]
But $\mathfrak{A}_g(N)^\natural$ is the universal vector extension of the universal abelian scheme $\mathfrak{A}_g(N)$ over $\cA_g(N)$, so $\mathrm{EM}_{K_{P/U}}(P/U,\cX^\natural_{P/U})_{\C}$ is the universal vector extension of the abelian scheme $\mathrm{M}_{K_{P/U}}(P/U,\varphi(\cX)_{P/U})(\C)\rightarrow\Sh_{K_G}(G,\varphi(\cX)_G)(\C)$. So Theorem~\ref{TheoremEMSVCanonicalModel} holds for $\mathrm{EM}_{K_{P/U}}(P/U,\cX^\natural_{P/U})_{\C}$.

Now we turn to $\mathrm{EM}_K(P,\cX^\natural)(\C)$. The Cartesian diagram in \eqref{EquationCartesianDiagramUniformizingSpaceProjectToMSD} induces a Cartesian diagram in the category of complex varieties, where all objects and morphism are algebraic over $E(P,\cX)$ except those concerning $\mathrm{EM}_K(P,\cX^\natural)_{\C}$:
\begin{equation}\label{EquationPullBackDiagramEMSV}
\xymatrix{
\mathrm{EM}_K(P,\cX^\natural)_{\C} \ar[r] \ar[d] \bigpullback & \mathrm{EM}_{K_{P/U}}(P/U,\cX^\natural_{P/U})(\C)   \ar[d] \\
\mathrm{M}_K(P,\varphi(\cX))(\C) \ar[r] & \mathrm{M}_{K_{P/U}}(P/U,\varphi(\cX)_{P/U})(\C)
}
\end{equation}
But $\mathrm{M}_K(P,\varphi(\cX))\rightarrow  \mathrm{M}_{K_{P/U}}(P/U,\varphi(\cX)_{P/U})$ is a $T$-torsor for some algebraic group $T$ over $E(P,\cX)$ by Pink \cite[3.12, 3.18]{PinkThesis}. So Theorem~\ref{TheoremEMSVCanonicalModel} holds for $\mathrm{EM}_K(P,\cX^\natural)_{\C}$.

\begin{eg}\label{ExampleEMSV}
\begin{enumerate}
\item  Let $(P_{2g},\cX^\natural_{2g})$ be as in $\mathsection$\ref{SubsubsectionEMSDPullOfTheCanonicalAmpleGmTorsorToUniversalVectorExtension}. Keep the notation of $\mathsection$\ref{SubSubsectionEMSVUniversalVectorExtensionOfUniversalAbelianVariety}. Let $N\ge  4$ be an even integer. Define $K_{2g}(N):=W_{2g}(\Z)\rtimes K_{\GSp}(N)$. Denote by $\mathfrak{L}_g(N):=\mathrm{M}_{K_{2g}(N)}(P_{2g},\cX_{2g})$. Pink \cite[10.10]{PinkThesis} proved that $\mathfrak{L}_g(N)$ is (the total space of) the canonical symmetric relatively ample $\G_m$-torsor over the universal abelian variety $\mathfrak{A}_g(N)\rightarrow\cA_g(N)$. The proof of Theorem~\ref{TheoremEMSVCanonicalModel} shows that $\mathrm{EM}_K(P_{2g},\cX_{2g})$, which we denote by $\mathfrak{L}_g^\natural(N)$, is defined over $\Q$ and is the pullback of the $\G_m$-torsor $\mathfrak{L}_g(N)$ over the universal abelian variety $\mathfrak{A}_g(N)\rightarrow\cA_g(N)$ to the universal vector extension $\mathfrak{A}_g(N)^{\natural}$ of $\mathfrak{A}_g(N)$.
\item  Let $(P_{2g,\mathrm{b}},\cX_{2g,\mathrm{b}})$ be as in $\mathsection$\ref{SubsubsectionEMSDPullOfThePoincareBiextensionToUniversalVectorExtension}. Keep the notation of $\mathsection$\ref{SubSubsectionEMSVUniversalVectorExtensionOfUniversalAbelianVariety}. Let $N\ge  4$ be an even integer. Define $
K_{2g,\mathrm{b}}(N):=W_{2g,\mathrm{b}}(\Z)\rtimes K_{\GSp}(N)$. Pink proved in \cite[Remark~2.13]{PinkA-Combination-o} that the mixed Shimura variety $\mathfrak{P}_g(N):=\mathrm{M}_{K_{2g,\mathrm{b}}(N)}(P_{2g,\mathrm{b}},\cX_{2g,\mathrm{b}})$ is (the total space of) the Poincar\'{e} biextension over $\mathfrak{A}_g(N)\times_{\cA_g(N)}\mathfrak{A}_g(N)^\vee$. The proof of Theorem~\ref{TheoremEMSVCanonicalModel} shows that $\mathrm{EM}_{K_{2g,\mathrm{b}}(N)}(P_{2g,\mathrm{b}},\cX^\natural_{2g,\mathrm{b}})$, which we denote by $\mathfrak{P}_g^\natural(N)$, is defined over $\Q$ and is the pullback the Poincar\'{e} biextension $\mathfrak{P}_g(N)$ to $\mathfrak{A}_g(N)^\natural\times_{\cA_g(N)}\mathfrak{A}_g(N)^{\vee\natural}$. This is the ``universal vectorial bi-extension'' considered by Coleman in \cite{ColemanThe-universal-v}.
\end{enumerate}
\end{eg}

\subsection{Morphisms between enlarged mixed Shimura varieties}
\begin{prop}\label{PropositionMorphismsBetweenEMSV}
Let $f\colon(P_1,\cX^\natural_1)\rightarrow(P_2,\cX^\natural_2)$ be a morphism of enlarged mixed Shimura data, and let $K_1\subset P_1(\A_f)$ and $K_2\subset P_2(\A_f)$ be open compact subgroups such that $f(K_1)\subset K_2$. Then we have:
\begin{enumerate}
\item The canonical map induced by $f$
\[
[f]\colon\mathrm{EM}_{K_1}(P_1,\cX^\natural_1)_{\C}\rightarrow\mathrm{EM}_{K_2}(P_2,\cX^\natural_2)_{\C},\quad[(x,p)]\mapsto[(f(x),f(p))]
\]
is algebraic, closed and descends to $E(P_1,\cX_1)$.
\item If furthermore $f$ is injective both on the underlying groups and the underlying spaces such that $K_1=f^{-1}(K_2)$, then $[f]$ is finite.
\item If $f$ is injective both on the underlying groups and the underlying spaces, then for every $K_1$ there exists a $K_2$ such that $[f]$ is a closed embedding.
\end{enumerate}
\begin{proof} This proposition can be proven using the geometric interpretations in $\mathsection$\ref{SubsubsectionEMSVAlgebraicityOverNFGeneralCase}. By Proposition~\ref{PropositionEMSDWithGenericMTGroupAlwaysExists}(1) and Remark~\ref{RemarkEMSVConnectedDecomposition}, we may assume that $(P_i,\cX^\natural_i)$ has generic Mumford-Tate group for $i=1,2$. By Lemma~\ref{LemmaFromXnaturalToDnatural} we may assume that $\cX^\natural_i\subset\hom(\S_\C,P_{i,\C})$ for $i=1,2$. By Pink \cite[3.8,~9.24]{PinkThesis} the proposition is true for
\[
[\bar{f}] \colon \mathrm{M}_{K_1}(P_1,\varphi_1(\cX_1))\rightarrow\mathrm{M}_{K_2}(P_2,\varphi_2(\cX_2))
\]
and
\[
[\bar{f}_{P/U}] \colon \mathrm{M}_{K_{1,P/U}}(P_1/U_1,\varphi_1(\cX_1)_{P/U})\rightarrow\mathrm{M}_{K_{2,P/U}}(P_2/U_2,\varphi_2(\cX_2)_{P/U})
\]
Hence by \eqref{EquationPullBackDiagramEMSV} it suffices to prove the proposition for
\[
[f_{P/U}]\colon\mathrm{EM}_{K_{1,P/U}}(P_1/U_1,\cX^\natural_{1,P/U})_{\C}\rightarrow\mathrm{EM}_{K_{2,P/U}}(P_2/U_2,\cX^\natural_{2,P/U})_{\C}.
\]

By \eqref{EqMorEMSVUnivSpace} $\im([f_{P/U}])$ is a group subscheme of $\mathrm{EM}_{K_{2,P/U}}(P_2/U_2,\cX^\natural_{2,P/U})_{\C}|_{\im([f_G])} \rightarrow \im([f_G])$ translated by a torsion section. Since $\mathrm{EM}_{K_{i,P/U}}(P_i/U_i,\cX^\natural_{i,P/U})$ is the universal vector extension of the abelian scheme $\mathrm{M}_{K_{i,P/U}}(P_i/U_i,\varphi(\cX)_{i,P/U}) \rightarrow \Sh_{K_{i,G}}(G_i,\varphi(\cX_i)_{G_i})$, the morphism
\[
\mathrm{EM}_{K_{2,P/U}}(P_2/U_2,\cX^\natural_{2,P/U})_{\C}|_{\im([f_G])} \rightarrow \mathrm{M}_{K_{2,P/U}}(P_2/U_2,\varphi_2(\cX_2)_{P/U})_{\C}|_{\im([f_G])}
\]
gives a bijection on the torsion sections, we conclude that: $[f_{P/U}]$ is the extension of $[\bar{f}_{P/U}]$ by the natural morphism $\omega_1 \rightarrow \omega_2$ induced by $[\bar{f}_{P/U}]$, with $\omega_i:=\omega_{\mathrm{M}_{K_{i,P/U}}(P_i/U_i,\varphi(\cX)_{i,P/U})^\vee / \Sh_{K_{i,G}}(G_i,\varphi(\cX_i)_{G_i})}$. Hence we are done.
\end{proof}
\end{prop}

\section{Bi-algebraic setting for enlarged mixed Shimura varieties}\label{SectionBiAlgebraicSettingForEnlargedMixedShimuraVarieties}

From this section, we will only consider connected enlarged mixed Shimura data and connected enlarged mixed Shimura varieties.

\begin{defn}
\begin{enumerate}
\item A \textbf{connected enlarged mixed Shimura datum} is a triple $(P,\cX^{\natural+},h)$ such that $(P,\cX^{\natural+},h)\subset(P,\cX^\natural,h)$ for some enlarged mixed Shimura datum $(P,\cX^\natural,h)$ with $h$ injective and that $\cX^{\natural+}$ is a left $P(\R)^+W(\C)$-homogeneous space. It is said to \textbf{have generic Mumford-Tate group} if $(P,\cX^\natural,h)$ has generic Mumford-Tate group.
\item A \textbf{connected enlarged mixed Shimura variety} $S^\natural$ associated with the connected enlarged mixed Shimura datum $(P,\cX^{\natural+},h)$ is a quotient $\Gamma\backslash\cX^{\natural+}$ for some congruence subgroup $\Gamma$ of $P(\Q)_+:=P(\Q)\cap P(\R)_+$, where $P(\R)_+$ is the stabilizer in $P(\R)$ of $h(\cX^{\natural+})\subset\hom(\S_\C,P_\C)$.
\end{enumerate}
\end{defn}

\begin{rmk}\label{RemarkConnectedEMSDV}
\begin{enumerate}
\item For any enlarged mixed Shimura datum $(P,\cX^\natural,h)$, let $\cX^{\natural+}$ be a connected component of $\cX^\natural$. Then $(P,\cX^{\natural+},h)$ is a connected enlarged mixed Shimura datum. Conversely every connected enlarged mixed Shimura datum arises in this way.
\item By Remark~\ref{RemarkEMSVConnectedDecomposition}, connected enlarged mixed Shimura varieties are precisely connected components of enlarged mixed Shimura varieties. So by Theorem~\ref{TheoremEMSVCanonicalModel} and Proposition~\ref{PropositionMorphismsBetweenEMSV}, any connected enlarged mixed Shimura variety is algebraic and has a model over $\bar{\Q}$ and any Shimura morphism between connected enlarged mixed Shimura varieties is algebraic and descends to $\bar{\Q}$.
\item In practice we often consider the case $Z(P)=1$. This can be achieved by replacing $(P,\cX^{\natural+})$ by $(P,\cX^{\natural+})/Z(P)$. Under this hypothesis every neat congruence subgroup is contained in $P^{\der}(\Q)$. Conversely fix a Levi decomposition $P=W\rtimes G$, then $P^\der=W\rtimes G^\der$. Hence any congruence subgroup $\Gamma<P^\der(\Q)$ is Zariski dense in $P^\der$ by Definition~\ref{DefinitionEnlargedMixedShimuraDatum}(4) and \cite[Theorem~4.10]{AlgebraicGroupBible}.
\end{enumerate}
\end{rmk}

We are interested in the uniformization $\unif^\natural\colon\cX^{\natural+}\rightarrow S^\natural=\Gamma\backslash\cX^{\natural+}$. Use the notation of $\mathsection$\ref{SubsectionStructureOfUnderlyingSpace} (note that $\cD^{\natural+}\cong\cX^{\natural+}$ and $\cD^+\cong\cX^+$ and so we do not distinguish $\cX$ and $\cD$). Let $\Gamma_{P/U}:=\pi_{P/U}(\Gamma)$ and let $\Gamma_G:=\pi(\Gamma)$. We use indices to distinguish different uniformizations:
\begin{align*}
\cX^{\natural+}_{P/U}\xrightarrow{\unif^\natural_{P/U}}S^\natural_{P/U}:=\Gamma_{P/U}\backslash\cX^{\natural+}_{P/U} & &
\varphi(\cX^+)\xrightarrow{\unif}S:=\Gamma\backslash\varphi(\cX^+) \\
\varphi(\cX^+)_{P/U}\xrightarrow{\unif_{P/U}}S_{P/U}:=\Gamma_{P/U}\backslash\varphi(\cX^+)_{P/U} & & \varphi(\cX^+)_G\xrightarrow{\unif_G}S_G:=\Gamma_G\backslash\varphi(\cX^+)_G
\end{align*}
\begin{notation}\label{NotationSubsetsOfEMSD}
For any subset $\tilde{Y}^\natural$ of $\cX^{\natural+}$, the diagram in $\mathsection$\ref{SubsectionNotationEMSD} gives rise to
\[
\xymatrix{
    \cX^{\natural+} \ar[r]^{\pi^\sharp} \ar[d]_{\pi_{P/U}^\natural} \bigpullback \ar@/_5pc/[dd]^{\pi^\natural} & \varphi(\cX^+)  \ar[d]^{\pi_{P/U}} \ar@/^5pc/[dd]_{\pi} & & \tilde{Y}^\natural \ar@{|->}[r] \ar@{|->}[d] & \tilde{Y} \ar@{|->}[d] \\
    \cX^{\natural+}_{P/U} \ar[r]^{\pi_{P/U}^\sharp} \ar[d]_{\pi_G^\natural} & \varphi(\cX^+)_{P/U} \ar[d]^{\pi_G} & & \tilde{Y}^\natural_{P/U} \ar@{|->}[r] \ar@{|->}[d] & \tilde{Y}_{P/U} \ar@{|->}[d]\\
   \varphi(\cX^+)_G \ar[r]^{\mathrm{id}} & \varphi(\cX^+)_G & & \tilde{Y}_G \ar@{=}[r] & \tilde{Y}_G
  }
\]
For any subset $Y^\natural$ of $S^\natural$, the diagram in $\mathsection$\ref{SubsectionNotationEMSD} gives rise to
\[
\xymatrix{
    S^\natural \ar[r]^{[\pi^\sharp]} \ar[d]_{[\pi_{P/U}^\natural]} \bigpullback \ar@/_5pc/[dd]^{[\pi^\natural]} & S  \ar[d]^{[\pi_{P/U}]} \ar@/^5pc/[dd]_{[\pi]} & & Y^\natural \ar@{|->}[r] \ar@{|->}[d] & Y \ar@{|->}[d] \\
    S^\natural_{P/U} \ar[r]^{[\pi_{P/U}^\sharp]} \ar[d]_{[\pi_G^\natural]} & S_{P/U} \ar[d]^{[\pi_G]} & & Y^\natural_{P/U} \ar@{|->}[r] \ar@{|->}[d] & Y_{P/U} \ar@{|->}[d]\\
   S_G \ar[r]^{\mathrm{id}} & S_G & & Y_G \ar@{=}[r] & Y_G
  }
\]
We emphasize that the diagrams on the right are purely set theoretic.
\end{notation}

\subsection{Algebraic structure on $\cX^{\natural+}$ and bi-algebraic subsets}\label{SubsectionAlgebraicStructureOnXnatural}
Recall that the complex structure on $\cX^{\natural+}$ is given by $\cX^{\natural+}\cong W(\C)\times\varphi(\cX^+)_G$ and $\varphi(\cX^+)_G\hookrightarrow\cX_G^\vee=G(\C)/F_{x_G}G_\C$ (for any $x_G\in\cX_G$) is a complex algebraic variety and $\varphi(\cX^+)_G$ is a semi-algebraic subset of $\cX_G^\vee$, open in the usual topology.

\begin{defn}\label{DefinitionAlgebraicOnTheTop}
Let $\tilde{Y}^\natural$ be a complex analytic subset of $\cX^{\natural+}$, then
\begin{enumerate}
\item $\tilde{Y}^\natural$ is called an \textbf{irreducible algebraic subset} of $\cX^{\natural+}$ if $\tilde{Y}^\natural$ is a complex analytic irreducible component of the intersection of its Zariski closure in $W(\C)\times\cX_G^\vee$ with $\cX^{\natural+}$.
\item $\tilde{Y}^\natural$ is called \textbf{algebraic} if it is a finite union of irreducible algebraic subsets of $\cX^{\natural+}$.
\end{enumerate}
\end{defn}

The inclusion $\varphi(\cX^+)\cong\cX^+\hookrightarrow\cX^\vee$ defines an algebraic structure on $\varphi(\cX^+)$ (see \cite[$\mathsection$6]{GaoTowards-the-And}). All maps in the diagrams in Notation \ref{NotationSubsetsOfEMSD} are algebraic. To see this, it suffices to prove the algebraicity of $\pi^\sharp\colon\cX^{\natural+}\rightarrow\varphi(\cX^+)$, which fits into the commutative diagram
\[
\xymatrix{
\cX^{\natural+}=P(\R)^+W(\C)/\mathrm{Cent}(h_x)\cong W(\C)\times\varphi(\cX^+)_G \ar@{^{(}->}[r] \ar[d]^{\pi^\sharp} & P(\C)/F^0_{x_G}G_\C\cong W(\C)\times\cX^\vee_G \ar[d] \\
\varphi(\cX^+)=P(\R)^+W(\C)/\mathrm{Stab}_{P(\R)^+W(\C)}(\varphi(h_x)) \ar@{^{(}->}[r] & \cX^\vee=P(\C)/F^0_{x_G}P_\C
}
\]
and hence is algebraic.

Now we are in the following situation: under $\unif^\natural\colon\cX^{\natural+}\rightarrow S^\natural$, the two spaces $\cX^{\natural+}$ and $S^\natural$ are algebraic but the map $\unif^\natural$ is transcendental. We want to understand the bi-algebraic objects in this context.

\begin{defn}
\begin{enumerate}
\item A subset $\tilde{Y}^\natural$ is called \textbf{bi-algebraic} if $\tilde{Y}^\natural$ is an irreducible algebraic subset of $\cX^{\natural+}$ and $Y^\natural:=\unif^\natural(\tilde{Y}^\natural)$ is a closed algebraic subvariety of $S^\natural$.
\item A closed subvariety $Y^\natural$ of $S^\natural$ is said to be \textbf{bi-algebraic} if it is the image of some bi-algebraic subset of $\cX^{\natural+}$ under $\unif^\natural$.
\end{enumerate}
\end{defn}

We will give a more concrete description of be-algebraic subvarieties. In order to do this we introduce ``quasi-linear subvarieties''.

\subsection{Quasi-linear subvarieties and characterization of bi-algebraic subvarieties}
Let $\tilde{Y}^\natural$ be a subset of $\cX^{\natural+}$. We use Notation~\ref{NotationSubsetsOfEMSD}. Let $(Q,\cY^+)$ be the smallest connected mixed Shimura subdatum of $(P,\varphi(\cX^+))$ such that $\tilde{Y}_{P/U}\subset\cY^+_{Q/U_Q}$. Assume that $\tilde{Y}_{P/U}$ is weakly special, namely $\tilde{Y}_{P/U}=(N/U_N)(\R)^+\tilde{y}_{P/U}$ for some normal subgroup $N$ of $Q$ (with $U_N:=U\cap N$) and a point $\tilde{y}_{P/U}\in\tilde{Y}_{P/U}$. By \eqref{EquationKugaTypeUniformizingSpace} we have for $(P/U,\varphi(\cX^+)_{P/U})$
\begin{equation}\label{EquationBiAlgebraicModificationPartOriginalEquation}
0\rightarrow\cF^0\cV\rightarrow\cX^{\natural+}_{P/U}=\cV\xrightarrow{\pi^\sharp_{P/U}}\varphi(\cX^+)_{P/U}=\cV/\cF^0\cV\rightarrow 0.
\end{equation}

Let us introduce the following notation. For any $G_{Q,R}$-submodule $V'$ of $V_{R}$ (with $R = \Q,\R$), denote by $\cF^0\cV'$ the intersection $\cF^0\cV \cap (V'(\C) \times \varphi(\cY^+)_G)$. It is a holomorphic bundle over $\varphi(\cY^+)_G$.

Let $V^{(0)}$ be the largest $G_{Q,\R}$-submodule of $V_\R$ on which $G_N(\R)^+$ acts trivially. Then we have a unique inclusion
\begin{equation}\label{EquationDecompositionXNaturalIntoProductOfVectorExtension}
\tilde{Y}_{P/U}^{\mathrm{univ}}\times_{\tilde{Y}_G} \cF^0\cV^{(0)}|_{\tilde{Y}_G} \subset\cX^{\natural+}_{P/U}|_{\tilde{Y}_{P/U}}:=\pi_{P/U}^{\sharp-1}(\tilde{Y}_{P/U})
\end{equation}
where $\tilde{Y}_{P/U}^{\mathrm{univ}}$ is defined by (compatible with \eqref{EquationBiAlgebraicModificationPartOriginalEquation})
\[
0\rightarrow\cF^0\cV_N|_{\tilde{Y}_G}\rightarrow\tilde{Y}_{P/U}^{\mathrm{univ}}\rightarrow\tilde{Y}_{P/U}\rightarrow0.
\]

\begin{defn}\label{DefinitionQuasiLinear}
\begin{enumerate}
\item A subset $\tilde{Y}^\natural$ of $\cX^{\natural+}$ is called \textbf{quasi-linear} if $\tilde{Y}$ is weakly special and
\[
\tilde{Y}^\natural=\tilde{Y}\times_{\tilde{Y}_{P/U}}\left(\tilde{Y}_{P/U}^{\mathrm{univ}}\times_{\tilde{Y}_G}(\tilde{L}^\natural\times\tilde{Y}_G)\times_{\tilde{Y}_G} G_N(\R)^+\tilde{K}^\natural \right)
\]
where $\tilde{L}^\natural$ is an irreducible algebraic subvariety of some fiber of $\cF^0\cV^{(0)}|_{\tilde{Y}_G} \rightarrow\tilde{Y}_G$ (which is isomorphic to $\C^k$ for some $k$), and $\tilde{K}^\natural$ is an irreducible subvariety of some fiber of $\cF^0\cV_N|_{\tilde{Y}_G} \rightarrow \tilde{Y}_G$.
\item A subset $Y^\natural$ of $S^\natural$ is said to be \textbf{quasi-linear} if it is the image of some quasi-linear subset of $\cX^{\natural+}$ under $\unif^\natural$.
\end{enumerate}
\end{defn}

It is clear that we can take furthermore assume $\tilde{K}^\natural$ to satisfy the following property: $\tilde{K}^\natural \subset \cF^0\cV^\dagger|_{\tilde{Y}_G}$ for some $G_{Q,\R}$-submodule $V^\dagger$ of $V_{\R}$ with $V^\dagger \cap V^{(0)} = 0$. Then $G_N(\R)^+\tilde{K}^\natural \cap \cF^0\cV^{(0)}$ is contained in the zero section of $\cF^0\cV_N|_{\tilde{Y}_G} \rightarrow \tilde{Y}_G$.

Before moving on, let us point out the following fact. If $\tilde{K}^\natural$ is contained in the fiber of $\tilde{y}_G \in \tilde{Y}_G$, then
\begin{equation}\label{EqAutomorphicSubvarietyFiberStableCompact}
\left( G_N(\R)^+\tilde{K}^\natural \right)_{\tilde{y}_G} = \mathrm{Stab}_{G_N(\R)^+}(\tilde{y}_G)\tilde{K}^\natural,
\end{equation}
which may be larger than $\tilde{K}^\natural$.

Before moving on, let us point out that any quasi-linear subset $\tilde{Y}^\natural$ of $\cX^{\natural+}$ is algebraic. To see this, it suffices to show that $G_N(\R)^+\tilde{K}^\natural$ is algebraic. This is true since, by theory of automorphic vector bundles, $G_N(\R)^+\tilde{K}^\natural = G_N(\C)\tilde{K}^\natural \cap \cF^0\cV_N|_{\tilde{Y}_G}$.

We are now ready to give the characterization of bi-algebraic subvarieties.

\begin{thm}\label{TheoremCharacterizationOfBiAlgebraicSubvarieties}
A subset $\tilde{Y}^\natural$ of $\cX^{\natural+}$ (resp. a subset $Y^\natural$ of $S^\natural$) is bi-algebraic if and only if  it is quasi-linear.
\begin{proof} The implication $\Leftarrow$ follows from Lemma~\ref{LemmaCriterionOfBiAlgebraicity}. In this section, we prove that a bi-algebraic subset $\tilde{Y}^\natural$ of $\cX^{\natural+}$ is quasi-linear. Up to replacing $\Gamma$ by a finite subgroup we assume that $\Gamma=\Gamma_W\rtimes\Gamma_G$. Then $S^\natural_{P/U}$ is the universal vector extension of the abelian scheme $S_{P/U}$ over $S_G$. Use the notation of Notation~\ref{NotationSubsetsOfEMSD}.

Denote by $\Gamma_{Y^\natural}:=\im(\pi_1(Y^{\natural,\mathrm{sm}})\rightarrow\pi_1(S^\natural)=\Gamma)$ and by $N:=(\Gamma_{Y^\natural}^{\Zar})^{\circ}$.
Then for any point $\tilde{y}^\natural\in\tilde{Y}^\natural$, we have $\Gamma_{Y^\natural}\cdot\tilde{y}^\natural\subset\tilde{Y}^\natural$. So the bi-algebraicity of $Y^\natural$ implies
\begin{equation}\label{EquationTildeYNaturalContainsEveryNOrbitBiQL}
N(\R)^+W_N(\C)\tilde{y}^\natural\subset\tilde{Y}^{\natural}\text{ for any point }\tilde{y}^\natural\in\tilde{Y}^{\natural}.
\end{equation}

Since $\tilde{Y}^\natural$ is bi-algeraic, we have that $\tilde{Y}$ is bi-algebraic for $\varphi(\cX^+) \rightarrow S$. 
By the characterization of bi-algebraic subsets for connected mixed Shimura varieties (\cite[Corollary~8.3]{GaoTowards-the-And}), $\tilde{Y}$ is a weakly special subset of $\varphi(\cX^+)$, namely $\tilde{Y}=N(\R)^+U_N(\C)\tilde{y}$ for some $\tilde{y}\in\tilde{Y}$ and $N$ is a normal subgroup of $Q$, where $(Q,\cY^+)$ is the smallest connected mixed Shimura subdatum of $(P,\varphi(\cX^+))$ such that $\tilde{Y}\subset\cY^+$. Then $Y_{P/U}\rightarrow Y_G$ is an abelian scheme and $S^\natural_{P/U}|_{Y_{P/U}}:=[\pi^\sharp_{P/U}]^{-1}(Y_{P/U})$ is a vector extension of $Y_{P/U}$ over $Y_G$.

We claim
\[
\tilde{Y}^\natural=\tilde{Y}\times_{\tilde{Y}_{P/U}}\tilde{Y}^\natural_{P/U} \subset\varphi(\cX^+)\times_{\varphi(\cX^+)_{P/U}}\cX^{\natural+}_{P/U} = \cX^{\natural+}.
\]
This is not hard to see: For any $\tilde{y}^\natural_{P/U} \in \tilde{Y}^\natural_{P/U}$, the fiber $\pi^{\sharp-1}(\tilde{y}^\natural_{P/U}) \cap (\tilde{Y}\times_{\tilde{Y}_{P/U}}\tilde{Y}^\natural_{P/U})$ is a $U(\C)$-orbit by the form of $\tilde{Y}$. But by \eqref{EquationTildeYNaturalContainsEveryNOrbitBiQL}, the fiber $\pi^{\sharp-1}(\tilde{y}^\natural_{P/U}) \cap \tilde{Y}^\natural$ is also a $U(\C)$-orbit. Thus we are done because $\tilde{Y}^\natural \subset \tilde{Y}\times_{\tilde{Y}_{P/U}}\tilde{Y}^\natural_{P/U}$.

Now it remains to study $\tilde{Y}^\natural_{P/U}$ by definition of quasi-linear subsets. Use the notation in \eqref{EquationDecompositionXNaturalIntoProductOfVectorExtension}. Let $V^{\mathrm{nt}}:=V_\R/(V_{N,\R}\bigoplus V^{(0)})$. Consider the vector bundle
\[
\cF^0\cV^{\mathrm{nt}}|_{\tilde{Y}_G}=\cX^{\natural+}_{P/U}|_{\tilde{Y}_{P/U}}/(\tilde{Y}_{P/U}^{\mathrm{univ}}\times_{\tilde{Y}_G}\cF^0\cV^{(0)}|_{\tilde{Y}_G})
\]
and the automorphic quotient-bundle
\[
\bV^{\mathrm{nt}}|_{Y_G}:=\Gamma_{G_N}\backslash(\cF^0\cV^{\mathrm{nt}}|_{\tilde{Y}_G})
\]
of $\omega_{[\pi_G]^{-1}(Y_G)^\vee/Y_G}\big/\omega_{Y_{P/U}^\vee/Y_G}$. 

Let $\tilde{Y}_{P/U}^{\mathrm{univ}}$ be as in \eqref{EquationDecompositionXNaturalIntoProductOfVectorExtension}. We are left the prove:
\begin{equation}\label{EquationGrossDescriptionOfTildeXNatural2}
\tilde{Y}^\natural_{P/U} = \tilde{Y}_{P/U}^{\mathrm{univ}}\times_{\tilde{Y}_G} (\tilde{L}^\natural\times\tilde{Y}_G)\times_{\tilde{Y}_G}\tilde{Y}^{\natural,\mathrm{extr}}_{P/U}
\end{equation}
for some splitting $V_\R/V_{N,\R}=V^{(0)}\bigoplus V^{\mathrm{nt}}$ as $G_Q$-modules, where $\tilde{L}^\natural$ is an irreducible algebraic subvariety of some fiber of $\cF^0\cV^{(0)}|_{\tilde{Y}_G} \rightarrow\tilde{Y}_G$ (which is isomorphic to $\C^k$ for some $k$), and $\tilde{Y}_{P/U}^{\natural,\mathrm{extr}}\subset\cF^0\cV^{\mathrm{nt}}|_{\tilde{Y}_G}$ which surjects to $\tilde{Y}_G$ satisfying
\begin{equation}\label{EquationGrossDescriptionOfTildeXNatural3}
\tilde{Y}^{\natural,\mathrm{extr}}_{P/U}=G_N(\R)^+(\tilde{Y}_{P/U}^{\natural,\mathrm{extr}})_{\tilde{y}_{G,0}}
\end{equation}
for some fixed $\tilde{y}_{G,0} \in \tilde{Y}_G$.

Let us prove \eqref{EquationGrossDescriptionOfTildeXNatural2} and \eqref{EquationGrossDescriptionOfTildeXNatural3}. First note that $\tilde{Y}^\natural_{P/U} \subset \cX^{\natural+}_{P/U}|_{\tilde{Y}_{P/U}} = \pi_{P/U}^{\sharp-1}(\tilde{Y}_{P/U}) = \tilde{Y}_{P/U}^{\mathrm{univ}} \times_{\tilde{Y}_G} \cF^0\cV^{(0)}|_{\tilde{Y}_G}  \times_{\tilde{Y}_G} \cF^0\cV^{\mathrm{nt}}|_{\tilde{Y}_G}$. Since each fiber of the projection $p \colon \cX^{\natural+}_{P/U}|_{\tilde{Y}_{P/U}}  \rightarrow \cF^0\cV^{(0)}|_{\tilde{Y}_G} \times_{\tilde{Y}_G} \cF^0\cV^{\mathrm{nt}}|_{\tilde{Y}_G}$ is isomorphic to $V_N(\C)$, we have that each non-empty fiber of $p|_{\tilde{Y}^\natural_{P/U}}$ is the whole fiber of $p$ by \eqref{EquationTildeYNaturalContainsEveryNOrbitBiQL}. Thus we have
\[
\tilde{Y}^\natural_{P/U} = \tilde{Y}_{P/U}^{\mathrm{univ}} \times_{\tilde{Y}_G} \tilde{Y}_{P/U}^{\natural,\mathrm{s}}
\]
for some $\tilde{Y}_{P/U}^{\natural,\mathrm{s}}$ in $\cF^0\cV^{(0)}|_{\tilde{Y}_G}  \times_{\tilde{Y}_G} \cF^0\cV^{\mathrm{nt}}|_{\tilde{Y}_G}$. Next take the fiber $V_N(\C) + \tilde{L}^\natural + \tilde{K}^\natural$ of $\tilde{Y}_{P/U}^\natural \rightarrow \tilde{Y}_G$ over $\tilde{y}_{G} \in \tilde{Y}_G$, with $\tilde{L}^\natural \subset \cF^0\cV^{(0)}$ and $\tilde{K}^\natural \subset \cF^0\cV^{\mathrm{nt}}$. By \eqref{EquationTildeYNaturalContainsEveryNOrbitBiQL}, we then have $V_N(\C) + \tilde{L}^\natural + g\tilde{K}^\natural \subset (\tilde{Y}_{P/U}^\natural)_{g\tilde{y}_{G}}$ for any $g \in G_N(\R)^+$. By furthermore considering the fibers over $g\tilde{y}_{G}$ and over $\tilde{y}_G = g^{-1} g \tilde{y}_{G}$, we obtain both \eqref{EquationGrossDescriptionOfTildeXNatural2} and \eqref{EquationGrossDescriptionOfTildeXNatural3}. 
\end{proof}
\end{thm}

\subsection{Weakly special subvarieties}
We close this section by defining weakly special subvarieties of $S^\natural$ and explain its relation with quasi-linear subvarieties. Note that for mixed Shimura varieties they coincide (\cite[Corollary~8.3]{GaoTowards-the-And}).
\begin{defn}\label{DefinitionWeaklySpecialSubvariety}
\begin{enumerate}
\item A subset $\tilde{Y}^\natural$ of $\cX^{\natural+}$ is said to be \textbf{weakly special} if there exists a diagram in the category \underline{$\mathbf{\cE\cM\cS\cD}$}
\[
(R,\cZ^{\natural+})\xleftarrow{\varphi}(Q,\cY^{\natural+})\xrightarrow{i}(P,\cX^{\natural+})
\]
and a point $\tilde{z}^\natural\in\cZ^{\natural+}$ such that $\tilde{Y}^\natural$ is a connected component of $i(\varphi^{-1}(\tilde{z}^\natural))$.
\item A subvariety $Y^\natural$ of $S^\natural$ is said to be \textbf{weakly special} if it is the image of some weakly special subset of $\cX^{\natural+}$ under $\unif^{\natural}$.
\end{enumerate}
\end{defn}
The geometric meaning of weakly special subvarieties is very clear: a subvariety $Y^\natural$ is weakly special in $S^\natural$ if and only if:
\begin{itemize}
\item $Y^\natural=Y\times_{Y_{P/U}}Y_{P/U}^\natural$.
\item $Y$ is weakly special. In particular $Y_{P/U}/Y_G$ is an abelian scheme.
\item $Y_{P/U}^\natural=Y_{P/U}^{\mathrm{univ}}$ where $Y_{P/U}^{\mathrm{univ}}$ is the universal vector extension of $Y_{P/U}/Y_G$.
\end{itemize}
Now it is clear that every weakly special subvariety of $S^\natural$ is quasi-linear. On the other hand for any quasi-linear
\[
\tilde{Y}^\natural=\tilde{Y}\times_{\tilde{Y}_{P/U}}\left(\tilde{Y}_{P/U}^{\mathrm{univ}}\times_{\tilde{Y}_G}(\tilde{L}^\natural\times\tilde{Y}_G)\times_{\tilde{Y}_G}\cF^0\cV^\dagger|_{\tilde{Y}_G}\right),~Y^\natural=Y\times_{Y_{P/U}}\left(Y_{P/U}^{\mathrm{univ}}\times_{Y_G}L^\natural_{Y_G}\times_{Y_G}\bV^\dagger|_{Y_G}\right),
\]
we denote by $(\tilde{Y}^\natural)^{\mathrm{ws}}=\tilde{Y}\times_{\tilde{Y}_{P/U}}(\tilde{Y}_{P/U}^{\mathrm{univ}}\times_{\tilde{Y}_G}\underline{0})$, $(Y^\natural)^{\mathrm{ws}}:=Y\times_{Y_{P/U}}(Y_{P/U}^{\mathrm{univ}}\times_{Y_G}\underline{0})$ and by
\begin{equation}\label{EquationWeaklySpecialAndQuasiLinear}
pr^{\mathrm{ws}}_{\tilde{Y}^\natural}\colon \tilde{Y}^\natural\rightarrow (\tilde{Y}^\natural)^{\mathrm{ws}}\quad\text{ and }\quad[pr]^{\mathrm{ws}}_{Y^\natural}\colon Y^\natural\rightarrow (Y^\natural)^{\mathrm{ws}}
\end{equation}
the natural projections (here we identify $\tilde{Y}_{P/U}^{\mathrm{univ}}\times_{\tilde{Y}_G}\underline{0}$ with $\tilde{Y}_{P/U}^{\mathrm{univ}}$ and $Y_{P/U}^{\mathrm{univ}}\times_{Y_G}\underline{0}$ with $Y_{P/U}^{\mathrm{univ}}$).


\section{The logarithmic Ax theorem}\label{SectionTheAxlogarithmicTheorem}
We prove in this section a transcendence theorem. We start with a criterion of bi-algebraicity, which will also be used in the next section.

Let $S^\natural$ be a connected enlarged mixed Shimura variety associated with $(P,\cX^{\natural+})$ and let $\unif^\natural\colon\cX^{\natural+}\rightarrow S^\natural$ be the uniformization. Use Notation~\ref{NotationSubsetsOfEMSD}. Let $\tilde{Y}^\natural$ be a subset of $\cX^{\natural+}$ such that $\tilde{Y}$ is weakly special, and hence $\tilde{Y}=N(\R)^+U_N(\C)\tilde{y}$ for some subgroup $N$ of $P$ and some point $\tilde{y}\in\varphi(\cX^+)$. Let $\tilde{Y}_{P/U}^{\mathrm{univ}}$ be as in \eqref{EquationDecompositionXNaturalIntoProductOfVectorExtension}. 

\begin{lemma}\label{LemmaCriterionOfBiAlgebraicity}
Let $\tilde{Y}^\natural$ be an irreducible algebraic subset of $\cX^{\natural+}$ such that
\begin{enumerate}
\item $\tilde{Y}^\natural=\tilde{Y}\times_{\tilde{Y}_{P/U}}\tilde{Y}^\natural_{P/U}\subset\varphi(\cX^+)\times_{\varphi(\cX^+)_{P/U}}\tilde{\cX}^{\natural+}_{P/U}$;
\item $\tilde{Y}$ is weakly special (and we use the notations above);
\item $\tilde{Y}^{\mathrm{univ}}_{P/U}\subset\tilde{Y}^\natural_{P/U}$;
\item $\tilde{Y}^{\natural}_{P/U}=G_N(\R)^+(\tilde{Y}^{\natural}_{P/U})_{\tilde{y}_G}$ for one (and hence any) point $\tilde{y}_G\in\tilde{Y}_G$.
\end{enumerate}
Then $\tilde{Y}^\natural$ is bi-algebraic.
\begin{proof}
It suffices to prove that $\tilde{Y}^\natural_{P/U}$ is bi-algebraic by condition (1) and (2). Recall \eqref{EquationKugaTypeUniformizingSpace}
\begin{equation}\label{EquationExtensionForPoverUInBiAlgebraicAxlogarithmic}
0\rightarrow\cF^0\cV\rightarrow\cX^{\natural+}_{P/U}=\cV\xrightarrow{\pi^\sharp_{P/U}}\varphi(\cX^+)_{P/U}=\cV/\cF^0\cV\rightarrow 0.
\end{equation}
We have
\begin{equation}\label{EquationDecompositionOfSNaturalVersionX}
S^\natural_{P/U}|_{Y_{P/U}}=Y_{P/U}^{\mathrm{univ}}\times_{Y_G}\left(\omega_{[\pi_G]^{-1}(X_Y)^\vee/Y_G}\big/\omega_{Y_{P/U}^\vee/Y_G}\right),
\end{equation}
where $Y_{P/U}^{\mathrm{univ}}$ is the universal vector extension of the abelian scheme $Y_{P/U}/Y_G$. Let $\tilde{Y}_{P/U}^{\mathrm{univ}}$ be defined by (compatible with \eqref{EquationExtensionForPoverUInBiAlgebraicAxlogarithmic})
\[
0\rightarrow\cF^0\cV_N|_{\tilde{Y}_G}\rightarrow\tilde{Y}_{P/U}^{\mathrm{univ}}\rightarrow\tilde{Y}_{P/U}\rightarrow0.
\]
Then the uniformization of \eqref{EquationDecompositionOfSNaturalVersionX} is
\begin{equation}\label{EquationDecompositionOfSNaturalVersionTildeX}
\tilde{\cX}_{P/U}^{\natural+}|_{\tilde{Y}_{P/U}}(=\pi_{P/U}^{\sharp-1}(\tilde{Y}_{P/U}))=\tilde{Y}_{P/U}^{\mathrm{univ}}\times_{\tilde{Y}_G}\cF^0\cV^{\mathrm{s}}|_{\tilde{Y}_G}
\end{equation}
for some $G_Q$-submodule $V^{\mathrm{s}}$ of $V$ such that $V=V_N\bigoplus V^{\mathrm{s}}$.

Now condition (3) implies
\[
\tilde{Y}_{P/U}^\natural=\tilde{Y}_{P/U}^{\mathrm{univ}}\times_{\tilde{Y}_G}\tilde{Y}_{P/U}^{\natural,\mathrm{s}}
\]
and hence
\[
Y_{P/U}^\natural=Y_{P/U}^{\mathrm{univ}}\times_{Y_G}\unif^{\natural,\mathrm{s}}_{P/U}(\tilde{Y}_{P/U}^{\natural,\mathrm{s}}).
\]
Hence it suffices to prove that $\unif^{\natural,\mathrm{s}}_{P/U}(\tilde{Y}_{P/U}^{\natural,\mathrm{s}})$ is a closed algebraic subvariety of $\omega_{[\pi_G]^{-1}(Y_G)^\vee/Y_G}\big/\omega_{Y_{P/U}^\vee/Y_G}$.

Consider the diagram
\[
\xymatrix{
\cF^0\cV^{\mathrm{s}}|_{\tilde{Y}_G} \ar[r]^-{\unif^{\natural,\mathrm{s}}_{P/U}} \ar[d] & \omega_{[\pi_G]^{-1}(Y_G)^\vee/Y_G}\big/\omega_{Y_{P/U}^\vee/Y_G} \ar[d] \\
\tilde{Y}_G \ar[r]^{\unif_G} & Y_G
}
\]

The restriction of $\unif^{\natural,\mathrm{s}}_{P/U}$ to any fiber
\[
\left(\cF^0\cV^{\mathrm{s}}|_{\tilde{Y}_G}\right)_{\tilde{y}_G}\rightarrow\left(\omega_{[\pi_G]^{-1}(Y_G)^\vee/Y_G}\big/\omega_{Y_{P/U}^\vee/Y_G}\right)_{y_G},
\]
where $\tilde{y}_G\in\tilde{Y}_G$ and $y_G=\unif_G(\tilde{y}_G)$, is the identity map. So by \cite[Theorem~1.9]{KlinglerThe-Hyperbolic-}, there exists a definable fundamental domain $\mathfrak{F}$ of $\unif_G\colon\tilde{Y}_G\rightarrow Y_G$ such that
\[
\unif^{\natural,\mathrm{s}}_{P/U}|_{\cF^0\cV^{\mathrm{s}}|_{\mathfrak{F}}}\text{ is definable in the o-minimal structure }\R_{\text{an},\exp}.
\]
Therefore Condition (4) implies that $\unif^{\natural,\mathrm{s}}_{P/U}(\tilde{Y}_{P/U}^{\mathrm{s}})$ is closed in the usual topology (and hence a complex analytic subvariety of $\omega_{[\pi_G]^{-1}(Y_G)^\vee/Y_G}\big/\omega_{Y_{P/U}^\vee/Y_G}$) and is definable in the o-minimal structure $\R_{\text{an},\exp}$ (because $\unif^{\natural,\mathrm{s}}_{P/U}(\tilde{Y}_{P/U}^{\mathrm{s}})=\unif^{\natural,\mathrm{s}}_{P/U}(\tilde{Y}_{P/U}^{\mathrm{s}}\cap\cF^0\cV^{\mathrm{s}}|_{\mathfrak{F}})$ and $\tilde{Y}_{P/U}^{\mathrm{s}}$ is algebraic). Hence $Y_{P/U}^{\natural,\mathrm{s}}:=\unif^{\natural,\mathrm{s}}_{P/U}(\tilde{Y}_{P/U}^{\mathrm{s}})$ is a closed algebraic subvariety of $\omega_{[\pi_G]^{-1}(Y_G)^\vee/Y_G}\big/\omega_{Y_{P/U}^\vee/Y_G}$ by the Peterzil-Starchenko o-minimal GAGA theorem \cite[Corollary~4.5]{PeterzilComplex-analyti}.
\end{proof}
\end{lemma}

\begin{cor}\label{CorollaryCriterionOfBiAlgebraicity}
Let $\tilde{Y}^\natural$ be an irreducible algebraic subset of $\cX^{\natural+}$ such that $\tilde{Y}$ is weakly special (and hence we can use the notation above Lemma~\ref{LemmaCriterionOfBiAlgebraicity}). Assume $\tilde{Y}^\natural=N(\R)^+W_N(\C)(\tilde{Y}^\natural)_{\tilde{y}_G}$ for one (and hence any) point $\tilde{y}_G\in\tilde{Y}_G$. Then $\tilde{Y}^\natural$ is bi-algebraic.
\begin{proof} It is not hard to see that all the conditions of Lemma~\ref{LemmaCriterionOfBiAlgebraicity} are satisfied.
\end{proof}
\end{cor}

\begin{thm}[logarithmic Ax]\label{TheoremAxlogarithmic}
Let $S^\natural$ be a connected enlarged mixed Shimura variety associated with $(P,\cX^{\natural+})$ with $\unif^\natural\colon\cX^{\natural+}\rightarrow S^\natural$. Let $Z^\natural$ be an irreducible subvariety of $S^\natural$. Let $\tilde{Z}^\natural$ be a complex analytic irreducible component of $\unif^{\natural-1}(Z^\natural)$ and let $\tilde{Z}^{\natural,\mathrm{Zar}}$ be the Zariski closure of $\tilde{Z}^\natural$ in $\cX^{\natural+}$. Then $\tilde{Z}^{\natural,\mathrm{Zar}}$ is quasi-linear.

\begin{proof} Up to replacing $\Gamma$ by a finite subgroup we assume that $\Gamma=\Gamma_W\rtimes\Gamma_G$. Then $S^\natural_{P/U}$ is the universal vector extension of the abelian scheme $S_{P/U}$ over $S_G$. Use the notation of Notation~\ref{NotationSubsetsOfEMSD}.

Denote by $\Gamma_{Z^\natural}:=\im(\pi_1(Z^{\natural,\mathrm{sm}})\rightarrow\pi_1(S^\natural)=\Gamma)$ and by $N:=(\Gamma_{Z^\natural}^{\Zar})^{\circ}$. Then for any point $\tilde{y}^\natural\in\tilde{Z}^\natural$, we have $\Gamma_{Z^\natural}\cdot\tilde{y}^\natural\subset\tilde{Z}^\natural$. So
\begin{equation}\label{EquationTildeYNaturalContainsEveryNOrbit}
N(\R)^+W_N(\C)\tilde{z}^\natural\subset\tilde{Y}^{\natural,\mathrm{Zar}}\text{ for any point }\tilde{z}^\natural\in\tilde{Z}^{\natural,\mathrm{Zar}}.
\end{equation}

Denote for simplicity $\tilde{Y}^\natural:=\tilde{Z}^{\natural,\mathrm{Zar}}$ and use the notation of Notation~\ref{NotationSubsetsOfEMSD}. By logarithmic Ax for connected mixed Shimura varieties (\cite[Theorem~8.1]{GaoTowards-the-And}), $\tilde{Y}$ is a weakly special subset of $\varphi(\cX^+)$: $\tilde{Y}=N(\R)^+U_N(\C)\tilde{z}$ for some $\tilde{z}\in\tilde{Y}$ and $N$ is a normal subgroup of $Q$, where $(Q,\cY^+)$ is the smallest connected mixed Shimura subdatum of $(P,\varphi(\cX^+))$ such that $\tilde{Z}\subset\cY^+$. But we have $\tilde{Y}^\natural=N(\R)^+W_N(\C)(\tilde{Y}^\natural)_{\tilde{y}_G}$ for one (and hence any) point $\tilde{y}_G\in\tilde{Y}_G$. Now it suffices to apply Corollary~\ref{CorollaryCriterionOfBiAlgebraicity}.
\end{proof}
\end{thm}

\section{The Ax-Lindemann theorem}\label{SectionAxLindemann}
Let $S^\natural=\Gamma\backslash\cX^{\natural+}$ be a connected enlarged mixed Shimura variety associated with $(P,\cX^{\natural+})$ with $\unif^\natural\colon\cX^{\natural+}\rightarrow S^\natural$.

\subsection{Statement of Ax-Lindemann}
We start by giving four equivalent statements for the Ax-Lindemann(-Weierstra\ss)~theorem and explain their equivalence.

\begin{thm}\label{AxLindemann}
Let $Y^\natural$ be an irreducible subvariety of $S^\natural$ and let $\tilde{Z}^\natural$ be an irreducible algebraic subset of $\cX^{\natural+}$ contained in $\unif^{\natural-1}(Y^\natural)$, maximal for these properties. Then $\tilde{Z}^\natural$ is quasi-linear.
\end{thm}

The next statement we give shall be called the \textit{semi-algebraic form of Ax-Lindemann}. Recall that a connected semi-algebraic subset of $\cX^{\natural+}$ is called \textbf{irreducible} if its $\R$-Zariski closure in $W(\C)\times\cX_G^\vee$ is an irreducible real algebraic variety. Note that any connected semi-algebraic subset of $\cX^{\natural+}$ has only finitely many irreducible components.

\begin{thm}\label{AxLindemannSemiAlgebraicForm}
Let $Y^\natural$ be an irreducible subvariety of $S^\natural$ and let $\tilde{Z}^\natural$ be a connected irreducible semi-algebraic subset of $\cX^{\natural+}$ contained in $\unif^{\natural-1}(Y^\natural)$, maximal for these properties. Then $\tilde{Z}^\natural$ is complex analytic and each complex analytic irreducible component of $\tilde{Z}^\natural$ is quasi-linear.
\end{thm}

The equivalence of Theorem~\ref{AxLindemann} and Theorem~\ref{AxLindemannSemiAlgebraicForm} follows easily from \cite[Lemma~4.1]{PilaAbelianSurfaces}, which claims that maximal connected irreducible semi-algebraic subsets of $\cX^{\natural+}$ which are contained in $\unif^{\natural-1}(Y^\natural)$ are all algebraic in the sense of Definition~\ref{DefinitionAlgebraicOnTheTop}
(there is a typo in the proof of \cite[Lemma~4.1]{PilaAbelianSurfaces}: $\C^{2n}$ should be $\C^n$).

The next two forms of Ax-Lindemann have more ``equidistributional'' taste.
\begin{thm}\label{AxLindemannEquiDistribution}
Let $\tilde{Z}^\natural$ be an irreducible algebraic subset of $\cX^{\natural+}$. Then $\unif^\natural(\tilde{Z}^\natural)^{\mathrm{Zar}}$ is quasi-linear.
\end{thm}

\begin{thm}\label{AxLindemannEquiDistributionSemiAlgebraic}
Let $\tilde{Z}^\natural$ be a semi-algebraic subset of $\cX^{\natural+}$. Then every irreducible component of $\unif^\natural(\tilde{Z}^\natural)^{\mathrm{Zar}}$ is quasi-linear.
\end{thm}

Let us explain now why Theorem~\ref{AxLindemann} implies Theorem~\ref{AxLindemannEquiDistribution}. Let $\tilde{Z}^\natural$ be as in Theorem~\ref{AxLindemannEquiDistribution}. Let $Y^\natural:=\unif^\natural(\tilde{Z}^\natural)^{\mathrm{Zar}}$ and let $\tilde{W}^\natural$ be an irreducible algebraic subset of $\cX^{\natural+}$ which contains $\tilde{Z}^\natural$ and is contained in $\unif^{\natural-1}(Y^\natural)$, maximal for these properties. Such a $\tilde{W}^\natural$ exists by, for example, dimension reason. Then $Y^\natural=\unif^\natural(\tilde{W}^\natural)^{\Zar}$ and $\tilde{W}^\natural$ is an irreducible algebraic subset of $\cX^{\natural+}$ which is contained in $\unif^{\natural-1}(Y^\natural)$, maximal for these properties. Theorem~\ref{AxLindemann} then implies that $\tilde{W}^\natural$ is irreducible bi-algebraic. So $Y^\natural=\unif^\natural(\tilde{W}^\natural)^{\mathrm{Zar}}=\unif^\natural(\tilde{W}^\natural)$ is quasi-linear. Theorem~\ref{AxLindemannSemiAlgebraicForm} implies Theorem~\ref{AxLindemannEquiDistributionSemiAlgebraic} by a similar argument because any semi-algebraic subset of $\cX^{\natural+}$ has only finitely many connected irreducible components.

Let us explain now why Theorem~\ref{AxLindemannEquiDistribution} implies Theorem~\ref{AxLindemann}. Let $Y^\natural$ and $\tilde{Z}^\natural$ be as in Theorem~\ref{AxLindemann}. Then Theorem~\ref{AxLindemannEquiDistribution} shows that $\unif^\natural(\tilde{Z}^\natural)^{\mathrm{Zar}}$ is an irreducible bi-algebraic subvariety of $S$, which we shall call $Y_0^\natural$. It is clear that $Y_0^\natural$ is a subvariety of $Y^\natural$. Let $\tilde{Y}_0^\natural$ be the complex analytic irreducible component of $\unif^{\natural-1}(Y_0^\natural)$ containing $\tilde{Z}^\natural$. Then $\tilde{Y}_0^\natural$ is irreducible algebraic. But then the maximality assumption on $\tilde{Z}^\natural$ implies $\tilde{Z}^\natural=\tilde{Y}_0^\natural$. Hence $\tilde{Z}^\natural$ is quasi-linear. Theorem~\ref{AxLindemannEquiDistributionSemiAlgebraic} implies Theorem~\ref{AxLindemannSemiAlgebraicForm} by a similar argument.

\subsection{Outline of the proof}\label{SubsectionAxLindemannPart1} We prove the Ax-Lindemann theorem (in the form Theorem~\ref{AxLindemann}) in the rest of the section. The proof is a modification of the author's previous work on the Ax-Lindemann theorem for mixed Shimura varieties \cite[$\mathsection$9-11]{GaoTowards-the-And}. Consider the diagram
\[
\xymatrix{
(P,\cX^{\natural+}) \ar[r]_{\pi^\sharp} \ar[d]^{\unif^\natural} \ar@/^1pc/[rr]|{\pi^\natural} & (P,\varphi(\cX^+)) \ar[r]_{\pi} \ar[d]^{\unif} & (G,\varphi(\cX^+)_G) \ar[d]^{\unif_G}, &\tilde{Z}^\natural \ar@{|->}[d] \ar@{|->}[r] & \tilde{Z} \ar@{|->}[d] \ar@{|->}[r] & \tilde{Z}_G \ar@{|->}[d] \\
S^\natural \ar[r]^{[\pi^\sharp]} \ar@/_1pc/[rr]|{[\pi^\natural]}& S \ar[r]^{[\pi]} & S_G, & Z^\natural \ar@{|->}[r] & Z \ar@{|->}[r] & Z_G
}
\]
and use Notation~\ref{NotationSubsetsOfEMSD}. First of all we do the following reduction:
\begin{itemize}
\item The conclusion is obviously true if $\tilde{Z}$ is a point. Hence we may assume $\dim\tilde{Z}>0$.
\item Let $(P^\prime,\cX^{\prime,\natural+})$ be the smallest connected enlarged mixed Shimura subdatum of $(P,\cX^{\natural+})$ such that $\tilde{Z}^\natural\subset\cX^{\prime,\natural+}$ and let $S^{\prime,\natural}$ be the corresponding connected enlarged mixed Shimura subvariety of $S^\natural$. Then $(P^\prime,\cX^{\prime,\natural+})$ has generic Mumford-Tate group by Proposition~\ref{PropositionEMSDWithGenericMTGroupAlwaysExists}. Replace $(P,\cX^{\natural+})$ by $(P^\prime,\cX^{\prime,\natural+})$, $S^\natural$ by $S^{\prime,\natural}$, $\unif^\natural$ by $\unif^{\prime,\natural}\colon\cX^{\prime,\natural+}\rightarrow S^{\prime,\natural}$ and $Y^\natural$ by the irreducible component $Y^{\prime,\natural}$ of $Y^\natural\cap S^{\prime,\natural}$ such that $\tilde{Z}^\natural\subset(\unif^{\prime,\natural})^{-1}(Y^{\prime,\natural})$, then $\tilde{Z}^\natural$ is an irreducible algebraic subset of $\cX^{\prime,\natural+}$ contained in $(\unif^{\prime,\natural})^{-1}(Y^{\prime,\natural})$ and is maximal for these properties. By definition, $\tilde{Z}^\natural$ is bi-algebraic in $\cX^{\prime,\natural+}$ if and only if it is bi-algebraic in $\cX^{\natural+}$. So making the replacement above does not change the assumption or the conclusion of the theorem.
\item Replace $(P,\cX^{\natural+})$ by $(P,\cX^{\natural+})/Z(P)=(P/Z(P),\cX^{\natural+})$ and other objects accordingly.
\item Let $Y_0^\natural:=\unif^\natural(\tilde{Z}^\natural)^{\mathrm{Zar}}$, then $\tilde{Z}^\natural$ is an irreducible algebraic subset of $\cX^{\natural+}$ contained in $\unif^{\natural-1}(Y_0^\natural)$ and is maximal for these properties. Hence we may assume $Y^\natural=Y_0^\natural$.
\item After the previous reduction, there is a unique complex analytic irreducible component of $\unif^{\natural-1}(Y^\natural)$ containing $\tilde{Z}^\natural$, which we denote by $\tilde{Y}^\natural$.
\item Replace $\Gamma$ by a neat subgroup contained in $P^{\der}(\Q)$ (Remark~\ref{RemarkConnectedEMSDV}(3)).
\end{itemize}

Let $\tilde{F}^\natural$ be the smallest quasi-linear subset of $\cX^{\natural+}$ containing $\tilde{Y}^\natural$. Let $N$ be the connected algebraic monodromy group associated with $Y^{\natural,\mathrm{sm}}$ and let $W_N:=W\cap N$ ($U_N:=U\cap N$). Denote by $\tilde{F}:=\pi^\sharp(\tilde{F}^\natural)$ and $\tilde{F}_G:=\pi(\tilde{F})$. Then by Theorem~\ref{TheoremAxlogarithmic}, we have
\[
\tilde{F}=N(\R)^+U_N(\C)\tilde{z}\text{ and }\tilde{F}_G=G_N(\R)^+\tilde{z}_G
\]
where $\tilde{z}\in\tilde{Z}$ and $\tilde{z}_G=\pi(\tilde{z})$.

Let $(Q,\cY^+)$ the the smallest connected mixed Shimura subdatum of $(P,\varphi(\cX^+))$ such that $\tilde{Z}\subset\cY^+$. Let $\Gamma_Q:=\Gamma\cap Q(\Q)$. Then $\tilde{F}$ is Hodge generic in $\cY^+$, $N\lhd Q$ and $N\lhd Q^{\mathrm{der}}$. Define
\[
\Gamma_{\tilde{Z}^\natural}:=\{\gamma\in\Gamma_Q|\gamma\cdot\tilde{Z}^\natural=\tilde{Z}^\natural\}
\text{ and }\Gamma_{G_Q,\tilde{Z}_G}:=\{\gamma\in\Gamma_{G_Q}|\gamma\cdot\tilde{Z}_G=\tilde{Z}_G\}
\]
and $H_{\tilde{Z}^\natural}:=(\Gamma_{\tilde{Z}^\natural}^{\Zar})^{\circ}$, $H_{\tilde{Z}_Q}:=(\Gamma_{G_Q,\tilde{Z}_Q}^{\Zar})^{\circ}$. Then $H_{\tilde{Z}^\natural}$ (resp. $H_{\tilde{Z}_G}$) is the largest connected subgroup of $Q^{\der}$ (resp. $G^{\der}$) such that $H_{\tilde{Z}^\natural}(\R)^+W_{H_{\tilde{Z}^\natural}}(\C)$ (resp. 
 $H_{\tilde{Z}_G}(\R)^+$) stabilizes $\tilde{Z}^\natural$ (resp. 
 $\tilde{Z}_G$).

Define $U_{H_{\tilde{Z}^\natural}}:=U\cap H_{\tilde{Z}^\natural}$ and $W_{H_{\tilde{Z}^\natural}}:=W\cap H_{\tilde{Z}^\natural}$
. Both of them are normal in $H_{\tilde{Z}^\natural}$
. Define $V_{H_{\tilde{Z}^\natural}}:=W_{H_{\tilde{Z}^\natural}}/U_{H_{\tilde{Z}^\natural}}$
 and $G_{H_{\tilde{Z}^\natural}}:=H_{\tilde{Z}^\natural}/W_{H_{\tilde{Z}^\natural}}$.

\begin{rmk}\label{RmkEMSVSMVAL} 
We know that $P(\R)W(\C)$ acts on $\varphi(\cX^+)$ but not on $\cX^+ \subset \cX^{\natural+}$. However its subgroup $P(\R)U(\C)$ (as a Lie group) acts on both, and the isomorphism $\varphi \colon \cX^+ \rightarrow \varphi(\cX^+)$ is $P(\R)U(\C)$-equivariant. One should not be confused with the different formulae for the action of $P(\R)U(\C)$ on $\varphi(\cX^+) \cong \cX^+$ in different coordinates. For example $\tilde{Z}^\natural$ being $W_{H_{\tilde{Z}^\natural}}(\C)$-stable implies that $\tilde{Z} = \pi^\sharp(\tilde{Z}^\natural)$ is stable under $W_{H_{\tilde{Z}^\natural}}(\C)$, and hence $(\varphi|_{\cX^+})^{-1}(\tilde{Z})$ is stable under $W_{H_{\tilde{Z}^\natural}}(\R)U_{H_{\tilde{Z}^\natural}}(\C)$.
\end{rmk}

\begin{lemma}\label{LemmaStabiliserOfZstabilisesY}
The set $\tilde{Y}^\natural$ is stable under $H_{\tilde{Z}^\natural}(\R)^+W_{H_{\tilde{Z}^\natural}}(\C)$.
\begin{proof} Every fiber of $\cX^{\natural+}\rightarrow\varphi(\cX^+)_G$ can be canonically identified with $W(\C)$. So it is enough to prove that $\tilde{Y}$ is stable under $H_{\tilde{Z}}(\R)^+$: If $W_{H_{\tilde{Z}^\natural}}(\R)\tilde{y}\subset\tilde{Y}^\natural$ for $\tilde{y}^\natural\in\tilde{Y}^\natural$, then $W_{H_{\tilde{Z}^\natural}}(\C)\tilde{y}^\natural\subset\tilde{Y}^\natural$ because $\tilde{Y}^\natural$ is complex analytic and $W_{H_{\tilde{Z}^\natural}}(\C)\tilde{y}^\natural$ is the smallest complex analytic subset of $\cX^{\natural+}$ containing $W_{H_{\tilde{Z}^\natural}}(\R)\tilde{y}^\natural$.

If not, then since $H_{\tilde{Z}^\natural}(\Q)$ is dense (in the usual topology) in $H_{\tilde{Z}^\natural}(\R)^+$, there exists $h\in H_{\tilde{Z}^\natural}(\Q)$ such that $h\tilde{Y}^\natural\neq\tilde{Y}^\natural$. Then $\tilde{Z}^\natural\subset\tilde{Y}^\natural\cap h\tilde{Y}^\natural$, and hence contained in a complex analytic irreducible component $\tilde{Y}^{\prime,\natural}$ of it. Consider the Hecke correspondence $T_h$
\[
\Gamma\backslash\cX^{\natural+}\xleftarrow{[h]}(\Gamma\cap h^{-1}\Gamma h)\backslash\cX^{\natural+}\xrightarrow{[1]}\Gamma\backslash\cX^{\natural+}
\]
Then $T_h(Y^\natural)=\unif^\natural(h\cdot \unif^{\natural-1}(Y^\natural))$. 
Hence
\[
Y^\natural\cap T_h(Y^\natural)=\unif^\natural(\unif^{\natural-1}(Y^\natural)\cap(h\cdot \unif^{\natural-1}(Y^\natural))).
\]
On the other hand, $T_h(Y^\natural)$ is equidimensional of the same dimension as $Y^\natural$, hence by reason of dimension, $h\tilde{Y}^\natural$ is an irreducible component of $\unif^{\natural-1}(T_h(Y^\natural))=\Gamma h\Gamma\tilde{Y}^\natural$. So $\unif(h\tilde{Y}^\natural)$ is an irreducible component of $T_h(Y^\natural)$.

Since $\tilde{Y}^{\natural,\prime}$ is a complex analytic irreducible component of $\tilde{Y}^\natural\cap h\tilde{Y}^\natural$, it is also a complex analytic irreducible component of $\unif^{^\natural-1}(Y^\natural)\cap(h\tilde{Y}^\natural)=\Gamma\tilde{Y}^\natural\cap h\tilde{Y}^\natural$. So $Y^{\natural,\prime}:=\unif^\natural(\tilde{Y}^{\natural,\prime})$ is a complex analytic irreducible component of $Y^\natural\cap \unif(h\tilde{Y}^\natural)$. So $Y^{\natural,\prime}$ is a complex analytic irreducible component of $Y^\natural\cap T_h(Y^\natural)$, and hence is algebraic since $Y^\natural\cap T_h(Y^\natural)$ is.

Since $h\tilde{Y}^\natural\neq\tilde{Y}^\natural$ and $Y^\natural$ is irreducible, $\dim(Y^{\natural,\prime})<\dim(Y^\natural)$. But $\tilde{Z}^\natural\subset\tilde{Y}^\natural\cap h\tilde{Y}^\natural\subset \unif^{^\natural-1}(Y^{\natural,\prime})$. This contradicts the minimality of $Y^\natural$.
\end{proof}
\end{lemma}

\begin{lemma}\label{LemmaStabilizerNormalInN}
$H_{\tilde{Z}^\natural}\lhd N$.
\begin{proof} We have $\tilde{Z}^\natural\subset\tilde{F}^\natural$. Consider $pr_{\tilde{F}^\natural}^{\mathrm{ws}}$, the projection of $\tilde{F}^\natural$ to its weakly special part $(\tilde{F}^\natural)^{\mathrm{ws}}$ (see \eqref{EquationWeaklySpecialAndQuasiLinear}). Let $(Q,\cY^{\natural+})$ be the enlarged mixed Shimura subdatum of $(P,\cX^{\natural+})$ associated to $(Q,\cY^+)$. Then the image of $pr_{\tilde{F}^\natural}^{\mathrm{ws}}(\tilde{Z}^\natural)$ under $(Q,\cY^{\natural+})\rightarrow(Q,\cY^{\natural+})/N$ is a Hodge generic point (since $\tilde{F}$ is Hodge generic in $\cY^+$), which is stable under $H_{\tilde{Z}^\natural}/(H_{\tilde{Z}^\natural}\cap N)$. So $H_{\tilde{Z}^\natural}<N$.

Let $H^\prime$ be the algebraic group generated by $\gamma^{-1}H_{\tilde{Z}^\natural}\gamma$ for all $\gamma\in\Gamma_{Y^{\natural,\sm}}$, where $\Gamma_{Y^{\natural,\sm}}$ is the monodromy group of $Y^{\natural,\sm}$. Since $H^\prime$ is invariant under conjugation by $\Gamma_{Y^{\natural,\sm}}$, it is invariant under $(\Gamma_{Y^{\natural,\sm}})^{\Zar}$, therefore invariant under conjugation by $N$.

By Lemma~\ref{LemmaStabiliserOfZstabilisesY}, $\tilde{Y}^\natural$ is invariant under $H_{\tilde{Z}^\natural}(\R)^+W_{H_{\tilde{Z}^\natural}}(\C)$. On the other hand, 
$\tilde{Y}^\natural$ is also invariant under $\Gamma_{Y^{\natural,\sm}}$. So $\tilde{Y}^\natural$ is invariant under the action of $H^\prime(\R)^+W_{H^\prime}(\C)$ where $W_{H^\prime}:=W\cap H^\prime$. Since $H^\prime(\R)^+W_{H^\prime}(\C)\tilde{Z}^\natural$ is semi-algebraic, there exists an irreducible algebraic subset of $\cX^{\natural+}$, say $\tilde{E}^\natural$, which contains $H^\prime(\R)^+W_{H^\prime}(\C)\tilde{Z}^\natural$ and is contained in $\tilde{Y}^\natural$ by \cite[Lemma~4.1]{PilaAbelianSurfaces}.
Now $\tilde{Z}^\natural\subset\tilde{E}^\natural\subset\tilde{Y}^\natural$, so 
$\tilde{Z}^\natural=\tilde{E}^\natural=H^\prime(\R)^+W_{H^\prime}(\C)\tilde{Z}^\natural$ by maximality of $\tilde{Z}^\natural$, and therefore 
$H^\prime=H_{\tilde{Z}^\natural}$ by the maximality of $H_{\tilde{Z}^\natural}$. So $H_{\tilde{Z}^\natural}$ is invariant under conjugation by $N$. But $H_{\tilde{Z}^\natural}<N$, so $H_{\tilde{Z}^\natural}$ is normal in $N$.
\end{proof}
\end{lemma}

\begin{cor}\label{CorollaryNormalSubgroups}
$G_{H_{\tilde{Z}^\natural}}\lhd G_Q^\der$.
\begin{proof} We have $G_{H_{\tilde{Z}^\natural}}\lhd G_N\lhd G_Q^\der$, and so $G_{H_{\tilde{Z}^\natural}}\lhd G_Q^\der$ since all the three groups are reductive. 
\end{proof}
\end{cor}

\begin{lemma}\label{LemmaProjectionOfMaximalStillMaximalAfterTakingClosure}
\begin{enumerate}
\item The variety $Y_G$ is weakly special. Hence $\tilde{Y}_G=\tilde{F}_G=G_N(\R)^+\tilde{y}_G$ for any $\tilde{y}_G\in\tilde{Y}_G$.
\item We have $\unif(\tilde{Z}_G)^{\Zar}=Y_G$.
\end{enumerate}
\begin{proof}
We have $\unif(\tilde{Z}_G)^{\Zar}=Y_G$ because $Y^\natural=\unif^\natural(\tilde{Z}^\natural)^{\Zar}$. The Ax-Lindemann Theorem for Shimura varieties \cite{KlinglerThe-Hyperbolic-} implies that $Y_G=Z_G^{\Zar}$ is weakly special.
\end{proof}
\end{lemma}

\begin{prop}\label{PropositionKeyProposition}
We have $\tilde{Z}_G=G_{H_{\tilde{Z}^\natural}}(\R)^+\tilde{z}_G$. Hence $\tilde{Z}_G$ is weakly special.
\end{prop}
The proof of Proposition~\ref{PropositionKeyProposition} will occupy $\mathsection$\ref{SubsectionAxLindemannPart2}. Here we show how to prove the Ax-Lindemann theorem assuming this proposition. We shall also use the following theorem, which we will prove as Theorem~\ref{ThmALFiber}.

\begin{thm}\label{TheoremAxLindemannForFiber}
The Ax-Lindemann theorem (Theorems~\ref{AxLindemann}$\sim$\ref{AxLindemannEquiDistributionSemiAlgebraic}) is true if $\tilde{Z}_G$ is a point.
\end{thm}
\begin{proof}[Proof of Theorem~\ref{AxLindemann}] We divide the proof into the following steps.

\noindent\textit{\underline{\textbf{Step I}}} Prove $\tilde{Z}^\natural=H_{\tilde{Z}^\natural}(\R)^+W_{H_{\tilde{Z}^\natural}}(\C)(\tilde{Z}^\natural)_{\tilde{z}_G}$ for some $\tilde{z}_G\in\tilde{Z}_G$.

Consider a fibre of $\tilde{Z}^\natural$ over a Hodge-generic point $\tilde{z}_G\in \tilde{Z}_G$ such that $\pi^\natural|_{\tilde{Z}^\natural}$ is flat at $\tilde{z}_G$ (such a point exists by \cite[$\mathsection 4$, Lemma~1.4]{AndreMumford-Tate-gr} and generic flatness). Suppose that $\tilde{K}^\natural$ is an irreducible algebraic component of $\tilde{Z}^\natural_{\tilde{z}_G}$ such that $\dim(\tilde{Z}^\natural_{\tilde{z}_G})=\dim(\tilde{K}^\natural)$, then the flatness of  $\pi^\natural|_{\tilde{Z}^\natural}$ at $\tilde{z}_G$ implies $\dim(\tilde{Z}^\natural)=\dim(\tilde{Z}_G)+\dim(\tilde{Z}^\natural_{\tilde{z}_G})=\dim(\tilde{Z}_G)+\dim(\tilde{K}^\natural)$.

Consider the set $\tilde{E}^\natural:=H_{\tilde{Z}^\natural}(\R)^+W_{H_{\tilde{Z}^\natural}}(\C)\tilde{K}^\natural$. It is semi-algebraic (since $\tilde{K}$ is algebraic and the action of $P(\R)^+W(\C)$ on $\cX^{\natural+}$ is algebraic). The fact $\tilde{K}^\natural\subset\tilde{Z}^\natural$ implies that $\tilde{E}^\natural\subset\tilde{Z}^\natural$. Now Proposition~\ref{PropositionKeyProposition} implies that $\pi^\natural(\tilde{E}^\natural)=G_{H_{\tilde{Z}^\natural}}(\R)^+\tilde{z}_G=\tilde{Z}_G$ and that the $\R$-dimension of every fiber of $\pi^\natural|_{\tilde{E}^\natural}$ is at least $\dim_\R(\tilde{K}^\natural)$. So
\[
\dim(\tilde{Z}^\natural)\ge \dim(\tilde{E}^\natural)\ge \dim(\pi^\natural(\tilde{E}^\natural))+\dim(\tilde{K}^\natural)=\dim(\tilde{Z}_G)+\dim(\tilde{K}^\natural)=\dim(\tilde{Z}^\natural).
\]
So $\tilde{E}^\natural=\tilde{Z}^\natural$ since $\tilde{Z}^\natural$ is irreducible.

Next let $\tilde{K}^{\natural,\prime}$ be an irreducible algebraic subset which contains $\tilde{Z}^\natural_{\tilde{z}_G}$ and is contained in $\unif^{\natural-1}(Y^\natural)_{\tilde{z}_G}$, maximal for these properties. Then $\tilde{K}^{\natural,\prime}$ is quasi-linear by Theorem~\ref{TheoremAxLindemannForFiber}. We have $\tilde{K}^{\natural,\prime}\subset\tilde{Y}^\natural$ since $\tilde{Y}^\natural$ is an irreducible component of $\pi^{^\natural-1}(Y^\natural)$. Consider $\tilde{E}^{\natural,\prime}:=H_{\tilde{Z}^\natural}(\R)^+W_{H_{\tilde{Z}^\natural}}(\C)\tilde{K}^{\natural,\prime}$. Then $\tilde{E}^{\natural,\prime}\subset\tilde{Y}^\natural$ 
by Lemma~\ref{LemmaStabiliserOfZstabilisesY}. But $\tilde{E}^{\natural,\prime}$ is semi-algebraic, so by \cite[Lemma~4.1]{PilaAbelianSurfaces}, there exists an irreducible algebraic subset of $\cX^{\natural,+}$, say $\tilde{E}^{\natural,\prime}_{\mathrm{alg}}$ which contains $\tilde{E}^{\natural,\prime}$ and is contained in $\tilde{Y}^\natural$. So $\tilde{Z}^\natural=\tilde{E}^\natural\subset\tilde{E}^{\natural,\prime}_{\mathrm{alg}}\subset\tilde{Y}^\natural$. Now the maximality of $\tilde{Z}^\natural$ implies that
\begin{equation}\label{EquationTildeZNaturalWriting}
\tilde{Z}^\natural=\tilde{E}^{\natural,\prime}_{\mathrm{alg}}=\tilde{E}^{\natural,\prime}=H_{\tilde{Z}^\natural}(\R)^+W_{H_{\tilde{Z}^\natural}}(\C)\tilde{K}^{\natural,\prime}\text{ and }\tilde{K}^{\natural,\prime}=\tilde{Z}^\natural_{\tilde{z}_G}.
\end{equation}

\noindent\textit{\underline{\textbf{Step II}}} Prove $\tilde{Z}^{\natural}=H_{\tilde{Z}^\natural}(\R)^+U_{H_{\tilde{Z}^\natural}}(\C)\tilde{z}^{\natural}$.

Since $\tilde{K}^{\natural,\prime}$ is quasi-linear, we can write $\tilde{K}^\prime(:=\pi^\sharp(\tilde{K}^{\natural,\prime}))=W^\prime(\R)U^\prime(\C)\tilde{z}$ with $W^\prime<W$, $U^\prime=U\cap W^\prime$ and $\tilde{z}\in\tilde{Z}_{\tilde{z}_G}$. Then $W_{H_{\tilde{Z}^\natural}}<W^\prime$ and $\tilde{K}^{\natural,\prime}$ is stable under $W^\prime(\C)$. The complex structure of $\pi^{-1}(\tilde{z}_G)$ comes from $W(\R)U(\C)\cong W(\C)/F^0_{\tilde{z}_G}W(\C)$, where $F^0_{\tilde{z}_G}W(\C)=\exp(F^0_{\tilde{z}_G}\lie W_{\C})$. So the fact that $\tilde{Z}_{\tilde{z}_G}$ is a complex subspace of $\pi^{-1}(\tilde{z}_G)$ implies that $W^\prime/U^\prime$ is a $\MT(\tilde{z}_G)=G_Q$-module. Hence $W^\prime$ is a $G_Q$-group. 

Define $Q^\prime:=W^\prime H_{\tilde{Z}^\natural}$, then $Q^\prime$ is a subgroup of $Q$ since $W^\prime>W_{H_{\tilde{Z}^\natural}}$ and $G_{H_{\tilde{Z}^\natural}}W^\prime=W^\prime$. Now we have that $\tilde{Z}^\natural=H_{\tilde{Z}^\natural}(\R)^+W_{H_{\tilde{Z}^\natural}}(\C)\tilde{K}^{\natural,\prime}$ is stable under $Q^\prime$. So $H_{\tilde{Z}^\natural}=Q^\prime$ because $H_{\tilde{Z}^\natural}$ is the largest subgroup of $Q^{\der}$ such that $H_{\tilde{Z}^\natural}(\R)^+W_{H_{\tilde{Z}^\natural}}(\C)$ stabilizes $\tilde{Z}^\natural$. So $W_{H_{\tilde{Z}^\natural}}=W^\prime$ is a $G_Q$-group. Therefore $\tilde{Z}=H_{\tilde{Z}^\natural}(\R)^+U_{H_{\tilde{Z}^\natural}}(\C)\tilde{z}$.

\noindent\textit{\underline{\textbf{Step III}}} Prove $H_{\tilde{Z}^\natural}\lhd Q$, and hence $\tilde{Z}$ is weakly special and $H_{\tilde{Z}^\natural}=N$.

First of all, $U_{H_{\tilde{Z}^\natural}}\lhd Q$ by Proposition~\ref{PropositionEMSDWithGenericMTGroupAlwaysExists}(2).

Next consider the complex structure of $\pi^{-1}(\tilde{z}_G)$ which comes from $W(\R)U(\C)$ $\cong W(\C)/F^0_{\tilde{z}_G}W(\C)$. So the fact that $\tilde{Z}_{\tilde{z}_G}$ is a complex subspace of $\pi^{-1}(\tilde{z}_G)$ implies that $V_{H_{\tilde{Z}^\natural}}$ is a $\MT(\tilde{z}_G)=G_Q$-module. Hence $W_{H_{\tilde{Z}^\natural}}$ is a $G_Q$-group. Besides, $G_{H_{\tilde{Z}^\natural}}\lhd G_Q$ by Proposition~\ref{PropositionKeyProposition}. In particular, $G_{H_{\tilde{Z}^\natural}}$ is reductive.

Then let us prove $W_{H_{\tilde{Z}^\natural}}\lhd Q$. It suffices to prove $W_{H_{\tilde{Z}^\natural}}\lhd W_Q$. For any $\tilde{z}\in\tilde{Z}$, we have proved in Step~I that $\tilde{Z}_{\tilde{z}_G}=W_{H_{\tilde{Z}^\natural}}(\R)U_{H_{\tilde{Z}^\natural}}(\C)\tilde{z}$ is weakly special. Hence there is a connected mixed Shimura subdatum $(R,\cZ^+)\hookrightarrow (Q,\cY^+)$ such that $\tilde{z}\in\cZ^+$ and $W_{H_{\tilde{Z}^\natural}}\lhd R$. Define $W^*$ to be the $G_Q$-subgroup (of $W_Q$) generated by $W_R:=\cR_u(R)$, then $W_{H_{\tilde{Z}^\natural}}\lhd W^*$ since $W_{H_{\tilde{Z}^\natural}}$ is a $G_Q$-group.

Fix a Levi decomposition $H_{\tilde{Z}^\natural}=W_{H_{\tilde{Z}^\natural}}\rtimes G_{H_{\tilde{Z}^\natural}}$ and choose a compatible Levi decomposition $Q=W_Q\rtimes G_Q$ (as is shown in \cite[Lemma~9.7]{GaoTowards-the-And}). Let $Q^*$ be the group generated by $G_QR$, then $\cR_u(Q^*)=W^*$ and $Q^*/Q^*=G_Q$. The group $Q^*$ defines a connected mixed Shimura datum $(Q^*,\cY^{*+})$ with $\cY^{*+}=Q^*(\R)^+U_Q^*(\C)\tilde{z}$. Now $\tilde{Z}=H_{\tilde{Z}^\natural}(\R)^+U_{H_{\tilde{Z}^\natural}}(\C)\tilde{z}\subset\cY^{*+}$. But $\tilde{Z}$ is Hodge generic in $\cY^+$ by assumption, hence $Q=Q^*$ and $W_Q=W^*$. So $W_{H_{\tilde{Z}^\natural}}\lhd W_Q$ and hence $W_{H_{\tilde{Z}^\natural}}\lhd Q$.

Use the notation in $\mathsection$\ref{SubsectionStructureOfUnderlyingGroup}. We are done if we can prove:
\[
(u,v,1)(0,0,g)(-u,-v,1)\in H_{\tilde{Z}^\natural}\text{ for any }u\in U_Q,~ v\in V_Q\text{ \textit{and} } g\in G_{H_{\tilde{Z}^\natural}}.
\]
By \cite[Corollary~2.14]{GaoTowards-the-And}, there exist decompositions $U_Q=U_N\bigoplus U_N^\bot$ and $V_Q=V_N\bigoplus V_N^\bot$ as $G_Q$-modules such that $G_N$ acts trivially on $U_N^\bot$ and $V_N^\bot$. Now 
\begin{align*}
(u,v,1)(0,0,g)(-u,-v,1) &
=(u,v,g)(-u,-v,1) \\
&=(u-g\cdot u,v-g\cdot v,g) \\
&=((u_N+u_N^\bot)-g\cdot(u_N+u_N^\bot),(v_N+v_N^\bot)-g\cdot(v_N+v_N^\bot),g) \\
&=(u_N-g\cdot u_N,v_N-g\cdot v_N,g) \\
&=(u_N,v_N,1)(0,0,g)(-u_N,-v_N,1)\in H_{\tilde{Z}^\natural},
\end{align*}
where the last inclusion follows from Lemma~\ref{LemmaStabilizerNormalInN}.

\noindent\textit{\underline{\textbf{Conclusion}}} In view of Step~I and Step~III, we can apply Corollary~\ref{CorollaryCriterionOfBiAlgebraicity} to $\tilde{Z}^\natural$. So $\tilde{Z}^\natural$ is bi-algebraic, and hence is quasi-linear by Theorem~\ref{TheoremCharacterizationOfBiAlgebraicSubvarieties}.
\end{proof}

\subsection{Estimate}\label{SubsectionAxLindemannPart2}
The goal of this subsection is to prove Proposition~\ref{PropositionKeyProposition}.

Keep notation and assumptions as in the previous subsection. Consider $\cX^{\natural+}|_{\cY^+}:=\pi^{\sharp-1}(\cY^+)$. For the uniformization $\unif^\natural\colon\cX^{\natural+}|_{\cY^+}\rightarrow S^\natural|_{S_Q}:=[\pi^\sharp]^{-1}(S_Q)$, there exists a fundamental set $\mathfrak{F}^\natural$ such that $\unif^\natural|_{\mathfrak{F}^\natural}$ is definable in the o-minimal structure $\R_{\mathrm{an},\exp}$: by \cite[$\mathsection$10.1]{GaoTowards-the-And} there exists a fundamental set $\mathfrak{F}$ for $\unif_Q\colon\cY^+\rightarrow S_Q$ such that $\unif_Q|_{\mathfrak{F}}$ is definable in $\R_{\mathrm{an},\exp}$. It suffices to take
\[
\mathfrak{F}^\natural:=\pi_{P/U}^{\sharp-1}(\mathfrak{F}_{Q/U_Q})\times_{\mathfrak{F}_{Q/U_Q}}\mathfrak{F}\subset\pi_{P/U}^{\sharp-1}(\cY^+_{Q/U_Q})\times_{\cY^+_{Q/U_Q}}\cY^+=\cX^{\natural+}|_{\cY^+}.
\]

Recall that $G_{H_{\tilde{Z}^{\natural}}} \lhd G_N \lhd G_Q^{\der}$ by Corollary~\ref{CorollaryNormalSubgroups}. Since $G_Q^{\der}$ is semisimple, there exists a semisimple subgroup $G_0$ of $G_N$ and a semisimple subgroup $G'$ of $G_Q$ such that $G_N = G_{H_{\tilde{Z}^\natural}} G_0$ and $G_Q = G_N G'$ as almost direct products. 

We wish to prove $G_0=1$. The almost direct product $G_Q = G_{H_{\tilde{Z}^{\natural}}} G_0 G'$ induces a decomposition of Shimura data
\[
(G_Q^{\mathrm{ad}} ,\cY^+_{G_Q}) \cong (G_{H_{\tilde{Z}^{\natural}}}^{\mathrm{ad}}, \cY^+_{G_{H_{\tilde{Z}^{\natural}}}}) \times  (G_0, \cY^+_{G_0}) \times (G',\cY^+_{G'}).
\]
Under this decomposition, $\tilde{Z}_G = \cY^+_{G_{H_{\tilde{Z}^{\natural}}}} \times \tilde{Z}_{G_0} \times \{\text{point}\}$ for some $\tilde{Z}_{G_0} \subset \cY^+_{G_0}$. 

Assume that $G_0$ is non-trivial. Then Lemma~\ref{LemmaProjectionOfMaximalStillMaximalAfterTakingClosure} implies that $\dim \tilde{Z}_{G_0}>0$.

Fix a point $\tilde{z}^{\natural}\in\mathfrak{F}^\natural\cap\tilde{Z}^{\natural}$. Take an algebraic curve $C_G\subset \tilde{Z}_G$ passing through $\pi^\natural(\tilde{z}^{\natural})$ such that 
\begin{equation}\label{EqConditionCG}
\text{the image of }C_G\text{ under the projection }\cY^+_{G_Q} \rightarrow \cY^+_{G_{H_{\tilde{Z}^{\natural}}}}\text{ is a point.}
\end{equation}
Now $\pi^\natural(\tilde{Z}^{\natural}\cap\pi^{\natural-1}(C_G))=\tilde{Z}_G\cap C_G=C_G$, and hence there exists an algebraic curve $C^\natural$ in $\tilde{Z}^{\natural}\cap\pi^{\natural-1}(C_G)$ passing through $\tilde{z}^{\natural}$ such that $\dim(\pi^{\natural}(C^\natural))=1$.

Recall that the set $\mathfrak{F}_G$  is a fundamental set of $\unif_{G_Q}$ and $\unif_{G_Q}|_{\mathfrak{F}_G}$ is definable. We define for any irreducible semi-algebraic subvariety $A^\natural$ (resp. $A_G$) of $\unif^{\natural-1}(Y^\natural)$ (resp. $\unif_G^{-1}(Y_G)$) the following sets:
define
\small
\[
\begin{array}{cc}
\Sigma(A^\natural):=\{g\in Q^{\der}(\R) = (W_Q G_Q^{\der})(\R)|\dim(gA^\natural\cap \unif^{\natural-1}(Y^\natural)\cap\mathfrak{F}^\natural)=\dim(A^\natural)\} \\
(\text{resp. }\Sigma_G(A_G):=\{g\in G_Q^{\der}(\R)|\dim(gA_G\cap \unif_{G_Q}^{-1}(\bar{Y_G})\cap\mathfrak{F}_G)=\dim(A_G)\})
\end{array}
\]
\normalsize
and
\small
\[
\begin{array}{cc}
\Sigma(A^\natural):=\{g\in Q(\R)|g^{-1}\mathfrak{F}^\natural\cap A^\natural\neq\emptyset\} \\
(\text{resp. }\Sigma_G(A_G):=\{g\in G_Q(\R)|g^{-1}\mathfrak{F}_G\cap A_G\neq\emptyset\})
\end{array}.
\]
\normalsize
Then $\Sigma(A^\natural)$ and $\Sigma_G(A_G)$ are by definition definable.

\begin{lemma}\label{the two Sigma-sets coincide}
$\Sigma^{\prime}(A^\natural)\cap\Gamma_Q=\Sigma(A^\natural)\cap\Gamma_Q$ (resp. $\Sigma^{\prime}_G(A_G)\cap\Gamma_{G_Q}=\Sigma_G(A_G)\cap\Gamma_{G_Q}$).
\begin{proof} First $\Sigma(A^\natural)\cap\Gamma_Q\subset\Sigma^{\prime}(A^\natural)\cap\Gamma_Q$ by definition. Conversely for any $\gamma\in\Sigma^{\prime}(A^\natural)\cap\Gamma_Q$, the set $\gamma^{-1}\mathfrak{F}^\natural\cap A^\natural$ contains an open subspace of $A^\natural$ since $\mathfrak{F}^\natural$ is by choice open in $\cX^{\natural+}|_{\cY^+}$. Hence $\gamma A^\natural\cap \unif^{\natural-1}(Y^\natural)\cap\mathfrak{F}^\natural=\gamma A^\natural\cap\mathfrak{F}^\natural$ contains an open subspace of $\gamma A^\natural$ which must be of dimension $\dim(A^\natural)$. Hence $\gamma\in\Sigma(A^\natural)\cap\Gamma_Q$. The proof for $A_G$ is the same.
\end{proof}
\end{lemma}

This lemma implies
\vspace{-2mm}
\small
\begin{equation}\label{relation between Sigma of curve and Sigma of Z}
\begin{array}{cc}
\Sigma(C^\natural)\cap\Gamma_Q=\Sigma^{\prime}(C^\natural)\cap\Gamma_Q\subset
\Sigma^{\prime}(\tilde{Z}^{\natural})\cap\Gamma_Q=\Sigma(\tilde{Z}^{\natural})\cap\Gamma_Q \\
(\text{resp. }\Sigma_G(C_G)\cap\Gamma_{G_Q}=\Sigma^{\prime}_G(C_{G_i})\cap\Gamma_{G_Q}\subset\Sigma^{\prime}_G(\bar{\tilde{Z}^{\natural}_G})\cap\Gamma_{G_Q}=\Sigma(\bar{\tilde{Z}^{\natural}_G})\cap\Gamma_{G_Q})
\end{array}.
\end{equation}
\normalsize

For $T>0$, define
\[
\Theta_G(C_G,T):=\{\gamma_G\in\Gamma_{G_Q}\cap\Sigma_G(C_G)|H(\gamma_G)\leq  T\}.
\]

\begin{lemma}\label{projection of theta is theta}
\begin{enumerate}
\item[(i)] We have $\pi^\natural(\Gamma_Q\cap\Sigma(C^\natural))=\Gamma_{G_Q}\cap\Sigma_G(C_G)$.
\item[(ii)] There exists a constant $c'>0$ such that the following statement holds. For any $\gamma \in \Gamma_Q\cap\Sigma(C^\natural)$ such that $\pi^{\natural}(\gamma) \in \Theta_G(C_G,T)$, we have that $H(\gamma) \le c'T$.
\end{enumerate}
\begin{proof}
By Lemma~\ref{the two Sigma-sets coincide}, it suffices to prove
$\pi^\natural(\Gamma\cap\Sigma^{\prime}(C^\natural))=\Gamma_{G_Q}\cap\Sigma^{\prime}_G(C_G)$.
The inclusion $\subset$ is clear by definition. For the other inclusion, $\forall\gamma_G\in\Gamma_{G_Q}\cap\Sigma^{\prime}_G(C_G)$, $\exists c_G\in C_G$ such that $\gamma_G\cdot c_G\in\mathfrak{F}_G$.

Take a point $c^\natural\in C^\natural$ such that $\pi^\natural(c^\natural)=c_G$ and define $c_G:=\pi^\natural(c^\natural)\in\cY_{G_Q}^+$.

Let $\gamma_G\in\Gamma_{G_Q}$ be such that $\gamma_G\cdot c^\natural\in\pi^{\natural-1}(\mathfrak{F}_G)$.
Therefore there exist $\gamma_V\in\Gamma_{V_Q},~\gamma_U\in\Gamma_{U_Q}$ such that $(\gamma_U,\gamma_V,\gamma_G)c^\natural\in\mathfrak{F}^\natural$. 
Denote by $\gamma=(\gamma_U,\gamma_V,\gamma_G)$, then $\gamma\in\Gamma_Q\cap\Sigma^{\prime}(C^\natural)$ and $\pi^\natural_i(\gamma)=\gamma_G$. This proves (i).

For (ii), recall that the algebraic structure on $\cX^{\natural+}$ is given by the inclusion $\cX^{\natural+} \cong W(\C) \times \cX_G^+ \subset W(\C) \times \cX_G^\vee$. We furthermore decompose $W(\C) \cong U(\C) \times V(\C)$ as varieties. We shall use the $\ell^{\infty}$ norms on the complex vector spaces $U(\C)$ and $V(\C)$. 

Denote by $\bar{C_G}$ the Zariski closure of $C_G$ in $\cX_G^\vee$. Then $C_G$ is a complex analytic irreducible component of $\bar{C_G} \cap \cX_G$. But $\cX_G$ is a compact subset of $\cX_G^\vee$ via the Harish-Chandra realization, and $C^{\natural} \rightarrow C_G$ is a finite map. Hence for any point $c = (c_U,c_V,c_G) \in C^{\natural}$, we have that $|| c_U ||$ and $|| c_V ||$ are uniformly bounded only in terms of $C^{\natural}$. Call this bound $c'-1$ where $c'>1$ is a number.

Now let $\gamma = (\gamma_U,\gamma_V,\gamma_G) \in \Gamma_Q$ be such that $\gamma \cdot c \in \mathfrak{F}^{\natural}$. We have $\gamma \cdot c = (\gamma_U+\gamma_G \cdot c_U, \gamma_V+ \gamma_G \cdot c_V, \gamma_G\cdot c_G)$. If $H(\gamma_G) \le T$, then
\[
H(\gamma_U) \le H(\gamma_G)||c_U|| + 1 \le c'T,~ H(\gamma_V) \le H(\gamma_G) ||c_V|| + 1 \le c'T.
\]
\end{proof}
\end{lemma}

\begin{prop}\label{theta for the base big enough}
There exists a constant $\delta>0$ such that for all $T\gg0$, $|\Theta_G(C_G,T)|\ge  T^\delta$.
\begin{proof} This follows directly from \cite[Theorem~1.3]{KlinglerThe-Hyperbolic-} applied to \small$((G_Q,\cX_{G_Q}^+), S_G, \tilde{Z}_G)$.
\end{proof}
\end{prop}
\normalsize
Let us prove how these facts imply $H_i<G_{H_{\tilde{Z}^{\natural}}}$. Take a faithful representation $G_Q\hookrightarrow\GL_n$ which sends $\Gamma_{G_Q}$ to $\GL_n(\Z)$. Consider the definable set $\Sigma_G(C_G)$. By the theorem of Pila-Wilkie \cite[Theorem~3.6]{PilaO-minimality-an}, there exist $J=J(\delta)$ definable block families
\[
B^j\subset\Sigma_G(C_G)\times\R^l,\qquad j=1,...,J
\]
and $c=c(\delta)>0$ such that for all $T\gg0$, $\Theta_G(C_G,T^{1/2n})$ is contained in the union of at most $cT^{\delta/4n}$ definable blocks of the form $B^j_y$ ($y\in\R^l$). By Proposition~\ref{theta for the base big enough}, there exist a $j\in\{1,...,J\}$ and a block $B_G:=B^j_{y_0}$ of $\Sigma_G(C_G)$
containing at least $T^{\delta/4n}$ elements of $\Theta_G(C_G,T^{1/2n})$.

Let $\Sigma:=\Sigma(C^\natural)\cap\Sigma(\tilde{Z}^{\natural})$, which is by definition a definable set. Consider $X^j:=(\pi^\natural_i\times 1_{\R^l})^{-1}(B^j)\cap(\Sigma\times\R^l)$, which is a definable family since $\pi^\natural$ is algebraic. By \cite[Ch. 3, 3.6]{DriesTame-Topology-a}, there exists a number $n_0>0$ such that each fibre $X^j_y$ has at most $n_0$ connected components.
So the definable set $\pi^{\natural-1}(B_G)\cap\Sigma$ has at most $n_0$ connected components. 
Now
\small
\[
\pi^\natural(\pi^{\natural-1}(B_G)\cap\Sigma\cap\Gamma_Q)=B_G\cap\pi^\natural(\Sigma(C^\natural)\cap\Gamma_Q)=B_G\cap\Sigma_G(C_G)\cap\Gamma_{G_Q}=B_G\cap\Gamma_{G_Q}
\]
\normalsize
by \eqref{relation between Sigma of curve and Sigma of Z} and Lemma~\ref{projection of theta is theta}(i). So there exists a connected component $B$ of $\pi^{\natural-1}(B_G)\cap\Sigma$ such that $\pi^\natural(B\cap\Gamma_Q)$ contains at least $T^{\delta/4n}/n_0$ elements of $\Theta_G(C_G,T^{1/2n})$. By Lemma~\ref{projection of theta is theta}(ii) $B$ contains at least $T^{\delta/4n}/n_0$ elements of height $\le c'T^{1/2n}$. Again by Pila-Wilkie \cite[Theorem~3.6]{PilaO-minimality-an}, there exists a number $n_1>0$ and a block $B'$ in $B$ such that $\pi^{\natural}(B'\cap \Gamma_Q)$ contains at least $T^{\delta/4n}/n_1$ elements of $\Theta_G(C_G,T^{1/2n})$.

We have $B'\tilde{Z}^{\natural}\subset \unif^{\natural-1}(Y^\natural)$ since $\Sigma(\tilde{Z}^{\natural})\tilde{Z}^{\natural}\subset \unif^{\natural-1}(Y^\natural)$ by analytic continuation, and $\tilde{Z}^{\natural}\subset\sigma^{-1}B'\tilde{Z}^{\natural}$ for any $\sigma\in B'\cap\Gamma_Q$. But $B'$ is connected, and therefore $\sigma^{-1}B'\tilde{Z}^{\natural}=\tilde{Z}^{\natural}$ by maximality of $\tilde{Z}^{\natural}$ and \cite[Lemma~4.1]{PilaAbelianSurfaces}. So $B'\subset\sigma\stab_{Q(\R)}(\tilde{Z}^{\natural})$ for any $\sigma\in B'\cap\Gamma_Q$.

Fix a $\gamma_0\in B'\cap\Gamma_Q$ such that $\pi^\natural(\gamma_0)\in\Theta_G(C_G,T^{1/2n})$. We have already shown that $\pi^\natural(B'\cap\Gamma_Q)$ contains at least $T^{\delta/4n}/n_1$ elements of $\Theta_G(C_G,T^{1/2n})$. For any $\gamma^\prime_G\in\pi^\natural(B'\cap\Gamma_Q)\cap\Theta_G(C_G,T^{1/2n})$, let $\gamma^\prime$ be one of its pre-images in $B\cap\Gamma_Q$. By \eqref{EqConditionCG}, both $\pi^\natural(\gamma_0)$ and $\gamma'_G$ satisfy the following property: the images are $1$ under the projections $G_Q \rightarrow G_{H_{\tilde{Z}}}^{\mathrm{ad}}$ and $G_Q \rightarrow G^{\prime,\mathrm{ad}}$ induced by almost direct product $G_Q = G_{H_{\tilde{Z}^{\natural}}}G_0 G'$.

Now $\gamma:=\gamma^{\prime-1}\gamma_0$ is an element of $\Gamma_Q\cap\stab_{Q(\R)}(\tilde{Z}^{\natural}) = \Gamma_{\tilde{Z}^{\natural}}$ such that $H(\pi^\natural(\gamma))\ll T^{1/2}$. 
Therefore for $T\gg0$, $\pi^\natural(\Gamma_{\tilde{Z}^{\natural}})$ contains at least $T^{\delta/4n}/n_1$ elements $\gamma_G$ such that $H(\gamma_G)\leq  T$. Moreover under the projections $G_Q \rightarrow G_{H_{\tilde{Z}}}^{\mathrm{ad}}$ and $G_Q \rightarrow G^{\prime,\mathrm{ad}}$ induced by the almost direct product $G_Q = G_{H_{\tilde{Z}^{\natural}}}G_0 G'$, the images of these $\gamma_G$ are $1$. 

Hence under the projection $G_Q \rightarrow G_0$ induced by the almost direct product $G_Q = G_{H_{\tilde{Z}^{\natural}}}G_0 G'$, the image of $\pi^\natural(H_{\tilde{Z}^{\natural}})$ is of positive dimension because it contains infinitely many rational points. But this contradicts the maximality of $H_{\tilde{Z}^{\natural}}$. Hence $G_0=1$ and we are done.

\subsection{Proof for the Fiber}
In this subsection we prove Ax-Lindemann for the fiber of $S^\natural \rightarrow S_G$. More precisely the following statement.
\begin{thm}\label{ThmALFiber}
Let $S^\natural$ be a connected enlarged mixed Shimura variety associated with $(P,\cX^{\natural+})$. Let $E^\natural$ be a fiber of $S^\natural \rightarrow S_G$. Let $Y^\natural$ be an irreducible closed subvariety of $E^\natural$. Assume $\tilde{Z}^\natural$ is an algebraic subset of $\cX^{\natural+}$ contained in $\unif^{\natural -1}(Y^\natural)$, maximal for this property. Then $\tilde{Z}^\natural$ is quasi-linear.
\end{thm}

As before we may and do assume that $Y^\natural = \mathrm{unif}^{\natural}(\tilde{Z}^\natural)^{\mathrm{Zar}}$. 
Let $N$ be the connected algebraic monodromy group of $Y^\natural$. Let $H_{\tilde{Z}^\natural}$ be as defined above Remark~\ref{RmkEMSVSMVAL}. Note that in this case, $N = W_N$ and $H_{\tilde{Z}^\natural}$ is also a unipotent group, which we denote by $W_0$. Again we have $W_0 \lhd W_N$ by Lemma~\ref{LemmaStabilizerNormalInN}.

Let us fix some extra notations for the proof. These notations will only be used in this subsection.
\begin{equation}\label{EqNotationFiberBigDiag}
\xymatrix{
W(\C) \ar[r]^{\pi^\natural} \ar[d]_{\unif^\natural} & V(\C) \ar[r]^-{\pi^\sharp_V} \ar[d]_{\unif^\natural_V} & V(\C)/F^0V(\C) \ar[d]^{\unif_V}, & \tilde{K}^\natural \ar@{|->}[r] \ar@{|->}[d] & \tilde{K}^\natural_V \ar@{|->}[r] \ar@{|->}[d] & \tilde{K}_V \ar@{|->}[d] \\
E^\natural \ar[r] & A^\natural \ar[r] & A & K^\natural \ar@{|->}[r] & K^\natural_V \ar@{|->}[r] & K_V
}
\end{equation}
Fix an isomorphism $\Gamma_V \cong \Z^{2n}$, which induces an isomorphism $V(\C)/F^0V(\C) \cong \R^{2n}$. Let $\mathfrak{F}_V := (-1,1)^{2n}$. Then $\mathfrak{F}_V$ is a fundamental set for $\unif_V$. Let $\mathfrak{F}^\natural_V := (\pi^{\sharp}_V)^*\mathfrak{F}_V$. Then $\mathfrak{F}^\natural_V$ is a fundamental set for $\unif^\natural_V$.

As algebraic varieties we have $W(\C) \cong U(\C) \times V(\C)$. Fix an isomorphism $\Gamma_U \cong \Z^m$, which induces an isomorphism $U(\C) \cong \C^m$. Hence $W(\C) \cong \C^m \times V(\C)$. Let 
\[
\mathfrak{F}^\natural = \{z=(z_i) \in \C^m : |\mathrm{Re}(z_i)| < 1\} \times \mathfrak{F}^\natural_V.
\]
Then $\mathfrak{F}^\natural$ is a fundamental set for $\unif^\natural$.

Let $\tilde{Z}_V^\natural = \pi^\natural(\tilde{Z}^\natural)$. Let $U_0 = W_0 \cap U$ and $V_0 = W_0/U_0$.

\begin{prop}\label{PropositionUPartIsWeaklySpecial}
Let $v^\natural \in \tilde{Z}_V^\natural$, and take any irreducible component $\tilde{K}^\natural$ of the fiber $(\tilde{Z}^\natural)_{v^{\natural}}$. Then the fiber $\tilde{K}^\natural = U'(\C)+\tilde{z}^\natural$ for some $U' < U$ and $\tilde{z}^\natural \in \tilde{Z}^\natural_{v^{\natural}}$.
\begin{proof}
Define
\[
\Xi^\natural_U:=\{u\in U(\R)|\dim \left( (u + \tilde{K}^\natural) \cap \unif^{\natural -1}(Y^\natural) \cap \mathfrak{F}^\natural \right)=\dim \tilde{Z}^{\natural}_{v^\natural}\}.
\]
Then $\Xi^\natural_U$ is definable and
\begin{equation}\label{EquationXiCapUEqualsXiUcapU}
\Xi^\natural_U\cap\Gamma_U=\{\gamma_U\in\Gamma_U|(\gamma_U+\mathfrak{F}^\natural_{v^\natural})\cap \tilde{K}^\natural \neq\emptyset\}.
\end{equation}
For any integer $M>0$, let
\[
\Theta^\natural_U(\tilde{K}^\natural,M):=\{\gamma_U\in\Gamma_U\cap\Xi^\natural_U|H(\gamma_U)\leq  M\}.
\]
\begin{claim}\label{ClaimThetaULargeEnough}
Assume $\dim \tilde{K}^\natural >0$. Then $|\Theta^\natural_U(\tilde{K}^\natural,M)|\gg M$.
\begin{proof} Define the norm of $x=(x_1,...,x_m)\in U(\C)$ to be $\parallel x \parallel:=\max(|x_1|,...,|x_m|)$. It is clear that there exists a constant $c_1>0$ such that
\begin{equation}\label{EquationLatticeHeightSmallEnoughForU}
H(\gamma_U)\leq  c_1 \parallel x \parallel
\end{equation}
for any $x\in U(\C)$ and any $\gamma_U\in\Gamma_U$ such that $\gamma_U+x\in\mathfrak{F}^\natural_{v^\natural}$. Let $\omega:=\bigwedge^{\dim \tilde{K}^\natural}(dx_1\wedge d\bar{x}_1+...+dx_m\wedge d\bar{x}_m)$ and let $\tilde{K}^\natural(M):=\{\tilde{u}\in\tilde{K}^\natural|\parallel \tilde{u} \parallel \leq  M\}$. Then there exists a constant $c_2^\prime$ such that
\begin{align*}
\int_{\tilde{K}^\natural(M)\cap(\gamma_U+\mathfrak{F}^\natural_{v^\natural})}\omega =\int_{(\tilde{K}^\natural(M)-\gamma_U)\cap\mathfrak{F}^\natural_{v^\natural}}\omega &\leq \sum_{I\subset\{1,...,m\}, |I|=\dim \tilde{K}^\natural}\deg(p_I|_{(\tilde{K}^\natural-\gamma_U)\cap\mathfrak{F}^\natural_{v^\natural}})\int_{p_I(\{x\in\mathfrak{F}^\natural_{v^\natural}|\text{Im}(x)\leq  M\})}\omega \\
&=c_2^\prime M\sum_{I\subset\{1,...,m\}, |I|=\dim \tilde{K}^\natural}\deg(p_I|_{(\tilde{K}^\natural -\gamma_U)\cap\mathfrak{F}^\natural_{v^\natural}}).
\end{align*}
The function $\Xi^\natural_U \rightarrow\Z$, $u\mapsto \deg(p_I|_{(\tilde{K}^\natural -\gamma_U)\cap\mathfrak{F}^\natural_{v^\natural}})$ is a definable function with value in $\Z$, and hence is uniformly bounded. Therefore there exists a constant $c_2$ such that
\begin{equation}\label{EquationVolumeSmallEnoughForU}
\int_{\tilde{K}^\natural(M)\cap(\gamma_U+\mathfrak{F}^\natural_{v^\natural})}\omega \leq  c_2M.
\end{equation}
By \cite[Theorem~0.1]{HwangVolumes-of-comp} there exists a constant $c_3>0$ such that
\begin{equation}\label{EquationVolumeLargeEnoughForU}
\int_{\tilde{K}^\natural(M)}\omega\ge  c_3M^{\dim \tilde{K}^\natural}\ge  c_3M^2.
\end{equation}
Now we have
\begin{align*}
\tilde{K}^\natural(M) &=\bigcup_{\substack{\gamma_U\in\Gamma_U \\ (\gamma_U+\mathfrak{F}^\natural_{v^\natural})\cap \tilde{K}^\natural \neq\emptyset}}(\gamma_U+\mathfrak{F}^\natural_{v^\natural})\cap \tilde{K}^\natural(M) &\subset \bigcup_{\gamma_U\in\Theta^\natural_U(\tilde{K}^\natural,2M)}(\gamma_U+\mathfrak{F}^\natural_{v^\natural})\cap \tilde{K}^\natural(M)
\end{align*}
by \eqref{EquationLatticeHeightSmallEnoughForU} and $H(\gamma_U)\leq  M\Rightarrow \gamma_U+\tilde{K}^\natural(M)\subset(\gamma_U+\tilde{K}^\natural)(2M)$. Integrating both sides with respect to $\omega$ we can conclude by \eqref{EquationVolumeSmallEnoughForU} and \eqref{EquationVolumeLargeEnoughForU}.
\end{proof}
\end{claim}

Assume $\dim \tilde{K}^\natural >0$. By Pila-Wilkie \cite[Theorem~3.6]{PilaO-minimality-an}, there exists a positive dimensional connected semi-algebraic subset $K\subset \Xi^\natural_U$ containing arbitrarily many points $\gamma_U\in\Gamma_U$. Fix a $\sigma_U\in K\cap\Gamma_U$. We have $K\subset\stab_{W(\R)}(\tilde{K}^\natural)\sigma_U$. So $(K\cap\Gamma_U)\sigma_U^{-1}\subset\stab_{W(\R)}(\tilde{K}^\natural)\cap\Gamma_U$. Therefore the $\Q$-group
\[
U^\prime_{v^\natural}:=(\stab_{W(\R)}(\tilde{K}^\natural)\cap\Gamma_U)^{\Zar},
\]
which is the largest subgroup of $U$ such that $\tilde{K}^\natural$ is stable under $U^\prime(\C)$, is of positive dimension. This group is normal in $P$ by Proposition~\ref{PropositionEMSDWithGenericMTGroupAlwaysExists}(2), so we can consider $\lambda_{v^\natural}\colon(P,\cX^{\natural+})\rightarrow(P,\cX^{\natural+})/U^\prime_{v^\natural}$. Note that $\tilde{K}^\natural = \lambda_{v^\natural}^{-1}(\lambda_{v^\natural}(\tilde{K}^\natural))$.
\end{proof}
\end{prop}

\begin{cor}\label{CorollaryFiberOverVpartMSV}
For the quotient $\lambda \colon (P,\cX^{\natural+}) \rightarrow (P,\cX^{\natural+})/U_0$, we have that each fiber of $\lambda(\tilde{Z}^\natural)$ over $\tilde{Z}^\natural_V$ is a point.
\begin{proof} 
It suffices to prove the following statement: There exists a subgroup $U' < U$ such that the following properties hold.
\begin{enumerate}
\item[(i)] The subvariety $\tilde{Z}^\natural$ is stable under $U'(\C)$.
\item[(ii)] For the quotient $\lambda \colon (P,\cX^{\natural+}) \rightarrow (P,\cX^{\natural+})/U'$, we have that each fiber of $\lambda(\tilde{Z}^\natural)$ over $\tilde{Z}^\natural_V$ is a point.
\end{enumerate}

For any $v^\natural\in \tilde{Z}^\natural_V$, recall our assumption that $\tilde{Z}^{\natural}$ is algebraic. So there exists an integer $n > 0$ such that $\tilde{Z}_{v^\natural}$ has $\le n$ irreducible components for each $v^\natural$. Generic flatness of algebraic varieties says that $\dim \tilde{Z}^\natural = \dim \tilde{Z}^\natural_V + \dim \tilde{Z}^\natural_{v^\natural}$ for a generic $v^\natural \in \tilde{Z}^\natural_V$. We have that $\dim \tilde{Z}^\natural_{v^\natural}$ equals the dimension of one of its irreducible components, and these irreducible components can be chosen such that their union is connected. Apply Proposition~\ref{PropositionUPartIsWeaklySpecial} to each $v^\natural$ and the irreducible component of $\tilde{Z}^\natural_{v^\natural}$ chosen above, we obtain a subgroup $U'_{v^\natural}$ 
of $U$.

Now let $\tilde{C}^\natural_V$ be any complex analytic irreducible curve in $\tilde{Z}^\natural_V$. There are uncountably many points in $\tilde{C}^\natural_V$ and there are only countably many subgroups of $U$, hence there exists a subgroup $U^\prime$ of $U$ such that $U^\prime_{v^\natural} = U^\prime$ for uncountably many $v^\natural \in\tilde{C}^\natural_V$. Replacing $(P,\cX^{\natural+})$ by $(P,\cX^{\natural+})/U^\prime$ (this can be done since $U^\prime\lhd P$ by Proposition~\ref{PropositionEMSDWithGenericMTGroupAlwaysExists}(2)) and every other object accordingly, we see that $U^\prime_{v^\natural, i}\subset U^\prime$ for all $i$ except for countably many $v^\natural \in\tilde{C}^\natural_V$ by generic flatness (in the complex analytic geometry sense) and dimension reasons. But $\dim U^\prime_{v^\natural}$ is constant except for countably many $v^\natural \in\tilde{C}^\natural_V$ by generic flatness (in the complex analytic geometry sense) and dimension reasons. Hence $U^\prime_{v^\natural}=U^\prime$ except for countably many $v^\natural \in\tilde{C}^\natural_V$. But then we must have $U^\prime_{v^\natural}=U^\prime$ for all $v^\natural \in\tilde{C}^\natural_V$ (this can be seen by, for example, considering the smallest complex analytic irreducible subvariety containing $\bigcup_{v^\natural \in\tilde{C}^\natural_V,~U^\prime_{v^\natural}=U^\prime}\tilde{Z}^\natural_{v^\natural}$). By varying $\tilde{C}^\natural_V$ we get $U^\prime_{v^\natural}=U^\prime$ for all $v^\natural\in\tilde{Z}^\natural_V$.

Hence we are done since $\tilde{Z}^\natural$ is irreducible.
\end{proof}
\end{cor}

Now we replace $(P,\cX^{\natural+})$ by $(P,\cX^{\natural+})/U_0$ and $S^\natural$ accordingly. Hence we may and do assume that every fiber of $\tilde{Z}^\natural \rightarrow \tilde{Z}^\natural_V$ is a point and that $U_0=1$. Thus the natural projection $W_0 \rightarrow V_0$ is an isomorphism. But we will still distinguish $W_0$ and $V_0$ to make precise which group we are considering.

\begin{prop}\label{PropositionStabilizerVpartBigForEnatural}
We have $\tilde{Z}^\natural_V = V_N(\C) + v_0^\natural + L^\natural$ for some $v^\natural \in V(\C)$ and $L^\natural \subset F^0V(\C)$. In particular $V_0 = V_N$.
\begin{proof}
For any algebraic subvariety $\tilde{K}^\natural$ of $W(\C)$, define
Define
\[
\Xi^\natural(\tilde{K}^\natural):=\{v\in W(\R): \dim \left( w\tilde{K}^\natural \cap \unif^{\natural-1}(Y^\natural) \cap \mathfrak{F}^\natural \right)=\dim \tilde{K}^\natural\}.
\]
Then we have
\begin{equation}\label{EqXiTwoWritingForIntegralPoints}
\Xi^\natural(\tilde{K}^\natural) \cap\Gamma = \{\gamma \in\Gamma : \gamma \mathfrak{F}^\natural\cap \tilde{K}^\natural \neq\emptyset\}.
\end{equation}
Now consider $\tilde{Z}^\natural$. Let $\tilde{C}^\natural$ be any algebraic curve in $\tilde{Z}^\natural$. Then
\[
\Xi^\natural(\tilde{C}^\natural) \cap\Gamma \subset \Xi^\natural(\tilde{Z}^\natural) \cap\Gamma
\]
by \eqref{EqXiTwoWritingForIntegralPoints}.

For any integer $M>0$, let
\[
\Theta^{\natural}(\tilde{C}^\natural,M)=\{\gamma\in\Gamma\cap\Xi^{\natural}(\tilde{C}^\natural) : H(\gamma)\leq  M\}.
\]

\begin{lemma}\label{LemmaEstimatesOnThetaVNatural}
If $\dim \pi_V^\sharp(\tilde{C}^\natural_V) > 0$, then $|\Theta^{\natural}(\tilde{C}^\natural,M)|\gg M$.
\begin{proof} 
This follows from the choice of $\mathfrak{F}_V^\natural$, the fact that $V$ is a unipotent group, and the assumption that $U_0 = 1$.
\end{proof}
\end{lemma}

Let us fix a decomposition $V = V_0 \bigoplus V_2$ of $\Q$-groups. Then we have $\tilde{Z}^\natural_V = V_0(\C) \times \tilde{Z}^\natural_{V,2}$. If $\pi^\sharp(\tilde{Z}^\natural_{V,2}) > 0$, then we can take an algebraic curve $\tilde{C}^\natural \subset \tilde{Z}^\natural$ such that $\tilde{C}^\natural_V = \{\text{pt}\} \times \tilde{C}^\natural_2$ and $\dim \pi^\sharp(\tilde{C}^\natural_V) = 1$. Then again by Lemma~\ref{LemmaEstimatesOnThetaVNatural} and Pila-Wilkie, we have that $\pi^\natural(W_0)$ contains a positive dimensional subgroup of $V_2$ because $W_0$ is the largest subgroup of $W$ that stabilizes $\tilde{Z}^\natural$. But this contradicts $V_2 \cap V_0 = 0$. So $\tilde{Z}^\natural_{V,2} \subset F^0V(\C)$. So we have $\tilde{Z}_V^\natural = V_0(\C) + v_0^\natural + L^\natural$ for some $v_0^\natural \in V_N(\C)$ and $L^\natural \subset F^0V(\C)$.

Denote by $\tilde{Z} = \pi^\sharp(\tilde{Z}^\natural)$, $Y = [\pi^\sharp](Y^\natural)$ and the notations in \eqref{EqNotationFiberBigDiag}.

Next by Remark~\ref{RmkEMSVSMVAL}, we have $\tilde{Z}_V = V_0(\R) + v$. But $\tilde{Z}_V$ is complex analytic, so $V_0(\R)$ is complex in $V(\R) \cong V(\C)/F^0V(\C)$. Hence $V_0$ is a $G_Q$-module. Then $V_0$ gives rise to an abelian subvariety $A_0 := \mathrm{unif}_V(V_0(\R))$ of $A$.

Thus $Y_V = \mathrm{unif}_V(\tilde{Z}_V)^{\mathrm{Zar}} = \mathrm{unif}_V(\tilde{Z}_V)$ is the translate of $A_0$ by a point. Hence $V_N = V_0$.
\end{proof}
\end{prop}

Now let us finish the proof of Theorem~\ref{ThmALFiber}. Recall that each fiber of $\tilde{Z}^\natural \rightarrow \tilde{Z}^\natural_V$ is a point. We will use $A_0^\natural$ to denote the universal vector extension of $A_0$ as a subvariety of $A^\natural$.

We start with the case $V_N = 0$. In this case we have that $\mathrm{unif}_V^\natural|_{L^\natural}$ is the identity map. So $Z^\natural$ is closed in the usual topology. It is complex analytic since $\tilde{Z}^\natural$ is. So $Z^\natural$ is a holomorphic section of the $T$-torsor $E^\natural|_{a_0^\natural+L^\natural} \rightarrow a_0^\natural+L^\natural$. Thus this torsor is trivial. So $\tilde{Z}^\natural$ is the image of an algebraic map $\rho \colon L^\natural \rightarrow U(\C) \cong \mathbb{G}_a^k(\C)$, and $Z^\natural$ is the image of the composite of $\rho$ with $\exp \colon U(\C) \cong \mathbb{G}_a^k(\C) \rightarrow T(\C) \cong \mathbb{G}_m^k(\C)$. Now Ax-Lindemann for algebraic tori implies that $((\exp\circ\rho)(L^\natural))^{\mathrm{Zar}}$ is bi-algebraic. Thus $Y^\natural = (Z^\natural)^{\Zar}$ is bi-algebraic. Hence by the maximality of $\tilde{Z}^\natural$, we have that $\tilde{Z}^\natural$ is bi-algebraic. In particular $Z^\natural$ is a constant section.

Let us go back to the general case. For each $l^\natural \in L^\natural$, consider
\[
\tilde{Z}^\natural_{l^\natural} := \tilde{Z}^\natural \cap (\pi^\natural)^{-1}(V_N(\C)+v_0^\natural+l^\natural).
\]
Recall that each fiber of $\tilde{Z}^\natural \rightarrow \tilde{Z}^\natural_V$ is a point. Moreover $\tilde{Z}^\natural$ is stable under $V_N(\C)$.\footnote{Here we identify $V_N = V_0$ with $W_0 \cong V_0$.} Hence Proposition~\ref{PropositionStabilizerVpartBigForEnatural} implies that $\tilde{Z}^\natural_{l^\natural}$ is a $V_N(\C)$-orbit. As $V_N$ is a $\Q$-group, we have that $\mathrm{unif}^\natural(\tilde{Z}^\natural_{l^\natural})$ is closed in the usual topology. But then it is a holomorphic section of the $T$-torsor $E^\natural|_{A_0^\natural+a_0^\natural+l^\natural} \rightarrow A_0^\natural+a_0^\natural+l^\natural$, thus making the $T$-torsor trivial. On the other hand $\tilde{Z}_{l^\natural}:=\pi^\sharp(\tilde{Z}^\natural_{l^\natural})$ is a $V_N(\R)$-orbit, and hence $\mathrm{unif}(\tilde{Z}_{l^\natural})$ is a holomorphic section of the $T$-torsor $E|_{A_0+a_0} \rightarrow A_0 + a_0$. Thus $\mathrm{unif}(\tilde{Z}_{l^\natural})$ gives rise to a holomorphic morphism $A_0+a_0 \rightarrow T$, which must be trivial. So $\mathrm{unif}(\tilde{Z}_{l^\natural})$ is a constant section. Back to $\tilde{Z}^\natural_{l^\natural}$, we then have that $\mathrm{unif}^\natural(\tilde{Z}^\natural_{l^\natural})$ is a constant section. 

This argument applies to all $l^\natural \in L^\natural$. So we can take the quotient of the $T$-torsor $E^\natural|_{A_0^\natural+a_0^\natural+L^\natural} \rightarrow A_0^\natural+a_0^\natural+L^\natural$ by $A_0^\natural$, namely we have the following Cartesian diagram
\[
\xymatrix{
E^\natural|_{A_0^\natural+a_0^\natural+L^\natural} \ar[r]^-{[p]} \ar[d] & \bar{E}^\natural \ar[d] \\
A_0^\natural + a_0^\natural + L^\natural \ar[r] & L^\natural
}
\]
with the right morphism being a $T$-torsor, and the bottom map being the natural projection. Thus we can conclude by the previous paragraph and by the case $V_N = 0$.

\section{The Ax-Schanuel conjecture}\label{SectionAxSchanuel}

\subsection{Formulation of the conjecture}
For mixed Shimura varieties we have the following Ax-Schanuel conjecture: let $S$ be a connected mixed Shimura variety associated with $(P,\cX^+)$ and let $\unif\colon\cX^+\rightarrow S$ be the uniformization. Let $pr_{\cX^+}\colon\cX^+\times S\rightarrow\cX^+$ and $pr_S\colon\cX^+\times S\rightarrow S$ be the natural projections.

\begin{conj}[Ax-Schanuel for mixed Shimura varieties]\label{ConjectureAxSchanuelMSV}
Let $\Delta\subset\cX^+\times S$ be the graph of $\unif$. Let $\mathbscr{Z}=\mathrm{graph}(\tilde{Z}\xrightarrow{\unif}Z)$ be a complex analytic irreducible subvariety of $\Delta$ and let $\mathbscr{B}$ be its Zariski closure in $\cX^+\times S$. Let $F$ be the smallest bi-algebraic (\textit{i.e.}  weakly special) subvariety of $S$ which contains $Z=pr_S(\mathbscr{Z})$. Then
\[
\dim\mathbscr{B}-\dim\mathbscr{Z}\ge \dim F.
\]
\end{conj}

Note that $Z$ is \textit{a priori} a disastrous set. But $F$ is well-defined: it is the smallest bi-algebraic subvariety of $S$ which contains $Z^{\Zar}$.

We explain this conjecture. First replace $\mathbscr{Z}$ by a complex analytic irreducible component of $\mathbscr{B}\cap\Delta$. Then denote by $\tilde{X}:=pr_{\cX^+}(\mathbscr{B})$ and $Y:=pr_S(\mathbscr{B})$. We have $\tilde{X} \subset \tilde{F}$, $Y \subset F$ and $\mathbscr{B}\subset\tilde{X}\times Y$. Therefore Conjecture~\ref{ConjectureAxSchanuelMSV} implies
\begin{equation}\label{EquationExplainAS1}
\dim\tilde{X}+\dim Y-\dim\tilde{Z}\ge \dim\mathbscr{B}-\dim\mathbscr{Z}\ge \dim F.
\end{equation}
On the other hand let $\tilde{Y}$ (resp. $\tilde{F})$ be the complex analytic irreducible component of $\unif^{-1}(Y)$ (resp. of $\unif^{-1}(F)$) containing $\tilde{Z}$, then $\tilde{Z}$ is a complex analytic irreducible component of $\tilde{X}\cap\tilde{Y}$ since $\mathbscr{Z}$ is a complex analytic irreducible component of $\mathbscr{B}\cap\Delta$. Hence we always have
\begin{equation}\label{EquationExplainAS2}
\dim\tilde{Z}\ge \dim\tilde{X}+\dim\tilde{Y}-\dim\tilde{F}.
\end{equation}
Now \eqref{EquationExplainAS1} and \eqref{EquationExplainAS2} together imply
\[
\dim\mathbscr{B}=\dim\tilde{X}+\dim Y,~\dim\tilde{Z}=\dim\tilde{X}+\dim\tilde{Y}-\dim\tilde{F},
\]
so Conjecture~\ref{ConjectureAxSchanuelMSV} is equivalent to:
\begin{itemize}
\item $\mathbscr{B}=\tilde{X}\times Y$;
\item $\tilde{X}$ and $\tilde{Y}$ intersect properly in $\tilde{F}$.
\end{itemize}
Moreover by the first bullet point above, $\mathbscr{B}$ being Zariski closed in $\cX^+ \times S$ implies that $\tilde{X}$ is Zariski closed in $\cX^+$ and $Y$ is Zariski closed in $S$. Hence $\tilde{X} = \tilde{Z}^{\Zar}$ and $Y = Z^{\Zar}$.

As we shall see, logarithmic Ax and Ax-Lindemann for mixed Shimura varieties are special cases of this conjecture. We wish to make a similar conjecture for enlarged mixed Shimura varieties. However the naive guess is NOT always true: if $\tilde{Z}^\natural$ is contained in a fiber of $\cX^{\natural+}\rightarrow\varphi(\cX^+)$, then $\unif^\natural|_{\tilde{Z}^\natural}$ is the identity map for some $\C^n$, and so $\mathbscr{Z}^\natural:=\mathrm{graph}(\tilde{Z}^\natural\xrightarrow{\unif^\natural}Z^\natural)$ is algebraic. So $(\mathbscr{Z}^\natural)^{\Zar}=\mathbscr{Z}^\natural\neq\tilde{Z}^\natural\times Z^\natural$ and $\dim(\mathbscr{Z}^\natural)^{\Zar}-\dim\mathbscr{Z}=0$.

Counterexamples to the naive guess as above exist due to the existence of those $F^\natural$ such that $(\Delta^\natural\cap(\tilde{F}^\natural\times F^\natural))^{\Zar}\neq\tilde{F}^\natural\times F^\natural$, and the existence of such $F^\natural$ is due to the non-linear part of $F^\natural$. However evidences suggest that this should be the only obstacle. So we make the following conjecture: let $S^\natural$ be a connected enlarged mixed Shimura variety associated with $(P,\cX^{\natural+})$ and let $\unif^\natural\colon\cX^{\natural+}\rightarrow S^\natural$ be the uniformization.

\begin{conj}[Ax-Schanuel for enlarged mixed Shimura varieties]\label{ConjectureAxSchanuel}
Let $\Delta^\natural\subset\cX^{\natural+}\times S^\natural$ be the graph of $\unif^\natural$. Let $\mathbscr{Z}^\natural=\mathrm{graph}(\tilde{Z}^\natural\xrightarrow{\unif^\natural}Z^\natural)$ be a complex analytic irreducible subvariety of $\Delta^\natural$. Let $F^\natural$ (resp. $\tilde{F}^\natural$) be the smallest bi-algebraic, or equivalently quasi-linear, subvariety of $S^\natural$ (resp. subset of $\cX^{\natural+}$) which contains $Z^\natural$ (resp. $\tilde{Z}^\natural$). In particular $F^\natural=\unif^\natural(\tilde{F}^\natural)$.
\begin{enumerate}
\item Let $\tilde{X}^\natural:=(\tilde{Z}^\natural)^{\Zar}$ and $Y^\natural:=(Z^\natural)^{\Zar}$. Then
\[
\dim\tilde{X}^\natural+\dim Y^\natural-\dim\tilde{Z}^\natural\ge \dim F^\natural.
\]
\item Let $\mathbscr{B}^\natural:=(\mathbscr{Z}^\natural)^{\Zar}\subset\cX^{\natural+}\times S^\natural$. Let $[pr]^{\mathrm{ws}}_{F^\natural}$ and $pr^{\mathrm{ws}}_{\tilde{F}^\natural}$ be as in \eqref{EquationWeaklySpecialAndQuasiLinear}. Denote for simplicity by $\mathbf{pr}^{\mathrm{ws}}_{F^\natural}:=(pr^{\mathrm{ws}}_{\tilde{F}^\natural},[pr]^{\mathrm{ws}}_{F^\natural})\colon\tilde{F}^\natural\times F^\natural\rightarrow(\tilde{F}^\natural)^{\mathrm{ws}}\times(F^\natural)^{\mathrm{ws}}$. Then
\[
\dim\mathbf{pr}^{\mathrm{ws}}_{F^\natural}(\mathbscr{B}^\natural)-\dim\mathbf{pr}^{\mathrm{ws}}_{F^\natural}(\mathbscr{Z}^\natural)\ge \dim(F^\natural)^{\mathrm{ws}}
\]
\end{enumerate}
\end{conj}
\begin{rmk}
The two parts of Conjecture~\ref{ConjectureAxSchanuel} do not imply each other: the first part implies logarithmic Ax (Theorem~\ref{TheoremAxlogarithmic}) and Ax-Lindemann (Theorem~\ref{AxLindemann}) while the second part does not, and the second part implies Conjecture~\ref{ConjectureAxSchanuelMSV} while the first part does not.
\end{rmk}

\subsection{Ax-Schanuel (Conjecture~\ref{ConjectureAxSchanuel}) implies logarithmic Ax (Theorem~\ref{TheoremAxlogarithmic})}
\begin{thm}\label{TheoremAxSchanuelAndAxlogarithmic}
\begin{enumerate}
\item The Ax-Schanuel conjecture implies the logarithmic Ax theorem.
\item The Ax-Schanuel conjecture is true if $Z^\natural$ is algebraic: it is implied by logarithmic Ax.
\end{enumerate}
\begin{proof}
For (1): Let $Z^\natural$ be an irreducible subvariety of $S^\natural$ and let $\tilde{Z}^\natural$ be a complex analytic irreducible component of $\unif^{\natural-1}(Z^\natural)$. We want to prove that $\tilde{X}^\natural=\tilde{Z}^{\natural,\Zar}$ is bi-algebraic assuming Conjecture~\ref{ConjectureAxSchanuel}. But the first part of Conjecture~\ref{ConjectureAxSchanuel} implies (since $Y^\natural=Z^\natural$)
\[
\dim\tilde{X}^\natural=\dim\tilde{X}^\natural+\dim Y^\natural-\dim\tilde{Z}^\natural\ge \dim\tilde{F}^\natural.
\]
Therefore $\tilde{X}^\natural=\tilde{F}^\natural$ is bi-algebraic since $\tilde{X}^\natural\subset\tilde{F}^\natural$.

For (2): logarithmic Ax implies $\tilde{X}^\natural=\tilde{F}^\natural$. Hence the first part is proven. Now it suffices to prove $\mathbf{pr}^{\mathrm{ws}}_{F^\natural}(\mathbscr{B}^\natural)=(\tilde{F}^\natural)^{\mathrm{ws}}\times [pr]^{\mathrm{ws}}_{F^\natural}(Z^\natural)$.

The group $P(\R)W(\C)$ acts on $\cX^{\natural+}\times S^\natural$ by its action on the first factor. Let $\Gamma_{Z^\natural}:=\im(\pi_1(Z^\natural)\rightarrow\pi_1(S^\natural)=\Gamma)$ be the monodromy group of $Z^\natural$ and let $N:=(\Gamma_{Z^\natural}^{\Zar})^{\circ}$. Then for any point $\tilde{z}^\natural\in\tilde{Z}^\natural$, we have $\Gamma_{Z^\natural}\cdot\tilde{z}^\natural\subset\mathbscr{Z}^\natural$. So $N(\R)^+W_N(\C)\tilde{z}^\natural\subset\mathbscr{B}^\natural$. So $\mathbf{pr}^{\mathrm{ws}}_{F^\natural}(\mathbscr{B}^\natural)_{[pr]^{\mathrm{ws}}_{F^\natural}(z^\natural)}:=\mathbf{pr}^{\mathrm{ws}}_{F^\natural}(\mathbscr{B}^\natural\cap pr_{S^\natural}^{-1}(z^\natural))$ is bi-algebraic, and hence equals $(\tilde{F}^\natural)^{\mathrm{ws}}$, for any $z^\natural\in Z^\natural$ by (the last paragraph of the proof of) logarithmic Ax. Now we are done.
\end{proof}
\end{thm}

\subsection{Ax-Schanuel (Conjecture~\ref{ConjectureAxSchanuel}) implies Ax-Lindemann (Theorem~\ref{AxLindemann})}
\begin{thm}\label{TheoremAxSchanuelAndAxLindemann}
\begin{enumerate}
\item The Ax-Schanuel conjecture implies the Ax-Lindemann theorem.
\item The Ax-Schanuel conjecture is true if $\tilde{Z}^\natural$ is algebraic: it is implied by Ax-Lindemann.
\end{enumerate}
\begin{proof}
For (1): Let $\tilde{Z}^\natural$ be an irreducible algebraic subset of $\cX^{\natural+}$ and let $Z^\natural:=\unif^\natural(\tilde{Z}^\natural)$. We want to prove that $Y^\natural=Z^{\natural,\Zar}$ is bi-algebraic. But the first part of Conjecture~\ref{ConjectureAxSchanuel} implies (since $\tilde{X}^\natural=\tilde{Z}^\natural$)
\[
\dim Y^\natural=\dim\tilde{X}^\natural+\dim Y^\natural-\dim\tilde{Z}^\natural\ge \dim F^\natural,
\]
and therefore $Y^\natural=F^\natural$ is bi-algebraic since $Y^\natural\subset F^\natural$.

For (2): Ax-Lindemann implies $Y^\natural=F^\natural$. Hence the first part is proven. Now it suffices to prove $\mathbf{pr}^{\mathrm{ws}}_{F^\natural}(\mathbscr{B}^\natural)=pr^{\mathrm{ws}}_{\tilde{F}^\natural}(\tilde{Z}^\natural)\times(F^\natural)^{\mathrm{ws}}$.

We start by the case where $pr^{\mathrm{ws}}_{\tilde{F}^\natural}(\tilde{Z}^\natural)$ is a curve. The Ax-Lindemann theorem implies $Z^{\natural,\Zar}=F^\natural$. So $\mathbf{pr}^{\mathrm{ws}}_{F^\natural}(\mathbscr{B}^\natural)\subset pr^{\mathrm{ws}}_{\tilde{F}^\natural}(\tilde{Z}^\natural)\times(F^\natural)^{\mathrm{ws}}$ and surjects to $pr^{\mathrm{ws}}_{\tilde{F}^\natural}(\tilde{Z}^\natural)$ and $(F^\natural)^{\mathrm{ws}}$. Therefore either $\mathbf{pr}^{\mathrm{ws}}_{F^\natural}(\mathbscr{B}^\natural)=pr^{\mathrm{ws}}_{\tilde{F}^\natural}(\tilde{Z}^\natural)\times(F^\natural)^{\mathrm{ws}}$ or $\mathbf{pr}^{\mathrm{ws}}_{F^\natural}(\mathbscr{B}^\natural)\rightarrow (F^\natural)^{\mathrm{ws}}$ is quasi-finite. Suppose that we are in the latter case. Up to replacing $S^\natural$ by a finite cover we may assume that $\mathbf{pr}^{\mathrm{ws}}_{F^\natural}(\mathbscr{B}^\natural)\rightarrow(F^\natural)^{\mathrm{ws}}$ is a bijection. Let $U^\natural$ be a non-empty open subvariety of $(F^\natural)^{\mathrm{ws}}$ such that $pr^{\mathrm{ws}}_{\tilde{F}^\natural}(\tilde{Z}^\natural)|_{U^\natural}:=pr^{\mathrm{ws}}_{\tilde{F}^\natural}(\tilde{Z}^\natural)\cap\unif^{\natural-1}(U^\natural)$ is smooth. Now we define $i\colon U^\natural\rightarrow pr^{\mathrm{ws}}_{\tilde{F}^\natural}(\tilde{Z}^\natural)|_{U^\natural}$, $s^\natural\mapsto\mathbf{pr}^{\mathrm{ws}}_{F^\natural}(\mathbscr{B}^\natural)_{s^\natural}$. This map is algebraic and bijective, and hence is an isomorphism of algebraic varieties since $pr^{\mathrm{ws}}_{\tilde{F}^\natural}(\tilde{Z}^\natural)|_{U^\natural}$ is smooth. But $\unif^\natural|_{pr^{\mathrm{ws}}_{\tilde{F}^\natural}(\tilde{Z}^\natural)|_{U^\natural}}$, which is not algebraic, is the inverse of $i$. This is a contradiction. Hence we must have $\mathbf{pr}^{\mathrm{ws}}_{F^\natural}(\mathbscr{B}^\natural)=pr^{\mathrm{ws}}_{\tilde{F}^\natural}(\tilde{Z}^\natural)\times(F^\natural)^{\mathrm{ws}}$ when $pr^{\mathrm{ws}}_{\tilde{F}^\natural}(\tilde{Z}^\natural)$ is a curve.

Now we turn to $pr^{\mathrm{ws}}_{\tilde{F}^\natural}(\tilde{Z}^\natural)$ of arbitrary dimension. Let $\tilde{z}^\natural\in pr^{\mathrm{ws}}_{\tilde{F}^\natural}(\tilde{Z}^\natural)$ and denote by $z^\natural:=\unif^\natural(\tilde{z}^\natural)$. Let $\tilde{C}^\natural$ be any algebraic curve in $pr^{\mathrm{ws}}_{\tilde{F}^\natural}(\tilde{Z}^\natural)$ passing through $\tilde{z}^\natural$. By the last paragraph we have $\tilde{C}^\natural\subset\mathbf{pr}^{\mathrm{ws}}_{F^\natural}(\mathbscr{B}^\natural)_{z^\natural}:=\mathbf{pr}^{\mathrm{ws}}_{F^\natural}(\mathbscr{B}^\natural\cap pr_{S^\natural}^{-1}(z^\natural))$. By varying $\tilde{C}^\natural$ we get $pr^{\mathrm{ws}}_{\tilde{F}^\natural}(\tilde{Z}^\natural)\subset\mathbf{pr}^{\mathrm{ws}}_{F^\natural}(\mathbscr{B}^\natural)_{z^\natural}$. By varying $\tilde{z}^\natural$ we get $pr^{\mathrm{ws}}_{\tilde{F}^\natural}(\tilde{Z}^\natural)\times [pr]^{\mathrm{ws}}_{F^\natural}(Z^\natural)\subset\mathbf{pr}^{\mathrm{ws}}_{F^\natural}(\mathbscr{B}^\natural)$. So
\[
pr^{\mathrm{ws}}_{\tilde{F}^\natural}(\tilde{Z}^\natural)\times(F^\natural)^{\mathrm{ws}}=(pr^{\mathrm{ws}}_{\tilde{F}^\natural}(\tilde{Z}^\natural)\times [pr]^{\mathrm{ws}}_{F^\natural}(Z^\natural))^{\Zar}\subset\mathbf{pr}^{\mathrm{ws}}_{F^\natural}(\mathbscr{B}^\natural)\subset pr^{\mathrm{ws}}_{\tilde{F}^\natural}(\tilde{Z}^\natural)\times(F^\natural)^{\mathrm{ws}}
\]
and we are done.
\end{proof}
\end{thm}

\section{Special subvarieties}\label{SectionSpecialSubvarieties}
Let $S^\natural$ be a connected enlarged mixed Shimura variety associated with $(P,\cX^{\natural+})$ and let $\unif^\natural\colon\cX^{\natural+}\rightarrow S=\Gamma\backslash\cX^{\natural+}$ be the uniformization. Use notation of $\mathsection$\ref{SubsectionNotationEMSD}.
\begin{defn}\label{DefinitionSpecialSubvariety}
\begin{enumerate}
\item A subset $\tilde{Y}^\natural$ of $\cX^{\natural+}$ is said to be \textbf{special} if it is the underlying space of some connected enlarged mixed Shimura subdatum $(Q,\cY^{\natural+})\hookrightarrow(P,\cX^{\natural+})$ (in the category \underline{$\mathbf{\cE\cM\cS\cD}$}). \textbf{Special points} of $\cX^{\natural+}$ are precisely the special subsets of dimension $0$.
\item A subset $Y^\natural$ of $S^\natural$ is called a \textbf{special subvariety} if it is the image of some subset of $\cX^{\natural+}$ under $\unif^\natural$ (Proposition~\ref{PropositionMorphismsBetweenEMSV}(1) implies that any such defined subset of $S^\natural$ is a closed algebraic subvariety of $S^\natural$). \textbf{Special points} of $S^\natural$ are precisely the special subvarieties of dimension $0$.
\end{enumerate}
\end{defn}

\begin{rmk}\label{RemarkSpecialPointsBijectionEMSDandMSD}
The definition of morphisms \underline{$\mathbf{\cE\cM\cS\cD}$} (Definition~\ref{DefinitionMorphismsInTheCategoryOfEMSD}) yields the following direct corollary: under the geometric comparison $\pi^\sharp\colon\cX^{\natural+}\rightarrow\varphi(\cX^+)$ (see $\mathsection$\ref{SubsectionNotationEMSD}), there is a bijection
\[
\{\text{special subsets of }\cX^{\natural+}\}\xrightarrow{\sim}\{\text{special subsets of }\varphi(\cX^+)\},\quad\tilde{Y}^\natural\mapsto\tilde{Y}:=\pi^\sharp(\tilde{Y}^\natural),
\]
which induces $\{\text{special points of }\cX^{\natural+}\}\xrightarrow{\sim}\{\text{special points of }\varphi(\cX^+)\}$.
\end{rmk}

\begin{thm}\label{TheoremSpecialPointsZariskiDenseInSpecialSubvarieties}
Let $Y^\natural$ be a special subvariety of $S^\natural$. Then the set of special points contained in $Y^\natural$ is Zariski dense.
\begin{proof} Use Notation~\ref{NotationSubsetsOfEMSD}. Then $Y_G$ is a special subvariety of $S_G$ and, by the theory of Shimura varieties, the set of special points of $S_G$ contained in $Y_G$ is Zariski dense. Because $Y^\natural$ is special, we have that $Y^\natural$ is a torsor over the universal vector extension of the abelian scheme $Y_{P/U}/Y_G$, and hence the fibers of $Y^\natural\rightarrow Y_G$ are of constant dimension, say $n$.

Let $Z^\natural\subset Y^\natural$ be the Zariski closure of the set of special points of $S^\natural$ contained in $Y^\natural$. Let $y_G\in Y_G$ be a special point of $S_G$, then the set of special points of $S^\natural$ contained in the fiber $Y^\natural_{y_G}$ is Zariski dense in $Y^\natural_{y_G}$: to see this it suffices to prove that the set of special points of $S^\natural_{P/U}$ contained in $Y^\natural_{P/U,y_G}$ is Zariski dense, which is true because $Y^\natural_{P/U,y_G}$ is the universal vector extension of the abelian variety $Y_{P/U,y_G}$. Hence $\dim Z^\natural_{y_G}=\dim Y^\natural_{y_G}=n$ for any special point $y_G\in Y_G$. However the semi-continuity implies that the set
\[
\Sigma:=\{y_G\in Y_G|\dim Z^\natural_{y_G}\ge  n\},
\]
which contains all special points of $Y_G$ by the argument above, is Zariski closed in $Y_G$. Hence $\Sigma=Y_G$. Therefore $Z^\natural=Y^\natural$.
\end{proof}
\end{thm}

The converse of Theorem~\ref{TheoremSpecialPointsZariskiDenseInSpecialSubvarieties} is an ``Andr\'{e}-Oort'' type conjecture.
\begin{conj}[Andr\'{e}-Oort for enlarged mixed Shimura varieties]\label{ConjectureAOforEMSV}
Let $Y^\natural$ be an irreducible subvariety of a connected enlarged mixed Shimura variety $S^\natural$. If the set of special points of $S^\natural$ contained in $Y^\natural$ is Zariski dense in $Y^\natural$, then $Y^\natural$ is a special subvariety.
\end{conj}

Conjecture~\ref{ConjectureAOforEMSV} implies directly the mixed Andr\'{e}-Oort conjecture:
\begin{conj}[mixed Andr\'{e}-Oort]\label{ConjectureAOmixed}
Let $Y$ be an irreducible subvariety of a connected mixed Shimura variety $S$. If the set of special points of $S$ contained in $Y$ is Zariski dense in $Y$, then $Y$ is a special subvariety of $S$.
\end{conj}

But in fact, we have
\begin{lemma}\label{LemmaAOEquivalentForEMSVandMSV}
Conjecture~\ref{ConjectureAOforEMSV} and Conjecture~\ref{ConjectureAOmixed} are equivalent.
\begin{proof} We only need to prove Conjecture~\ref{ConjectureAOforEMSV} assuming Conjecture~\ref{ConjectureAOmixed}. Let $Y^\natural$ be as in Conjecture~\ref{ConjectureAOforEMSV} and let $\Sigma^\natural$ be the set of special points of $S^\natural$ contained in $Y^\natural$. Use Notation~\ref{NotationSubsetsOfEMSD}. Then $\Sigma:=[\pi^\sharp](\Sigma)$ is a set of special points of $S$ such that $\Sigma^{\mathrm{Zar}}=Y$. Hence $Y$ is a special subvariety of $S$ by Conjecture~\ref{ConjectureAOmixed}.

Let $Z^\natural$ be the special subvariety of $S^\natural$ which maps to $Y$ under the bijection
\[
\{\text{special subvarieties of }S^\natural\}\xrightarrow{\sim}\{\text{special subvarieties of }S\}
\]
given by Remark~\ref{RemarkSpecialPointsBijectionEMSDandMSD}. But $[\pi^\sharp] \colon \Sigma^\natural \rightarrow \Sigma$ is a bijection by Remark~\ref{RemarkSpecialPointsBijectionEMSDandMSD}. So the set of special points of $S^\natural$ contained in $Z^\natural$ is $\Sigma^\natural$. Now Theorem~\ref{TheoremSpecialPointsZariskiDenseInSpecialSubvarieties} implies that $(\Sigma^\natural)^{\mathrm{Zar}}=Z^\natural$. Hence $Y^\natural=Z^\natural$ is special.
\end{proof}
\end{lemma}

\begin{rmk}\label{RemarkWhyTheTwoAOareEquivalent}
A key point to prove Lemma~\ref{LemmaAOEquivalentForEMSVandMSV} is that the bijection in Remark~\ref{RemarkSpecialPointsBijectionEMSDandMSD} preserves $0$-dimensional special subvarieties. In general this is not true: an $r$-dimensional special subvariety of $S^\natural$ is often mapped to a special subvariety of $S$ of smaller dimension.
\end{rmk}

Conjecture~\ref{ConjectureAOmixed} has been intensively studied. It is proven unconditionally for any mixed Shimura variety of abelian type (\textit{i.e.}  its pure part is of abelian type) by Pila-Tsimerman \cite{PilaAxLindemannAg}, Tsimerman \cite{TsimermanA-proof-of-the-} and Gao \cite{GaoTowards-the-And}. It is also proven under the generalized Riemann Hypothesis (for CM fields) for all mixed Shimura varieties by Klingler-Ullmo-Yafaev \cite{KlinglerThe-Hyperbolic-}, Ullmo \cite{UllmoQuelques-applic}, Ullmo-Yafaev \cite{UllmoNombre-de-class}, Daw-Orr \cite{DawHeights-of-pre-} and Gao \cite{GaoAbout-the-mixed}. Then Lemma~\ref{LemmaAOEquivalentForEMSVandMSV} implies the corresponding results for enlarged mixed Shimura varieties.

Taking into account of not only special points but also special subvarieties of higher dimensions, we have a ``Zilber-Pink'' type conjecture.
\begin{conj}[Zilber-Pink for enlarged mixed Shimura varieties]\label{ConjectureZPforEMSV}
Let $Y^\natural$ be an irreducible subvariety of a connected enlarged mixed Shimura variety $S^\natural$. Assume that $S^\natural$ has no proper special subvariety which contains $Y^\natural$. Then
\[
\bigcup_{\substack{S^{\natural\prime}\text{ special}, \\ \dim(S^{\natural\prime})<\codim(Y^\natural)}}S^{\natural\prime}\cap Y^\natural
\]
is not Zariski dense in $Y^\natural$.
\end{conj}

The Zilber-Pink conjecture is naturally a generalization of the Andr\'{e}-Oort conjecture. Unlike the Andr\'{e}-Oort conjecture, there is no immediate equivalence (as given by Lemma~\ref{LemmaAOEquivalentForEMSVandMSV}) between Conjecture~\ref{ConjectureZPforEMSV} and the Zilber-Pink conjecture for mixed Shimura varieties (see Remark~\ref{RemarkWhyTheTwoAOareEquivalent}). There are some results about this conjecture in some special cases (Habegger-Pila \cite{HabeggerSome-unlikely-i, HabeggerO-minimality-an}, Orr \cite{OrrFamilies-of-abe}, Gao \cite{GaoA-special-point}), but in general little is known.

We finish this paper by the following lemma, which explains the relation between special and weakly special subvarieties (Definition~\ref{DefinitionWeaklySpecialSubvariety}).

\begin{lemma}\label{LemmaCharacterizationOfSpecialSubsets}
A subset $\tilde{Y}^\natural$ of $\cX^{\natural+}$ is special if and only if $\tilde{Y}^\natural$ is weakly special and contains a special point. Equivalently, a subvariety $Y^\natural$ of $S^\natural$ is special if and only if it is weakly special and contains a special point.
\begin{proof} The ``only if'' part is clear. We prove the ``if'' part. Let $(Q,\cY^{\natural+})$ be as in Definition \ref{DefinitionWeaklySpecialSubvariety}(1) defining $\tilde{Y}^\natural$ and let $N:=\ker\varphi$. Let $(T,\tilde{y}^\natural)$ be the connected enlarged mixed Shimura datum giving a special point contained in $\tilde{Y}^\natural$. Let $Q^\prime$ be the subgroup of $P$ generated by $i(N)$ and $T$. Then $(Q^\prime,Q^\prime(\R)^+W_{Q^\prime}(\C)\tilde{y}^\natural)$ defines a connected enlarged mixed Shimura subdatum of $(P,\cX^{\natural+})$ (in the category \underline{$\mathbf{\cE\cM\cS\cD}$}), and $\tilde{Y}^\natural=Q^\prime(\R)^+W_{Q^\prime}(\C)\tilde{y}^\natural$. Hence $\tilde{Y}^\natural$ is special.
\end{proof}
\end{lemma}



\begin{thebibliography}{10}

\bibitem{AndreMumford-Tate-gr}
Y.~Andr\'{e}.
\newblock {M}umford-{T}ate groups of mixed {H}odge structures and the theorem
  of the fixed part.
\newblock {\em Compos. Math.}, 82(1):1--24, 1992.

\bibitem{AxOn-Schanuels-co}
J.~Ax.
\newblock On {S}chanuel's conjectures.
\newblock {\em Annals Math.}, 93:252--268, 1971.

\bibitem{AxSome-topics-in-}
J.~Ax.
\newblock Some topics in differential algebraic geometry {I}: Analytic
  subgroups of algebraic groups.
\newblock {\em American Journal of Mathematics}, 94:1195--1204, 1972.

\bibitem{Barbieri-VialeSharp-de-Rham-r}
B.~Barbieri-Viale and A.~Bertapelle.
\newblock Sharp {d}e {R}ham realization.
\newblock {\em Advances in Math.}, 222:1308--1338, 2009.

\bibitem{BertrandRelative-Manin-}
D.~Bertrand, D.~Masser, A.~Pillay, and U.~Zannier.
\newblock Relative {M}anin-{M}umford for semi-abelian surfaces.
\newblock {\em Proceedings of the Edinburgh Mathematical Society}, page online,
  2016.

\bibitem{BertrandA-Lindemann-Wei}
D.~Bertrand and A.~Pillay.
\newblock A {L}indemann-{W}eierstrass theorem for semi-abelian varieties over
  function fields.
\newblock {\em J.Amer.Math.Soc.}, 23(2):491--533, 2010.

\bibitem{BertrandGalois-theory-f}
D.~Bertrand and A.~Pillay.
\newblock {G}alois theory, function {L}indemann-{W}eierstrass, and {M}anin
  maps.
\newblock {\em Pacific J. of Math.}, 281:51--82, 2016.

\bibitem{Buiumdifferential-Al}
A.~Buium.
\newblock {\em Differential Algebraic Groups of Finite Dimension}.
\newblock Number 1506 in LNM. Springer-Verlag, 1992.

\bibitem{BuiumEffective-bound}
A.~Buium.
\newblock Effective bound for the geometric {L}ang conjecture.
\newblock {\em Duke Math. J.}, 71:475--499, 1993.

\bibitem{BuiumGeometry-of-dif}
A.~Buium.
\newblock Geometry of differential polynomial functions {I}: algebraic groups.
\newblock {\em American Journal of Mathematics}, 115:1385--1444, 1993.

\bibitem{CohenHumbert-surface}
P.~Cohen.
\newblock Humbert surfaces and transcendence properties of automorphic
  functions.
\newblock {\em Rocky Mountain J. Math.}, 26:987--1001, 1996.

\bibitem{ColemanThe-universal-v}
R.~Coleman.
\newblock The universal vectorial bi-extension and $p$-adic heights.
\newblock {\em Inv. Math.}, 103:631--650, 1991.

\bibitem{DawHeights-of-pre-}
C.~Daw and M.~Orr.
\newblock Heights of pre-special points of {S}himura varieties.
\newblock {\em Math. Annalen}, online.

\bibitem{GaoThe-mixed-Ax-Li}
Z.~Gao.
\newblock {\em Le th\'{e}or\`{e}me d'{A}x-{L}indemann et ses applications \`{a}
  la conjecture de {Z}ilber-{P}ink (The mixed {A}x-{L}indemann theorem and its
  applications to the {Z}ilber-{P}ink conjecture)}.
\newblock PhD thesis, Leiden University and Universit\'e Paris-Sud, 2014.

\bibitem{GaoAbout-the-mixed}
Z.~Gao.
\newblock About the mixed {A}ndr\'{e}-{O}ort conjecture: reduction to a lower
  bound for the pure case.
\newblock {\em Comptes rendus Mathematiques}, 354:659--663, 2016.

\bibitem{GaoTowards-the-And}
Z.~Gao.
\newblock Towards the {A}ndr\'{e}-{O}ort conjecture for mixed {S}himura
  varieties: the {A}x-{L}indemann-weierstrass theorem and lower bounds for
  {G}alois orbits of special points.
\newblock {\em J.Reine Angew. Math (Crelle)}, online, 2015.

\bibitem{GaoA-special-point}
Z.~Gao.
\newblock A special point problem of {A}ndr\'{e}-{P}ink-{Z}annier in the
  universal family of abelian varieties.
\newblock {\em Annali della Scuola Normale Superiore di Pisa, Classe di
  Scienze}, to appear.
  
  
\bibitem{GrothendieckNaturalExtension}
J.~Giraud, A.~Grothendieck, S.L.~Kleiman, M.~Raynaud and J.~Tate.
\newblock {\em Crystals and the De Rham Cohomology of schemes}. In {\em Dix expos{\'e}s sur la cohomologie des sch{\'e}mas}.
\newblock Springer, 1974.



\bibitem{HabeggerSome-unlikely-i}
P.~Habegger and J.~Pila.
\newblock Some unlikely intersections beyond {A}ndr\'{e}-{O}ort.
\newblock {\em Compos. Math.}, 148(01):1--27, January 2012.

\bibitem{HabeggerO-minimality-an}
P.~Habegger and J.~Pila.
\newblock O-minimality and certain atypical intersections.
\newblock {\em Annales scientifiques de l'{E}cole {N}ormale {S}up\'erieure}, to
  appear.

\bibitem{HwangVolumes-of-comp}
J.~Hwang and W.~To.
\newblock Volumes of complex analytic subvarieties of {H}ermitian symmetric
  spaces.
\newblock {\em American Journal of Mathematics}, 124(6):1221--1246, 2002.

\bibitem{KlinglerThe-Hyperbolic-}
B.~Klingler, E.~Ullmo, and A.~Yafaev.
\newblock The hyperbolic {A}x-{L}indemann-{W}eierstrass conjecture.
\newblock {\em Publ. math. IHES}, 123:333--360, 2016.

\bibitem{LaumonTransformation-}
G.~Laumon.
\newblock Transformation de {F}ourier g\'{e}n\'{e}ralis\'{e}e.
\newblock {\em http://arxiv.org/abs/alg-geom/9603004, Preprint IHES}.

\bibitem{ManinAlgebraic-curve}
Y.~Manin.
\newblock Algebraic curves over fields with differentiation.
\newblock {\em Izv. Akad. Nauk SSSR. Ser. Mat.}, 22:737--756, 1958.

\bibitem{ManinRational-points}
Y.~Manin.
\newblock Rational points on algebraic curves over function fields.
\newblock {\em Izv. Akad. Nauk SSSR. Ser. Mat.}, 27:1395--1440, 1963.

\bibitem{MazurMessing}
B.~Mazur and W.~Messing.
\newblock {\em Universal Vector Extensions and One Dimensional Crystalline Cohomology}, volume~370 of {\em LNM}.
\newblock Springer, 1974.


\bibitem{MoonenLinearity-prope}
B.~Moonen.
\newblock Linearity properties of {S}himura varieties, {I}.
\newblock {\em Journal of Algebraic Geometry}, 7(3):539--467, 1988.

\bibitem{OrrFamilies-of-abe}
M.~Orr.
\newblock Families of abelian varieties with many isogenous fibres.
\newblock {\em J.Reine Angew. Math (Crelle)}, 2015:211--231, 2015.

\bibitem{PetersMixed-Hodge-Str}
C.~Peters and J.~Steenbrink.
\newblock {\em Mixed {H}odge Structures}, volume~52 of {\em A Series of Modern
  Surveys in Mathematics}.
\newblock Springer, 2008.

\bibitem{PeterzilComplex-analyti}
Y.~Peterzil and S.~Starchenko.
\newblock Complex analytic geometry and analytic-geometric categories.
\newblock {\em J.Reine Angew. Math (Crelle)}, 626:39--74, 2009.

\bibitem{PeterzilDefinability-of}
Y.~Peterzil and S.~Starchenko.
\newblock Definability of restricted theta functions and families of abelian varieties.
\newblock {\em Duke Journal of Mathematics}, 162:731--765, 2013.

\bibitem{PilaO-minimality-an}
J.~Pila.
\newblock O-minimality and the {A}ndr\'{e}-{O}ort conjecture for
  $\mathbb{C}^n$.
\newblock {\em Annals Math.}, 173:1779--1840, 2011.

\bibitem{PilaAbelianSurfaces}
J.~Pila and J.~Tsimerman.
\newblock The {A}ndr\'{e}-{O}ort conjecture for the moduli space of {A}belian
  surfaces.
\newblock {\em Compos. Math.}, 149:204--216, February 2013.

\bibitem{PilaAxLindemannAg}
J.~Pila and J.~Tsimerman.
\newblock {A}x-{L}indemann for $\mathcal{A}_g$.
\newblock {\em Annals Math.}, 179:659--681, 2014.

\bibitem{PilaAx-Schanuel-for}
J.~Pila and J.~Tsimerman.
\newblock {A}x-{S}chanuel for the $j$-function.
\newblock {\em Duke Math. J.}, To appear.

\bibitem{PinkThesis}
R.~Pink.
\newblock {\em Arithmetical compactification of mixed {S}himura varieties}.
\newblock PhD thesis, Bonner Mathematische Schriften, 1989.

\bibitem{PinkA-Combination-o}
R.~Pink.
\newblock A combination of the conjectures of {M}ordell-{L}ang and
  {A}ndr\'{e}-{O}ort.
\newblock In {\em Geometric Methods in Algebra and Number Theory}, volume 253
  of {\em Progress in Mathematics}, pages 251--282. Birkh{\"a}user, 2005.

\bibitem{AlgebraicGroupBible}
V.~Platonov and A.~Rapinchuk.
\newblock {\em Algebraic Groups and Number Theory}.
\newblock Academic Press, INC., 1994.

\bibitem{ShigaCriteria-for-co}
H.~Shiga and J.~Wolfart.
\newblock Criteria for complex multiplication and transcendence properties of
  automorphic functions.
\newblock {\em J.Reine Angew. Math (Crelle)}, 463:1--25, 1995.

\bibitem{TsimermanAx-Schanuel-and}
J.~Tsimerman.
\newblock {A}x-{S}chanuel and o-minimality.
\newblock In {\em O-Minimality and Diophantine Geometry}, number 421 in London
  Math. Soc. Lecture Note Series. Camb. Univ. Press, 2015.

\bibitem{TsimermanA-proof-of-the-}
J.~Tsimerman.
\newblock A proof of the {A}ndr\'{e}-{O}ort conjecture for $\mathcal{A}_g$.
\newblock {\em Preprint, available on the author's page}, 2015.

\bibitem{UllmoQuelques-applic}
E.~Ullmo.
\newblock Applications du th\'{e}or\`{e}me de {A}x-{L}indemann hyperbolique.
\newblock {\em Compos. Math.}, 150:175--190, 2014.

\bibitem{UllmoStructures-spec}
E.~Ullmo.
\newblock Structures sp\'{e}ciales et probl\`{e}me de {P}ink-{Z}ilber.
\newblock Ast\'{e}risque. to appear.

\bibitem{UllmoA-characterisat}
E.~Ullmo and A.~Yafaev.
\newblock A characterisation of special subvarieties.
\newblock {\em Mathematika}, 57(2):263--273, 2011.

\bibitem{UllmoNombre-de-class}
E.~Ullmo and A.~Yafaev.
\newblock Nombre de classes des tores de multiplication complexe et bornes
  inf\'{e}rieures pour orbites {G}aloisiennes de points sp\'{e}ciaux.
\newblock {\em Bull. de la SMF}, 143:197--228, 2015.

\bibitem{UllmoThe-Hyperbolic-}
E.~Ullmo and A.~Yafaev.
\newblock The hyperbolic {A}x-{L}indemann in the compact case.
\newblock {\em Duke Math. J.}, to appear.

\bibitem{DriesTame-Topology-a}
L.~van~der Dries.
\newblock {\em Tame topology and o-minimal structures}, volume 248 of {\em
  London Math. Soc. Lecture Note Series}.
\newblock Camb. Univ. Press, 1998.

\bibitem{WustholzAlgebraic-group}
G.~W\"{u}stholz.
\newblock Algebraic groups, {H}odge theory, and transcendence.
\newblock In {\em Proceedings of the {I}nternational {C}ongress of
  {M}athematicians}, volume 1,2, pages 476--483, 1986.

\end{thebibliography}
\end{document}